\documentclass[11pt,a4paper]{article} 
\usepackage{amsfonts}
\usepackage{amsmath,amssymb,array,graphicx}
\usepackage{mathrsfs} 
\usepackage{bbm}
\usepackage[numbers,sort]{natbib}
\usepackage{enumerate}
\usepackage[shortlabels]{enumitem}
\usepackage{graphicx}
\usepackage[dvipsnames]{xcolor}
\usepackage{dsfont}
\usepackage{float} 
\usepackage[hyphens]{url}
\usepackage{epsfig}
\usepackage[utf8]{inputenc}
\usepackage{epstopdf}
\usepackage{url}
\usepackage[normalem]{ulem}
\usepackage{multirow}
\usepackage{subfigure}
\usepackage{leftidx} 
\usepackage{setspace}
\usepackage{bbm} 
\usepackage{tikz}
\usetikzlibrary{trees, positioning}
\usepackage{amsthm}
\usepackage{color} 
\usepackage{multirow,cases}
\usepackage{booktabs}
\usepackage[colorlinks,linkcolor=blue,anchorcolor=blue,citecolor=blue]{hyperref}
\usetikzlibrary{shapes,arrows,positioning,fit,calc}
\usepackage[titletoc,toc,title]{appendix}
\usepackage{verbatim}
\usepackage{tensor}
\usepackage{hyperref}
\usepackage{cleveref}
\usepackage{empheq} 
\usepackage{mathtools} 
\usepackage{authblk}

\theoremstyle{plain}

\newtheorem{definition}{Definition}[section]
\newtheorem{theorem}{Theorem}[section]

\newtheorem{lemma}{Lemma}[section]
\newtheorem{remark}{Remark}[section]
\newtheorem{claim}{Claim}[section]
\newtheorem{assumption}{Assumption}[section]
\newtheorem{problem}{Problem}[section]
\newtheorem{proposition}{Proposition}[section]

\usepackage[a4paper,top=2.7cm,bottom=2.8cm,left=2cm,right=2cm,marginparwidth=1.75cm]{geometry}

\newcommand{\beq}{\begin{equation}}
	\newcommand{\eeq}{\end{equation}}
\newcommand{\bey}{\begin{eqnarray}}
	\newcommand{\eey}{\end{eqnarray}}

\def\p{{\partial}}

\def\s{{\sigma}}

\def\bA{\mathbf{A}}
\def\bB{\mathbf{B}}
\def\bI{\mathbf{I}}
\def\bJ{\mathbf{J}}

\newcommand{\bigcdot}{\mathbin{\vcenter{\hbox{\scalebox{1.3}{$\cdot$}}}}}

\usepackage{scalerel,stackengine}
\newcommand\pig[1]{\scalerel*[5.5pt]{\big#1}{ 
		\ensurestackMath{\addstackgap[1.15pt]{\big#1}}}}

\newcommand\pigr[1]{\mathclose{\pig{#1}}}

\numberwithin{equation}{section}

\allowdisplaybreaks
\title{Optimal Matching Strategies in Two-sided Markets: A Mean Field Approach}

\author[,1]{\normalsize Erhan Bayraktar\footnote{E-mail: erhan@umich.edu}}
\author[,2]{\normalsize Dantong Chu\footnote{E-mail: dtchu@link.cuhk.edu.hk}}
\author[,3]{\normalsize Bohan Li\footnote{E-mail: bhli@suda.edu.cn}}
\author[,4]{\normalsize Ho Man Tai\footnote{E-mail: homan.tai@sydney.edu.au}}

\affil[1]{\small\it Department of Mathematics, University of Michigan, USA}
\affil[2]{\small\it Department of Systems Engineering and Engineering Management, The Chinese University of Hong Kong, Hong Kong, China}
\affil[3]{\small\it Center for Financial Engineering, Soochow University, Suzhou, China}
\affil[4]{\small\it School of Mathematics and Statistics, The University of Sydney, Sydney, Australia}

\begin{document}
	\maketitle
	\begin{abstract}
	This paper develops a mean field game framework for dynamic two-sided matching markets, extending existing matching theory by integrating micro-macro dynamics in two-sided environments. Unlike traditional matching models focusing on static equilibrium or unilateral optimization, our framework simultaneously captures dynamic interactions and strategic behaviors of both market sides, as well as the equilibrium. We model two types of agents who meet each other via Poisson processes and make simultaneous matching decisions to maximize their respective objective functionals, and find the corresponding equilibrium. Our approach formulates the equilibrium as a fully coupled Hamilton-Jacobi-Bellman and Fokker-Planck system with nonlocal structure coupling two distinct populations. The mathematical analysis addresses significant challenges from the dual-layered coupling structure and nonlocal structure. We also provide insights into individual behaviors shaping aggregate patterns in labor markets through numerical experiments.
	\end{abstract}
	
	\date{}
	

	\tableofcontents
	
		\section{Introduction}

		In many economic and social systems, outcomes depend critically on how agents from different groups are paired, whether workers with firms, students with schools, buyers with sellers, or partners in personal relationships. A matching model provides a systematic way to analyze these pairing processes, predict resulting allocations, and evaluate their efficiency or stability.  
		
		\subsection{Literature Review}
		Matching models, particularly the two-sided matching model, originate from the seminal work of Gale--Shapley \cite{gale1962college}, where agents rank members of the opposite group by preference. They proved that a  stable matching, with no pair of agents preferring each other over their assigned partners, always exists in contexts like college admissions and marriage markets. Their deferred acceptance algorithm provides a mechanism to achieve such stable outcomes. Roth--Sotomayor \cite{kamecke1992two} later expanded this framework through game-theoretic modeling and analysis, deepening the theoretical foundation of matching markets.
		
		Research in matching theory has evolved along several distinct trajectories. One prominent strand characterizes equilibrium outcomes in matching markets without explicitly modeling the underlying individual optimization problems \cite{gale1962college,chiappori2020theory}. Another direction examines optimal behavior for agents on one market side while treating the other as exogenous \cite{perthame2018career,li2025repositioning}. Besides, theoretical advances have established foundational insights into the existence and properties of random matching processes \cite{duffie2007existence,duffie2025continuous,duffie2018dynamic,menzel2015large}. In parallel, substantial progress has been demonstrated through extensive applications across diverse domains—including labor economics \cite{mccall1970economics,diamond1982aggregate,perthame2018career}, marriage markets \cite{chiappori2020theory,gale1962college,mortensen1988matching}, ride-hailing systems \cite{alonso2017demand,li2025repositioning}, healthcare \cite{agarwal2015empirical} {\color{black} and matching platform designs \cite{KanoriaSaban2021,immorlica2023designing}.}  Recently, novel methodologies such as optimal transport and mean field game models have been employed to study matching \cite{li2019learning,echenique2024stable,bayraktar2025equilibrium,bayraktar2019rate,paty2022algorithms}. However, these approaches typically focus either on static matching or unilateral dynamic behaviors, seldom simultaneously capturing the dynamic interactions in two-sided markets and the aggregate impact of micro-level decisions on macroscopic equilibrium.

		\subsection{Our Model and Main Contributions}\label{sec. Our Model and Main Contributions}
		
		We advance the existing matching literature by bridging micro-level individual behaviors with macro-level equilibrium outcomes in dynamic, two-sided matching markets, adopting the mean field game framework pioneered by Lasry–Lions \cite{lasry2007mean} and Huang–Malhamé–Caines \cite{huang2006large}.

		{\color{black}We study a dynamic two-sided matching market over a finite horizon $[0,T]$ with two continuum populations, labeled by $\bA$ and $\bB$. Each agent has a fixed quality level and is initially unmatched. Agents meet potential partners from the opposite side according to independent Poisson processes. A match is formed only when both agents are unmatched at the time of the meeting, and they mutually accept each other according to their quality-dependent acceptance criteria. Once matched, agents no longer participate in matching. An agent chooses the acceptance criteria to maximize the payoff consisting of a flow payoff while unmatched, a lump sum equal to the matched partner’s quality if matching occurs before $T$, and a terminal payoff if no match is formed by $T$. Our objective is to characterize a mean field Nash equilibrium in this setting. A detailed formulation is deferred to \Cref{sec Problem Settings}.} Therefore, this problem can be formulated in a way similar to graphon games (or graphon mean field games) in the literature \cite{lacker2023label,8619367,AurellCarmonaLauriere2022,BayraktarWuZhang2023}, with a block graphon on labels, $G(\alpha,\beta):=\mathbf{1}_{\{\alpha \in \bA\}}\mathbf{1}_{\{\beta \in \bB\}}
		+\mathbf{1}_{\{\alpha \in \bB\}}\mathbf{1}_{\{\beta \in \bA\}}$. However, our model departs slightly from standard graphon games because each agent’s decision depends on the acceptance thresholds of the opposite population, so the interaction enters through the joint law of states and controls rather than solely through a state distribution. Further details are given in \Cref{rmk:extended_graphon}.
		
		For instance, in labor markets, type-$\bA$ agents represent job seekers, and type-$\bB$ agents represent hiring managers of firms. See Section~\ref{sec. applications} for more examples and further discussion. To conclude, in our model, agents search for and meet partners, and then evaluate whether to proceed with matching. These two fundamental components, the meeting mechanism and the matching mechanism, will be rigorously formulated in Section \ref{sec. Dynamic Matching Problem with Finite Agents}.

		Our contributions are threefold:
		\begin{enumerate}[(1).]
			\item We advance matching theory by integrating micro-level individual decision-making with macro-level equilibrium outcomes in dynamic, two-sided markets. Traditional models often focus on static stable matchings characterized by aggregate properties, as pioneered in the foundational work \cite{gale1962college} and further developed in labor market analyses by \cite{mortensen1988matching}. While these classical approaches provide powerful equilibrium insights, they typically treat equilibrium outcomes without explicitly modeling the underlying individual optimization processes or temporal dynamics.
			
			Recent contributions \cite{duffie2018dynamic,duffie2025continuous} have examined dynamic matching. Yet, these works, along with others, for example, models of job search with fixed wage distributions \cite{mccall1970economics} or ride-hailing platforms \cite{li2025repositioning}, examine individual agent behavior on one side of the market while treating the opposing side as exogenous. Such unilateral models fail to capture the bidirectional strategic interactions and feedback mechanisms essential in many real-world two-sided markets. {\color{black}Recent works \cite{KanoriaSaban2021,immorlica2023designing} bridges micro-level decisions with macro-level outcomes in two-sided matching markets. However, they focus on stationary equilibria and platform design or policy interventions; by contrast, our framework analyzes dynamic equilibrium, and a detailed comparison is provided later.
			}

			Our framework contributes to the literature by employing mean field game theory to model two interacting, heterogeneous populations whose matching strategies and distributions evolve dynamically. We explicitly represent how agents’ acceptance criteria, modeled as control variables, determine matching outcomes, linking individual acceptance threshold decisions to the evolution of unmatched agent distributions over time. This micro-to-macro linkage allows us to analyze how heterogeneous agent preferences, strategic behavior, and stochastic meeting processes combine to produce emergent equilibrium patterns, thus offering a richer and more realistic understanding of dynamic matching phenomena than previous models.
			
			\item We develop novel mathematical tools for analyzing multi-group mean field games, ensuring the well-posedness of our model. Although our problem can be formulated as a graphon game, it is nonstandard relative to the literature, which prevents us from applying the results or approaches there to our case.  See \Cref{rmk:extended_graphon} for details. The main technical result of our paper is the global well-posedness of the equilibrium of the mean field matching model. This equilibrium can be obtained by the fully coupled Hamilton-Jacobi-Bellman and Fokker-Planck (HJB-FP) system \eqref{HJBx}-\eqref{boundary}, which only involves time derivatives and features a nonlocal structure in the state variable. Beyond addressing the forward–backward coupling of the system, we must further decouple the interactions arising from the interplay between the two agent populations. This dual-layered coupling significantly increases the complexity of the system and necessitates a careful decoupling strategy on both structural and population levels. {\color{black}Moreover, as we consider an unbounded domain for the quality space, compactness arguments in traditional Lebesgue spaces are no longer available. This necessitates the use of techniques from the Wasserstein space to establish the well-posedness, adding another layer of technical complexity beyond the coupled forward–backward structure.}
			
			To prove existence via a fixed-point argument, we derive novel estimates tailored to the system’s nonlocal nature, by treating distinct cases separately. Moreover, we must identify a suitable function space in which to apply Schauder’s fixed point theorem. However, due to the presence of only first-order time derivatives, establishing the required compactness is challenging. 
			To overcome this, we exploit the fact that the solution of the FP equation is a defective probability measure and construct another probability measure from it. This allows us to work within a subset of time-dependent probability measures that are equicontinuous in time, for which compactness can be ensured. A detailed outline of the existence proof is provided in Appendix \ref{sec. Main result and Sketch of Proof}. It is worth noting that, although the system exhibits a forward–backward structure, the existence of solutions does not rely on any monotonicity condition (for example, Lasry--Lions monotonicity or displacement monotonicity), owing to the fact that the controlled processes are purely driven by Poisson processes. Nevertheless, uniqueness still requires certain structural conditions on the coefficient functions, see \Cref{thm unique}.

			\item We offer valuable insights into how individual behaviors shape aggregate matching patterns, particularly in labor markets. The numerical results reveal how job seekers and firms adjust their acceptance thresholds over time. 
			Most job seekers become more willing to reduce their acceptance threshold to accept lower-quality partners as opportunities dwindle. High-quality agents see sharper declines in their thresholds, while low-ability seekers’ values remain stable due to limited prospects. On the contrary, the acceptance thresholds of hiring managers demonstrate a decreasing trend over time in low to mid-level firms, while remaining nearly unchanged for high-level firms. Matching success is highest for mid-quality job seekers, with delays at both ends due to limited appeal or high selectiveness. Jobs offered by low-tier firms are filled quickly, but high-tier ones take longer to fill. Partner quality distributions overlap across adjacent groups, reflecting market imperfections and allowing some agents upward mobility through better-than-expected matches. {\color{black}Moreover, we find that the loss from remaining unmatched relative to matching generally increases with quality, except for very low-quality job seekers and for very high-quality firms.}
		\end{enumerate}
		
		It is worth noting the differences between this current study and the interesting work \cite{duffie2025continuous} (see the discrete-time version in \cite{duffie2007existence} and the directed and enduring matching extensions in \cite{duffie2018dynamic}) which developed an uncontrolled continuous-time random-matching model with the primitives (parameters) being fixed: initial distribution, spontaneous (non-matching) mutation intensity, matching intensity and type-change distribution induced by matching. The matching there is an instantaneous random pairing with fixed exogenous intensities, there is no accept/reject decision and no persistent pair formation, unlike our mutual-acceptance model in which agents decide to accept or reject the matching to maximize individual value. The main result of \cite{duffie2025continuous} is a rigorous construction and analysis of a continuous-time, independent random-matching system, comprising a random type process and a random matching process. In contrast to our strategic mean field games setting, \cite{duffie2025continuous} does not study equilibrium arising from individual optimization.

		{\color{black}
			Another closely related set of papers includes \cite{KanoriaSaban2021,immorlica2023designing}, which study continuous-time two-sided matching markets from the perspective of matching platform design. This class of papers differs from ours in two main respects. First, in terms of setting and objective, they focus on how platforms can improve efficiency in two-sided matching markets. The work \cite{KanoriaSaban2021} considers a market that has two quality levels on one side and a homogeneous population on the other, with additional heterogeneity arising from hidden pair-specific match values. Agents receive Poisson search opportunities and decide whether to request a candidate, whether to pay a screening cost to learn a hidden pair-specific match value, and whether to send a proposal. Only after a proposal is sent does the recipient decide whether to screen and then accept or reject, so matching follows a proposal-response procedure rather than direct mutual acceptance. As a result, agents of the same quality level on the same side use the same acceptance threshold both for screened candidates and for incoming proposals. On this basis, they characterize explicit equilibrium regimes and derive welfare-improving platform interventions. The work \cite{immorlica2023designing} studies the platform-directed search in matching platforms and shows that simple meeting designs can achieve a constant-factor approximation to optimal welfare despite incentive constraints and NP-hardness. In contrast, we consider a market of two continua of agents with heterogeneous quality levels on both sides, where matching occurs via direct mutual acceptance with no proposal--response procedure. We focus on the optimal behaviors of agents of different quality levels and on how these optimal behaviors shape the evolution of the population distributions. Second, mathematically, these papers characterize stationary equilibria in which the distribution is invariant and optimal strategies are independent of time, whereas we study the corresponding dynamic equilibrium in which the population distribution evolves and optimal strategies vary across time; our equilibrium is therefore characterized by a system of fully coupled partial integro-differential equations.}

		\paragraph{Organization of the paper:} The rest of the paper is organized as follows. In Section~\ref{sec Problem Settings}, we formulate the matching model, along with the associated control problems and the mean field Nash game. Section~\ref{sec. main results} presents the main theoretical results, including the global well-posedness of the fully coupled HJB–FP system and the verification theorem. In Section~\ref{sec. Numerical Results of the Matching Model}, we perform numerical experiments and offer economic interpretations of the results. Finally, we conclude the article in \Cref{sec con}. The proofs of the assertions in Section~\ref{sec. main results} are provided in Appendices.

		\section{Problem Settings} \label{sec Problem Settings}
		This section presents a rigorous mathematical formulation of the matching model. {\color{black} Some notations are necessarily technical, so we begin with a brief overview to facilitate reading. We study a matching market over a finite horizon $[0,T]$ with two continuum populations, denoted by $\bA$ and $\bB$. Each agent has a fixed quality level and a time-varying matching status, indicating whether the agent is currently unmatched or already matched. Agents with the same type, quality level, and matching status are homogeneous and therefore treated as indistinguishable. All agents are initially unmatched. The model proceeds in two stages: first, meetings occur; then, matching decisions are made. Let $\bI, \bJ\in \{\bA,\bB\}$ and $\bI \neq \bJ$.
			
			\paragraph{Meeting Mechanism.} Each unmatched type-$\bI$ agent meets potential partners from the opposite side according to a Poisson jump, a standard assumption in searching and matching models \cite{diamond1982aggregate,lucas2014knowledge}. At each meeting, the encountered type-$\bJ$ agent’s quality level and matching status are drawn randomly from the current distribution of type-$\bJ$ agents in the market. If that agent has already matched, then the meeting fails, and the type-$\bI$ agent waits for the next meeting.

			\paragraph{Matching Mechanism.} When a meeting occurs at time $t\in[0,T]$ and both agents are unmatched, they simultaneously decide whether to accept each other. More precisely, a type-$\bI$ agent with quality $x$ chooses an acceptance threshold $u_{\bI}(x,t)$. The type-$\bI$ agent accepts the encountered type-$\bJ$ agent with quality $y$ if and only if $
			y \ge u_{\bI}(x,t)$, and similarly for the type-$\bJ$ agent. Thus, a higher value of $u_{\bI}(x,t)$ indicates more selective behavior. A match is formed only under mutual acceptance: a type-$\bI$ agent of quality $x$ and a type-$\bJ$ agent of quality $y$ match at time $t$ if they meet and if $y \ge u_{\bI}(x,t)$ and $x \ge u_{\bJ}(y,t)$. We assume a first-come, first-served (FCFS) matching rule: as soon as an unmatched agent meets the first mutually acceptable partner, they match immediately. Matches are one-to-one and binding, so matched agents are no longer available for another match, although they may still be encountered by others.
			
			\smallskip
			An unmatched type-$\bI$ agent of quality $x$ chooses an acceptance threshold to maximize the expected discounted payoff which consists of three components: a running payoff $r_{\bI} (x,\bigcdot)$ while searching; a lump sum payoff equal to the matched partner's quality upon matching before the terminal time $T$; and a terminal payoff $h_{\bI}(x)$, interpreted as an outside option, if no match is formed.

			The article aims to characterize the equilibrium of the resulting mean field Nash game. Each agent takes as given the distribution and the acceptance threshold of agents on the opposite side, then chooses a best response (acceptance threshold). In equilibrium, individual best responses and the distribution of agents must be consistent with the aggregate outcomes they jointly generate. There is no direct strategic interaction among agents of the same type;  strategic interaction arises only between agents of opposite types.

			In Section \ref{sec. Probabilistic Setting and Meeting Mechanism}, we construct the probability spaces, filtrations, and Poisson random measures that capture the initial quality draw of a representative agent and the random meeting process, given any distribution flow of the opposite type. Section \ref{sec. Dynamic Matching Problem with Finite Agents} specifies the matching mechanism and formulates the agent’s optimization problem. For a representative agent, meetings are governed by a Poisson process; at each meeting time, the encountered agent’s quality and matching status are sampled from the given distribution flow of the opposite type. Taking as given the acceptance threshold policy of agents on the opposite side, the representative agent chooses an optimal threshold policy. In Section \ref{sec. Mean Field Nash Equilibrium}, we characterize how the population distribution of each type is induced by these individual optimal decisions. The mean field Nash equilibrium is defined as a situation where, starting from a pair of threshold policies and a pair of distribution flows on the two sides, the optimal best responses and the resulting population distributions lead back to the same pair of policies and the same pair of flows. At equilibrium, the consistency condition ensures that the distribution flow used for sampling coincides with the actual distribution generated by the optimal policies, and the assumed threshold policies on the opposite side coincide with the optimal ones.  Section \ref{sec. discussions} compares our setting with two related formulations: the graphon game approach and the finite‑player matching game. Finally, Section \ref{sec. applications} presents two applications illustrating the model's tractability and economic insights.

			\subsection{Probabilistic Setting}\label{sec. Probabilistic Setting and Meeting Mechanism}
			
			We now present a rigorous mathematical formulation of the model. Consider a matching market over the finite time horizon $[0,T]$, consisting of two distinct continuum populations of agents, denoted by type-$\bA$ and type-$\bB$. Each agent is identified by three attributes: type, quality level, and matching status. We begin by constructing the underlying probability space.

			For each type $\bI \in \{\bA,\bB\}$, let $X_{\bI} :=[0,\infty):= \mathbb{R}_{\geq 0}$ denote the set of values of type-$\bI$ agents' quality levels, and let $\mu_{\bI,0}$ be a fixed time independent probability measure on $X_{\bI}$. There exists a canonical probability space $(\Omega^{\bI,0},\mathcal F^{\bI,0},\mathbb P^{\mu_{\bI,0}})$ supporting an $X_{\bI}$-valued random variable $\mathbf{ql}_{\bI}$ with distribution $\mu_{\bI,0}$. This random variable represents the quality level of a representative type-$\bI$ agent; 
			it is drawn at the initial time and remains constant over time. The matching status of a type-$\bI$ agent, denoted by $\mathbf{st}_{\bI,\bigcdot}$, takes values in $\{0,1\}$, where $0$ indicates “unmatched” and $1$ indicates “matched”. All agents are initially unmatched. Accordingly, the state space of a type-$\bI$ agent is $X_{\bI} \times \{0,1\}$, equipped with its product Borel $\sigma$-algebra $\mathcal{B}(X_{\bI} \times \{0,1\})$.

			Let $\bJ \in \{\bA,\bB\} \setminus \{\bI\}$ denote the opposite type. Define $\mathscr{S}_\bI$ the set of flows of probability distribution 
			\begingroup 
			\fontsize{10.2pt}{10pt}\begin{align*}
				\hspace{-10pt}\mathscr{S}_{\bI}
				:=
				\left\{
				\psi_{\bI,\bigcdot}:\mathcal{B}(X_{\bI}\times\{0,1\})\times[0,T]\to[0,1]:
				\begin{aligned}
					&\textup{for any $B\in \mathcal{B}(X_{\bI}\times\{0,1\}),\ \psi_{\bI,\bigcdot}(B):[0,T]\to[0,1]$ is} \\
					&\textup{Borel measurable, $\int_{X_{\bI}} x\,\psi_{\bI,t}(dx,0)<\infty$ for $t\in[0,T]$} 
				\end{aligned}
				\right\}.\hspace{-10pt}
			\end{align*}\normalsize\endgroup 
			Given any $\psi_{\bJ,\bigcdot} \in \mathscr{S}_\bJ$, representing the state distribution of type-$\bJ$ agents, there exists a canonical probability space $(\Omega^{\bJ,1}, \mathcal{F}^{\bJ,1}, \mathbb{P}^{\psi_{\bJ}})$ accommodating a Poisson random measure $N^{\psi_{\bJ}}$ on $X_{\bJ} \times \{0,1\} \times [0,T]$ with compensator $\lambda_{\bJ} \, \psi_{\bJ,t}(dy,dk) \, dt$, where $\lambda_{\bJ}>0$. This probability space describes the environment faced by a representative type-$\bI$ agent. The random measure $N^{\psi_{\bJ}}$ captures the random arrival and random sampling of potential partners from type $\bJ$; a similar setting can be found in \cite{nutz2018mean}.    \begin{remark}
				The probability space $(\Omega^{\bI,0},\mathcal F^{\bI,0},\mathbb P^{\mu_{\bI,0}})$ can be constructed by choosing $\Omega^{\bI,0} = X_\bI$, $\mathcal F^{\bI,0} = \mathcal{B}(X_\bI)$, the Borel $\sigma$-field of $ X_\bI$, and $\mathbb P^{\mu_{\bI,0}} = \mu_{\bI,0}$. The random variable $\mathbf{ql}_\bI$ is the coordinate mapping $\mathbf{ql}_\bI(\omega^{\bI,0}) = \omega^{\bI,0}$. For $(\Omega^{\bJ,1}, \mathcal{F}^{\bJ,1}, \mathbb{P}^{\psi_{\bJ}})$, we first take $\Omega^{\bJ,1}$ to be the space of integer-valued Radon measures on $X_\bJ \times \{0,1\} \times [0,T]$, and $\mathcal{F}^{\bJ,1}$ the Borel $\sigma$-field of $\Omega^{\bJ,1}$; note that this measurable space $(\Omega^{\bJ,1}, \mathcal{F}^{\bJ,1})$ is independent of $\psi_{\bJ,\bigcdot}$. Then we let $N^{\psi_{\bJ}}$ be the coordinate mapping on $\Omega^{\bJ,1}$, and define a family of finite dimensional distributions by using compensator $\lambda_{\bJ} \, \psi_{\bJ,t}(dy,dk) \, dt$. By applying the Kolmogorov extension theorem \cite[Section 3]{resnick1987extreme} and \cite[Section 9.2]{daley2008introduction}, there exists a probability measure $\mathbb{P}^{\psi_{\bJ}}$ such that $N^{\psi_{\bJ}}$ is a Poisson random measure with compensator $\lambda_{\bJ} \, \psi_{\bJ,t}(dy,dk) \, dt$ under $\mathbb{P}^{\psi_{\bJ}}$.
			\end{remark}

			\paragraph{Probability space:} We define the product probability space $
			(\Omega,\mathcal F^{\psi_{\bA},\psi_{\bB}} ,\mathbb P^{\psi_{\bA},\psi_{\bB}} )$ as the completion of the product space of $(\Omega^{\bA,0},\mathcal F^{\bA,0},\mathbb P^{\mu_{\bA,0}})$, $(\Omega^{\bB,0},\mathcal F^{\bB,0},\mathbb P^{\mu_{\bB,0}})$, $(\Omega^{\bA,1}, \mathcal{F}^{\bA,1}, \mathbb{P}^{\psi_{\bA}})$ and $(\Omega^{\bB,1}, \mathcal{F}^{\bB,1}, \mathbb{P}^{\psi_{\bB}})$. We equip the product space with a filtration $\mathbb{F}^{\psi_{\bA},\psi_{\bB}} := (\mathcal{F}^{\psi_{\bA},\psi_{\bB}}_t)_{0\leq t\leq T}$ with $\mathcal{F}^{\psi_{\bA},\psi_{\bB}}_t = \mathcal{F}_{\bA,t}^{\psi_{\bB}} \vee \mathcal{F}_{\bB,t}^{\psi_{\bA}}$, and
			\begin{equation*}
				\mathcal{F}_{\bI,t}^{\psi_{\bJ}} := \sigma\pig( \mathbf{ql}_{\bI} \pig)\vee\sigma\pig( N^{\psi_{\bJ}}(B \times [0,s]) : B \in \mathcal{B}(X_{\bJ} \times \{0,1\}), 0 \le s \le t \pig) \vee \mathcal{N}
			\end{equation*}
			where $\mathcal{N}$ denotes the collection of all $\mathbb P^{\psi_{\bA},\psi_{\bB}}$-null sets in $\mathcal{F}^{\psi_{\bA},\psi_{\bB}}$. This space serves as the underlying probability space for the matching model. Writing $\omega=(\omega^{\bA,0},\omega^{\bB,0},\omega^{\bA,1},\omega^{\bB,1})$ where $\omega^{\bI,0} \in \Omega^{\bI,0}$ and $\omega^{\bI,1} \in \Omega^{\bI,1}$, we regard $\mathbf{ql}_{\bI}(\omega)=\mathbf{ql}_{\bI}(\omega^{\bI,0})$ and $N^{\psi_{\bI}}(\omega)=N^{\psi_{\bI}}(\omega^{\bI,1})$.

			\paragraph{Sequential representation of meetings:}  The Poisson random measure $N^{\psi_{\bJ}}$ can be characterized by the sequence $\{(\tau_{\bJ,k}, Z_{\bJ,\tau_{\bJ,k}})\}$, where:
			\begin{itemize}
				\item $\tau_{\bJ,k}$ takes values in $[0,T]$ and represents the arrival time of the $k$-th meeting event before time $T$. The sequence $0 < \tau_{\bJ,1} < \tau_{\bJ,2} < \cdots$ corresponds to the arrival times of the counting process. We let $\mathcal{T}^{\bJ}:=\{\tau_{\mathbf{J},k}\}$ denote the set of these arrival times.
				\item $Z_{\bJ,\tau_{\bJ,k}} = (\mathbf{QL}_{\bJ,\tau_{\bJ,k}}, \mathbf{ST}_{\bJ,\tau_{\bJ,k}})$ takes values in $X_{\bJ} \times \{0,1\}$ and represents the sample associated with the $k$-th meeting. Here $\mathbf{QL}_{\bJ,\tau_{\bJ,k}}$ is the quality level, and $\mathbf{ST}_{\bJ,\tau_{\bJ,k}}$ is the matching status of the $k$-th encountered type-$\bJ$ agent. The conditional law of $Z_{\bJ,\tau_{\bJ,k}} = (\mathbf{QL}_{\bJ,\tau_{\bJ,k}}, \mathbf{ST}_{\bJ,\tau_{\bJ,k}})$ given the realized arrival time $\tau_{\bJ,k}$ is $\psi_{\bJ,\tau_{\bJ,k}}$. 
			\end{itemize}

			This sequence, together with the agent’s own quality level $\mathbf{ql}_\bI$ drawn from the initial distribution, constitutes the sources of randomness for the representative type-$\bI$ agent's decision-making.\smallskip

			\begin{remark}[\bf Notation: representative agent vs.\ encountered partner]
				For a fixed representative type-$\bI$ agent, we distinguish between the agent's own state and the attributes of encountered type-$\bJ$ agents:
				\begin{itemize}
					\item Lowercase notation $(\mathbf{ql}_{\bI}, \mathbf{st}_{\bI,t})$ denotes the quality and matching status of the representative type-$\bI$ agent.
					
					\item Uppercase notation $(\mathbf{QL}_{\bJ,\tau_{\bJ,k}}, \mathbf{ST}_{\bJ,\tau_{\bJ,k}})$, appearing as marks in the Poisson random measure $N^{\psi_{\bJ}}$, denotes the quality and matching status of the type-$\bJ$ agent encountered at the $k$-th meeting time $\tau_{\bJ,k}$.
				\end{itemize}
				By construction of the product probability space, $\mathbf{ql}_{\bI}$ is independent of the encountered marks $\{Z_{\bJ,\tau_{\bJ,k}}\}$, and therefore also independent of the encountered qualities $\{\mathbf{QL}_{\bJ,\tau_{\bJ,k}}\}$.
			\end{remark} 
			
			For notational simplicity, we write
			$$\mathbb{P} = \mathbb{P}^{\psi_{\bA},\psi_{\bB}}$$ and let $\mathbb{E}$ be the associated expectation, throughout the article.

			\subsection{Dynamic Optimal Threshold Control Problems}\label{sec. Dynamic Matching Problem with Finite Agents}
			In this section, we describe how a match is formed from the perspective of a representative agent. For any $\bI, \bJ \in \{\bA,\bB\}$ with $\bI \neq \bJ$ and any $t \in [0,T]$, a representative type-$\bI$ agent is identified by the triple $(\mathbf{ql}_{\bI}, \mathbf{st}_{\bI,t^-}, t)$. Since matched agents cease to participate in the matching process, we can restrict attention to unmatched agents and write $(x,t)$ as shorthand for the state $(\mathbf{ql}_{\bI}=x, \, \mathbf{st}_{\bI,t^-}=0, \, t)$. The notation $(x,t)$ is used throughout the definition of the value function and the derivation of the HJB-FP equations. In our setting, agents with the same state $(x,t)$ face the same optimization problem and thus adopt the same optimal strategy; consequently, we treat them as indistinguishable.

			We now rigorously define the meeting mechanism for a representative type-$\bI$ agent, for any given distribution flow $\psi_{\bJ,\bigcdot} \in  \mathscr{S}_\bJ$. Recall that $\{(\tau_{\bJ,k}, \mathbf{QL}_{\bJ,\tau_{\bJ,k}}, \mathbf{ST}_{\bJ,\tau_{\bJ,k}})\}$ represents the sequence of all meetings. We assume that meetings between agents of the two types follow a meeting structure similar to the M/G/$\infty$ queue \cite[Example 2.4(B) of Section 2.4]{ross1995stochastic}:
			\begin{assumption}[\bf Meeting Mechanism]\label{ass.meeting rule0} 
				\sloppy We assume that a representative type-$\bI$ agent meets potential partners from the entire population of type-$\bJ$ agents (both matched and unmatched) at times $\tau_{\bJ,k}$. If  $\,\mathbf{ST}_{\bJ,\tau_{\bJ,k}} = 1$, then the meeting fails and no match occurs.
			\end{assumption}

			We now formulate the decision problem. For a representative unmatched type-$\bI$ agent $(x,t)$, the control variable is the \textit{acceptance threshold policy}. 
			
			\begin{definition}[\bf Acceptance Threshold Policy and Admissible Set]\label{control variable}
				Let $\bI,\bJ \in \{\bA,\bB\}$ with $\bI \neq \bJ$. An admissible strategy of unmatched type-$\bI$ agents is an acceptance threshold function $u_\bI \in C([0,T];L^{\infty}_{loc}(X_\bI;X_\bJ))$. The admissible set $\mathcal{U}_\bI$ of strategies is defined as the subset of $C([0,T];L^{\infty}_{loc}(X_\bI;X_\bJ))$ containing all functions $u_{\bI} :X_\bI \times[0,T] \to X_\bJ $ that are strictly increasing and Borel measurable in $x \in X_\bI$  for each $t \in [0,T]$. 
			\end{definition}\smallskip
			
			For a type-$\bI$ agent $(x,t)$, the acceptance threshold $u_\bI(x,t)$ is the minimum quality level of a type-$\bJ$ agent that the agent $(x,t)$ is willing to accept. More precisely, upon encountering an unmatched type-$\bJ$ agent $(y,t)$ at time $t$, the type-$\bI$ agent $(x,t)$ accepts the match if and only if $y \geq u_\bI(x,t)$. Symmetrically, the type-$\bI$ agent is accepted by the encountered type-$\bJ$ agent if and only if $x \geq v_\bJ(y,t)$, where $v_\bJ \in \mathcal{U}_\bJ$ denotes the acceptance threshold of type-$\bJ$ agents. The strict monotonicity of the threshold is economically motivated: higher-quality agents typically require higher-quality partners, reflecting stricter acceptance standards. Moreover, we assume that agents are matched according to the following \emph{FCFS} matching mechanism \cite{adan2012exact}:

			\begin{assumption}[\bf Matching Mechanism]\label{ass.matching rule}
				Let $\bI, \bJ \in \{\bA,\bB\}$ with $\bI \neq \bJ$. For a representative unmatched type-$\bI$ agent $(x, t)$ with $\mathbf{ql}_{\bI} = x$ and $\mathbf{st}_{\bI,t^-} = 0$, the search and match process on the time interval $[t,T]$ is driven by the triples $\{(\mathbf{QL}_{\bJ,\tau_{\bJ,k}}, \mathbf{ST}_{\bJ,\tau_{\bJ,k}},\tau_{\bJ,k})\}$ of the Poisson random measure $N^{\psi_{\bJ}}$ and proceeds as follows:
				
				\begin{enumerate}
					\item[\bf (1).] \emph{Meeting a type-$\bJ$ agent:} 
					At each arrival time $\tau_{\bJ,k} > t$, the type-$\bI$ agent $(x,t)$ meets a potential partner, characterized by $(\mathbf{QL}_{\bJ,\tau_{\bJ,k}}, \mathbf{ST}_{\bJ,\tau_{\bJ,k}})$; 
					
					\item[\bf (2).] \emph{Mutual assessment:} 
					A match is formed at time $\tau_{\bJ,k}$ if and only if both of the following conditions hold:
					\begin{itemize}
						\item[(i)] \emph{Availability:} the encountered type-$\bJ$ agent is unmatched, that is, $\mathbf{ST}_{\bJ,\tau_{\bJ,k}} = 0$; 
						\item[(ii)] \emph{Mutual acceptance:} both agents satisfy each other's acceptance thresholds at time $\tau_{\bJ,k}$: $\mathbf{QL}_{\bJ,\tau_{\bJ,k}} \geq u_\bI(x, \tau_{\bJ,k})$ and $x \geq v_\bJ(\mathbf{QL}_{\bJ,\tau_{\bJ,k}}, \tau_{\bJ,k})$;
					\end{itemize}
					
					\item[\bf (3).] \emph{Matching outcome:}  If (i) and (ii) of (2) are satisfied, then the status of the representative type-$\bI$ agent 
					jumps immediately: $\mathbf{st}_{\bI, \tau_{\bJ,k}} = 1$.
					The agent ceases to participate in matching, although the agent may still be encountered by others.
					Otherwise, no match is formed, either because the encountered agent is already matched or because mutual acceptance fails. In this case, the representative agent remains unmatched, that is, $\mathbf{st}_{\bI, \tau_{\bJ,k}} = 0$, and waits for the next meeting time $\tau_{\bJ,k+1}$.
				\end{enumerate}
			\end{assumption}

			Under Assumption \ref{ass.matching rule}, the matching status process $\mathbf{st}_{\bI,\bigcdot}$ of a representative type-$\bI$ agent $(x,t)$ evolves  as
			\begin{equation}\label{sde.matching_status}
				d \mathbf{st}_{\bI,s} = \int_{X_{\bJ} \times \{0,1\}} \mathbf{1}_{\{k=0\}} \mathbf{1}_{\{\mathbf{st}_{\bI,s^-} = 0\}} \mathbf{1}_{\{x \geq v_{\bJ}(y,s)\}} \mathbf{1}_{\{y \geq u_{\bI}(x,s)\}} N^{\psi_{\bJ}} (dy,dk,ds),
			\end{equation}
			for any $s\in[t,T]$ with $\mathbf{st}_{\bI,t} = 0$. Equation \eqref{sde.matching_status} states that the matching status of agent $(x,t)$ jumps from $0$ to $1$ (from unmatched to matched) at the arrival time $\tau_{\mathbf{J},k} > t$ at which condition (2) in Assumption \ref{ass.matching rule} is satisfied by the encountered type-$\mathbf{J}$ agent. The process $\mathbf{st}_{\bI,\bigcdot}$ is $\mathbb{F}^{\psi_{\bA},\psi_{\bB}}$-adapted.

			Let the acceptance thresholds of both sides be $u_\bI$ and $v_\bJ$. By Assumption \ref{ass.matching rule}, 
			the matching time of the type-$\bI$ agent $(x,t)$ is defined pathwise as follows: for any given sample path $\omega\in\Omega$,  
			
			\begin{align}\label{def match time of A}
				\tau_{\hspace{0.5pt}\bI}(x,t,\omega) :=\inf\left\{s\in\{\tau_{\bJ,k}(\omega)\}\cap[t,T]: 
				\begin{aligned}
					& \mathbf{ql}_{\bI}(\omega) = x, \,  \mathbf{st}_{\bI,s^-}(\omega) =\mathbf{ST}_{\bJ,s}(\omega)=0, \\
					&x \geq v_\bJ (\mathbf{QL}_{\bJ,s}(\omega),s) \text{ and } \mathbf{QL}_{\bJ,s}(\omega) \geq u_\bI (x,s)
				\end{aligned}\right\}\!,\hspace{-10pt}
			\end{align}
			with the convention that $\inf \emptyset = \infty$.
			By this construction, $\tau_{\hspace{0.5pt}\bI}(x,t)$ is a $\mathbb{F}^{\psi_{\bA},\psi_{\bB}}$-stopping time.

			For any $\bI, \bJ\in \{\bA,\bB\}$ with $\bI \neq \bJ$, we define 
			\begingroup \begin{equation}\label{def.PIxt}
				\mathbb{P}_{\bI,x,t}(\bigcdot) := \mathbb{P}(\,\bigcdot\,|\,\mathbf{ql}_{\bI}=x,\mathbf{st}_{\bI,t^-}=0)
			\end{equation}\endgroup
			as the conditional probability given that the representative type-$\bI$ agent has quality level $x$ and status $0$ at time $t$. Let $\mathbb{E}_{\bI,x,t}$ denote the corresponding conditional expectation. We introduce the running utility $r_\bI: X_\bI \times [0,T] \to \mathbb{R}_{\geq 0}$ and the terminal utility $h_\bI:X_\bI \to \mathbb{R}_{\geq 0}$, both assumed to be strictly increasing in the state variable $x$ for each time $t$. We also fix a discount rate $\rho > 0$. Given the joint distribution $\psi_{\bJ,\bigcdot} \in \mathscr{S}_\bJ$ of the quality level and status of type-$\bJ$ agents, and the acceptance threshold $v_\bJ$, the type-$\bI$ agent $(x,t)$ chooses the acceptance threshold $u_\bI$ to maximize the objective functional:
			\begingroup 
			\begin{align}\label{def.finite.objective function}
				J_\bI (u_\bI;x,t,v_\bJ,\psi_{\bJ,\bigcdot})
				:=\mathbb{E}_{\bI,x,t}\Bigg[& \underbrace{\int_t^{\tau_{\bI}(x,t) \wedge T} e^{-\rho (s-t)} r_\bI(x,s) \, ds}_{\text{part-time/idle Payoff}} \nonumber \\
				&+ \underbrace{ e^{-\rho [\tau_{\bI}(x,t)-t]} \mathbf{QL}_{\bJ, \tau_{\bI}(x,t)} \cdot \mathbf{1}_{\{\tau_{\bI} (x,t)\leq T\}}}_{\text{matching Payoff}} + \underbrace{e^{-\rho (T-t)} h_\bI(x) \cdot \mathbf{1}_{\{\tau_{\bI}(x,t) > T\}}}_{\text{outside Option Payoff}} \Bigg].\hspace{-10pt}
			\end{align}\endgroup 
		}
		The objective functional \eqref{def.finite.objective function} consists of the following components:
		\begin{enumerate}[(1).]
			\item \noindent\textit{Before matching:} While the type-$\bI$ agent $(x,t)$ remains unmatched, the agent receives a running payoff $r_\bI(x,s) \geq 0$ over $s \in [t, \tau_{\hspace{0.5pt}\bI} (x,t) \wedge T]$, representing productivity from part-time engagement or idle resource use. This phase ends at the matching time $\tau_{\hspace{0.5pt}\bI}(x,t)$ or at the terminal horizon $T$.
			
			\item \noindent\textit{Matching at time $\tau_{\hspace{0.5pt}\bI}<T$:} If the type-$\bI$ agent $(x,t)$ is matched with a random type-$\bJ$ agent $(\mathbf{QL}_{\bJ,\tau_{\bJ,k}},\tau_{\bJ,k})$ at time $\tau_{\bI}(x,t) = \tau_{\bJ,k} < T$ , then the type-$\bI$ agent receives a lump sum payoff equal to $\mathbf{QL}_{\bJ,\tau_{\bJ,k}}$.
			
			\item \noindent\textit{No match by the terminal time $T$:} If the type-$\bI$ agent $(x,t)$ is not matched by time $T$, then the agent receives a terminal utility $h_\bI(x) \geq 0$, interpreted as the value of an outside option or payoff from alternative opportunities.
		\end{enumerate} 
		
		The matching decision problem for a type-$\bI$ agent $(x,t)$ is formulated as the optimal control problem:
		\begin{problem}[\bf Optimal Acceptance Threshold for Matching Problem]\label{pblm:type A, threshold}
			Let $\bI,\bJ \in \{\bA,\bB\}$ with $\bI \neq \bJ$. 
			For any fixed acceptance threshold {\color{black}$v_\bJ \in \mathcal{U}_{\bJ}$ of type-$\bJ$ agents and the joint probability distribution flow $\psi_{\bJ,\bigcdot} \in \mathscr{S}_\bJ$ on $X_\bJ \times \{0,1\}$}, the representative type-$\bI$ agent $(x,t)$ chooses an acceptance threshold $u_\bI \in \mathcal{U}_{\bI}$ to maximize the objective functional $J_\bI$ defined in \eqref{def.finite.objective function}. The associated value function is			
			\begingroup 
			\begin{align}\label{def.V_I.threshold}
				\mathcal{V}_\bI (x,t,v_\bJ,\psi_{\bJ,\bigcdot}) & :=\sup\limits_{u_\bI  \,\in \,\mathcal{U}_{\bI} }J_\bI (u_\bI;x,t,v_\bJ,\psi_{\bJ,\bigcdot}). 
			\end{align}\endgroup
		\end{problem} 
		
		Alternatively, this problem admits an equivalent reformulation as an optimal stopping problem, in the spirit of \cite{mccall1970economics}, which more directly highlights the interpretation of the acceptance threshold. Rather than choosing a threshold $u_{\mathbf{I}}$, a type-$\mathbf{I}$ agent $(x,t)$ instead selects an optimal stopping time at which to accept a type-$\bJ$ agent drawn from the flow $\psi_{\mathbf{J},\bigcdot}$. Formally, we define the admissible set of stopping times 
		{\color{black} 
			\fontsize{10pt}{10pt}\begin{align}\label{def.arrival time, at t}
				\hspace{-5pt}\mathcal{T}^{\bJ}_t:=\Big\{\tau \in M\pig(\Omega;[t,\infty]\pig): 
				\begin{aligned}
					\textup{for any $\omega\in \Omega$, either $\exists \, \tau^\bJ \in \mathcal{T}^\bJ$  s.t. }
					\textup{$\tau(\omega) \mathbf{1}_{\{\tau(\omega) \leq T\}}= \tau^\bJ(\omega) $ or $\tau(\omega) = \infty$}
				\end{aligned} 
				\Big\}
			\end{align}\normalsize
			where $M\pig(\Omega;[t,\infty]\pig)$ consists of all $\mathcal{F}^{\psi_{\bA},\psi_{\bB}}_{\bigcdot}$-stopping times. Here, \(\tau(\omega)=\infty\) describes the situation in which the agent rejects all encountered partners and never matches; otherwise, the agent accepts a partner at the meeting time specified by \(\tau(\omega)\leq T\).} The type-$\bI$ agent then chooses an exit time $\s_{\bI} \in \mathcal{T}^{\bJ}_t$ to maximize the objective functional
		{\color{black}
			\begingroup 
			\begin{align}\label{def obj fun of app.asymA}
				\nonumber \widetilde{J}_\bI(\s_\bI;x,t,v_\bJ,\psi_{\bJ,\bigcdot})
				:=\mathbb{E}_{\bI,x,t}\bigg[&\int_t^{\s_\bI  \land T} e^{-\rho (s-t)} r_\bI(x,s) ds\\
				\nonumber &+  e^{-\rho (\s_\bI-t) } \mathbf{QL}_{\bJ,\s_\bI}\mathbf{1}_{\left\{
					\mathbf{ST}_{\bJ,\s_\bI}=0, \, x \geq v_\bJ(\mathbf{QL}_{\bJ,\s_\bI},\s_\bI )\right\}}\mathbf{1}_{\{\s_\bI \leq T\}} \\
				&+ e^{-\rho (T-t)}\mathbf{1}_{\{\s_\bI  > T\}}h_\bI (x)\bigg],
			\end{align}\endgroup 
			where the dynamic of the status the agent is given by
			\begin{equation}\label{sde.matching_status,stopping}
				d \mathbf{st}_{\bI,s} = \int_{X_{\bJ} \times \{0,1\}} \mathbf{1}_{\{k=0\}} \mathbf{1}_{\{\mathbf{st}_{\bI,s^-} = 0\}} \mathbf{1}_{\{x \geq v_{\bJ}(y,s)\}} \mathbf{1}_{\{s = \s_\bI \}} N^{\psi_{\bJ}} (dy,dk,ds), \quad s\in[t,T],
			\end{equation}
			with $\mathbf{st}_{\bI,t} = 0$.
		}
		The optimal exit time is given by $
		\tau^*_\bI  = \arg\max_{\s_\bI \in \mathcal{T}^{\bJ}_t}\widetilde{J}_\bI (\s_\bI;x,t,v_\bJ, \psi_{\bJ,\bigcdot})$ if it exists. If $\tau^*_\bI \leq T$, then a match occurs at time $\tau^*_\bI$ with the partner whose attributes are realized at that time. This optimal stopping problem can be formally stated as follows.

		\begin{problem}[\bf Optimal Stopping Time for Matching Problem]\label{pblm:type A, stopping time}
			Under the same assumptions as in Problem~\ref{pblm:type A, threshold}, a type-$\bI$ agent $(x,t)$ chooses the time $\sigma_\bI \in \mathcal{T}^{\bJ}_t$ to maximize the objective functional $\widetilde{J}_\bI$ given in \eqref{def obj fun of app.asymA}. The associated value function is defined as			\begingroup 
			\begin{align}
				\label{def.V_I.optimal stopping}
				\widetilde{\mathcal{V}}_\bI(x,t,v_\bJ,\psi_{\bJ,\bigcdot}) & :=\sup\limits_{\sigma_\bI  \,\in \, \mathcal{T}^{\bJ}_t}\widetilde{J}_\bI ( \sigma_\bI;x,t,v_\bJ,\psi_{\bJ,\bigcdot} )
			\end{align}\endgroup
			where $\mathcal{T}^{\bJ}_t$ is defined in \eqref{def.arrival time, at t}.
		\end{problem} 
		\smallskip
		It is immediate that $\widetilde{\mathcal{V}}_\bI(x,t,v_\bJ,\psi_{\bJ,\bigcdot}) \geq \mathcal{V}_\bI (x,t,v_\bJ,\psi_{\bJ,\bigcdot})$. The following lemma establishes the reverse inequality, thereby implying the equivalence of the value functions of Problems~\ref{pblm:type A, threshold} and \ref{pblm:type A, stopping time}.

		\begin{lemma}
			{\color{black} Assume that $r_{\bA}(\cdot,t)$, $r_{\bB}(\cdot,t)$, $h_{\bA}(\cdot)$ and $h_{\bB}(\cdot)$ are strictly increasing for each $t\in[0,T]$.} Let $\bI,\bJ \in \{\bA,\bB\}$ with $\bI \neq \bJ$ and $t\in[0,T]$.  Suppose that the acceptance threshold $v_\bJ \in \mathcal{U}_{\bJ}$ of type-$\bJ$ agents and the flow of joint probability distribution $\psi_{\bJ,\bigcdot} \in \mathscr{S}_\bJ$ on $X_\bJ \times \{0,1\}$ are fixed. If Problem \ref{pblm:type A, stopping time} admits the optimal control $\tau^*_\bI \in \mathcal{T}^{\bJ}_t$, then there exists a deterministic function $ u^*_\bI \in \mathcal{U}_{\bI}$, such that
			{\color{black}
				\begin{align}\label{368}
					\hspace{-5pt} \tau_{\hspace{0.5pt}\bI}^*(x,t,\omega) =\inf\left\{s\in\!\{\tau_{\bJ,k}(\omega)\} \cap [ t,T]: 
					\begin{aligned}
						&\mathbf{ql}_{\bI}(\omega) = x, \, \mathbf{ST}_{\bJ,s}(\omega)=
						\mathbf{st}_{\bI,s^-}(\omega) =0, \\
						& x \geq v_\bJ (\mathbf{QL}_{\bJ,s}(\omega),s) \text{ and } \mathbf{QL}_{\bJ,s}(\omega) \geq u^*_\bI (x,s)
					\end{aligned}\!\right\},\hspace{-5pt} 
				\end{align}
				with the convention that $\inf \emptyset = \infty$.} Hence, $\widetilde{\mathcal{V}}_\bI(x,t,v_\bJ,\psi_{\bJ,\bigcdot}) = \mathcal{V}_\bI (x,t,v_\bJ,\psi_{\bJ,\bigcdot})$.\label{lem eq. of P1 P2}
		\end{lemma}

		\begin{remark}
			By the relationship between $u^*_\bI$ and $\tau_{\hspace{0.5pt}\bI}^*$ established in \Cref{lem eq. of P1 P2}, the equivalence of Problems~\ref{pblm:type A, threshold} and \ref{pblm:type A, stopping time} can be interpreted as follows: the optimal acceptance threshold in Problems~\ref{pblm:type A, threshold} ensures that whenever a type-$\bI$ agent $(x,t)$ meets an unmatched type-$\bJ$ with $\mathbf{QL}_{\bJ,\tau_{\bJ,k}} \geq u^*_\bI (x,\tau_{\bJ,k} )$ and $x \geq v_\bJ (\mathbf{QL}_{\bJ,\tau_{\bJ,k}},\tau_{\bJ,k})$, then the agents immediately accept each other and match.
		\end{remark}
		\begin{proof}[Proof of \Cref{lem eq. of P1 P2}]
			In the spirit of the dynamic programming principle, we prove \Cref{lem eq. of P1 P2} by analyzing different scenarios of the first meeting time $\tau_{\bJ,1,t}(\omega):=\inf\big\{s\in \{\tau_{\bJ,1}(\omega) ,\tau_{\bJ,2}(\omega),\cdots\} 
			\cap [t,\infty)\big\}$ and comparing the agent's outcomes across the resulting scenarios.
			
			The value function in \eqref{def.V_I.optimal stopping} can be rewritten as			\begingroup 
			\begin{align}\label{value function in thm 1}
				\widetilde{\mathcal{V}}_\bI (x,t,v_\bJ,\psi_{\bJ,\bigcdot})
				=& \sup\limits_{ \sigma_\bI  \in \mathcal{T}^{\bJ}_t}\mathbb{E}_{\bI,x,t} \left[\widetilde{\mathscr{J}}_\bI \left( \sigma_\bI;x,t,v_\bJ,\psi_{\bJ,\bigcdot};\mathbf{QL}_{\bJ,\tau_{\bJ,1,t}},\mathbf{ST}_{\bJ,\tau_{\bJ,1,t}},\tau_{\bJ,1,t}\right)\right],
			\end{align}\endgroup
			\sloppy where $\widetilde{\mathscr{J}}_\bI ( \sigma_\bI;x,t,v_\bJ,\psi_{\bJ,\bigcdot};\mathbf{QL}_{\bJ,\tau_{\bJ,1,t}},\mathbf{ST}_{\bJ,\tau_{\bJ,1,t}},\tau_{\bJ,1,t})$ is the expected payoff of the representative type-$\bI$ agent $(x,t)$ given the first 
			meeting time $\tau_{\bJ,1,t}$ and the randomly sampled type-$\bJ$ agent $(\mathbf{QL}_{\bJ,\tau_{\bJ,1,t}},\mathbf{ST}_{\bJ,\tau_{\bJ,1,t}},\tau_{\bJ,1,t})$ from the conditional distribution $\psi_{\bJ,\tau_{\bJ,1,t}}$ given $\tau_{\bJ,1,t}$. More precisely, 
			\begingroup 
			\begin{align}
				&\widetilde{\mathscr{J}}_\bI \left(\sigma_\bI;x,t,v_\bJ,\psi_{\bJ,\bigcdot};\mathbf{QL}_{\bJ,\tau_{\bJ,1,t}},\mathbf{ST}_{\bJ,\tau_{\bJ,1,t}},\tau_{\bJ,1,t}\right)\nonumber\\
				&:=\mathbb{E}_{\bI,x,t}\bigg(\int_t^{ \sigma_\bI \land T} e^{-\rho (s-t)} r_\bI(x,s) ds 
				+ e^{-\rho ( \sigma_\bI -t)}\mathbf{QL}_{\bJ, \sigma_\bI  }\mathbf{1}_{\{ \mathbf{ST}_{\bJ, \sigma_\bI  }=0, \, x \ge v_\bJ (\mathbf{QL}_{\bJ, \sigma_\bI }, \sigma_\bI  ) \}}\mathbf{1}_{\left\{ \sigma_\bI   \leq T\right\}} 
				\nonumber \\
				&\hspace{80pt}+ e^{-\rho (T-t)}\mathbf{1}_{\{ \sigma_\bI  > T\}}h_\bI (x)\,\bigg|\,\mathbf{QL}_{\bJ,\tau_{\bJ,1,t}},\mathbf{ST}_{\bJ,\tau_{\bJ,1,t}},\tau_{\bJ,1,t}\bigg),\label{592}
			\end{align}\endgroup
			for any $ \sigma_\bI  \in \mathcal{T}^{\bJ}_t$. In this case, \eqref{value function in thm 1} shows that
			\begingroup 
			\begin{equation}\label{value function in thm 1.2}
				\widetilde{\mathcal{V}}_\bI (x,t,v_\bJ,\psi_{\bJ,\bigcdot})\leq \mathbb{E}_{\bI,x,t} \left[\widetilde{\mathscr{V}}_\bI (x,t,v_\bJ,\psi_{\bJ,\bigcdot};\mathbf{QL}_{\bJ,\tau_{\bJ,1,t}},\mathbf{ST}_{\bJ,\tau_{\bJ,1,t}},\tau_{\bJ,1,t})\right],
			\end{equation}\endgroup
			where			\begingroup 
				\begin{equation}\label{value function in thm 1.3}
					\widetilde{\mathscr{V}}_\bI \left(x,t,v_\bJ,\psi_{\bJ,\bigcdot};\mathbf{QL}_{\bJ,\tau_{\bJ,1,t}},\mathbf{ST}_{\bJ,\tau_{\bJ,1,t}},\tau_{\bJ,1,t}\right) := \sup\limits_{ \sigma_\bI  \in \mathcal{T}^{\bJ}_t}
					\widetilde{\mathscr{J}}_\bI \left( \sigma_\bI ;x,t,v_\bJ,\psi_{\bJ,\bigcdot};\mathbf{QL}_{\bJ,\tau_{\bJ,1,t}},\mathbf{ST}_{\bJ,\tau_{\bJ,1,t}},\tau_{\bJ,1,t}\right).
				\end{equation}\endgroup
			According to the values of $\mathbf{QL}_{\bJ,\tau_{\bJ,1,t}}$ and $\mathbf{ST}_{\bJ,\tau_{\bJ,1,t}}$, there are two possible scenarios for time $\tau_{\bJ,1,t}$:

			\begin{enumerate}[(1).]
				
				\item The first scenario corresponds to a failed meeting or matching attempt, after which the representative type-$\bI$ agent $(x,t)$ remains unmatched and waits for subsequent meetings. This scenario happens when the type-$\bI$ agent chooses to reject the match with the sampled type-$\bJ$ agent. For convenience, we denote the value function of the representative type-$\bI$ agent $(x,t)$ in this situation at time $\tau_{\bJ,1,t}$ for $\tau_{\bJ,1,t} \leq T$ by  $\widetilde{\mathscr{V}}_\bI \left(x,\tau_{\bJ,1,t},v_\bJ,\psi_{\bJ,\bigcdot};0,0,\tau_{\bJ,1,t}\right)$, since the agent will certainly reject a partner with a quality level of $0$ due to the presence of the running or terminal utility. Consequently, the value function for this scenario satisfies
				\begingroup 
				\begin{align*}
					&\widetilde{\mathscr{V}}_\bI \left(x,t,v_\bJ,\psi_{\bJ,\bigcdot};\mathbf{QL}_{\bJ,\tau_{\bJ,1,t}},\mathbf{ST}_{\bJ,\tau_{\bJ,1,t}},\tau_{\bJ,1,t}\right) \\
					&\geq \int_t^{\tau_{\bJ,1,t} \wedge T}e^{-\rho (s-t)}r_\bI(x,s)ds + e^{-\rho (\tau_{\bJ,1,t}-t)}\widetilde{\mathscr{V}}_\bI \left(x,\tau_{\bJ,1,t},v_\bJ,\psi_{\bJ,\bigcdot};0,0,\tau_{\bJ,1,t}\right) \mathbf{1}_{\{\tau_{\bJ,1,t} \leq T\}} \\
					&\quad+ e^{-\rho (T-t)}\mathbf{1}_{\{\tau_{\bJ,1,t} > T\}}h_\bI(x).
				\end{align*}\endgroup
				
				\item \sloppy The second scenario is characterized by a successful match: $\mathbf{ST}_{\bJ,\tau_{\bJ,1,t}} = 0$, $x\geq v_{\bJ}(\mathbf{QL}_{\bJ,\tau_{\bJ,1,t}},\tau_{\bJ,1,t})$, and the type-$\bI$ agent chooses to accept the match. In this situation, we have
				\begingroup \begin{align*}
						&\widetilde{\mathscr{V}}_\bI \left(x,t,v_\bJ,\psi_{\bJ,\bigcdot};\mathbf{QL}_{\bJ,\tau_{\bJ,1,t}},\mathbf{ST}_{\bJ,\tau_{\bJ,1,t}},\tau_{\bJ,1,t}\right) \\
						&\geq \int_t^{\tau_{\bJ,1,t} \wedge T}  e^{-\rho (s-t)}r_\bI(x,s)ds 
						+e^{-\rho (\tau_{\bJ,1,t}-t)}\mathbf{QL}_{\bJ,\tau_{\bJ,1,t}} \mathbf{1}_{\big\{\mathbf{ST}_{\bJ,\tau_{\bJ,1,t}} = 0 \big\}} \mathbf{1}_{\big\{x \ge v_\bJ (\mathbf{QL}_{\bJ,\tau_{\bJ,1,t}},\tau_{\bJ,1,t})\big\}} \mathbf{1}_{\left\{\tau_{\bJ,1,t} \leq T\right\}} \\
						&\quad+ e^{-\rho (T-t)}\mathbf{1}_{\{\tau_{\bJ,1,t} > T\}}h_\bI(x).
				\end{align*}\endgroup
			\end{enumerate}
			
			Combining the above cases and using the principle of optimality, we obtain	\begingroup 
			\begin{align*} 
				\nonumber&\widetilde{\mathscr{V}}_\bI \left(x,t,v_\bJ,\psi_{\bJ,\bigcdot};\mathbf{QL}_{\bJ,\tau_{\bJ,1,t}},\mathbf{ST}_{\bJ,\tau_{\bJ,1,t}},\tau_{\bJ,1,t}\right) \\
				&=e^{-\rho (\tau_{\bJ,1,t}-t)}\max\left\{\mathbf{QL}_{\bJ,\tau_{\bJ,1,t}}\mathbf{1}_{\big\{\mathbf{ST}_{\bJ,\tau_{\bJ,1,t}} = 0\big\}}\mathbf{1}_{\big\{x \ge v_\bJ (\mathbf{QL}_{\bJ,\tau_{\bJ,1,t}},\tau_{\bJ,1,t})\big\}}, \widetilde{\mathscr{V}}_\bI \left(x,\tau_{\bJ,1,t},v_\bJ,\psi_{\bJ,\bigcdot};0,0,\tau_{\bJ,1,t}\right)\right\} \\
				&\quad\quad + \int_t^{\tau_{\bJ,1,t} \wedge T}e^{-\rho (s-t)}r_\bI(x,s)ds + e^{-\rho (T-t)}\mathbf{1}_{\{\tau_{\bJ,1,t} > T\}}h_\bI(x).
			\end{align*}\endgroup
			\sloppy The first term inside the maximum term corresponds to the value function if a match occurs at $\tau_{\bJ,1,t}$, while the second corresponds to the value function for no match. Therefore, $\mathbf{QL}_{\bJ,\tau_{\bJ,1,t}}\mathbf{1}_{\big\{\mathbf{ST}_{\bJ,\tau_{\bJ,1,t}} = 0\big\}}\mathbf{1}_{\big\{x \ge v_\bJ (\mathbf{QL}_{\bJ,\tau_{\bJ,1,t}},\tau_{\bJ,1,t})\big\}} \geq \widetilde{\mathscr{V}}_\bI \left(x,\tau_{\bJ,1,t},v_\bJ,\psi_{\bJ,\bigcdot};0,0,\tau_{\bJ,1,t}\right)$ is a necessary condition for a match to occur at $\tau_{\bJ,1,t}$. {\color{black}This implies that the conditions $\mathbf{ST}_{\bJ,s} = 0$, $\mathbf{QL}_{\bJ,s} \ge \widetilde{\mathscr{V}}_\bI \left(x,s,v_\bJ,\psi_{\bJ,\bigcdot};0,0,s\right)$ and $x \ge v_\bJ (\mathbf{QL}_{\bJ,s},s)$ are necessary conditions for a match to occur at any stopping time $s \in \mathcal{T}^\bJ$.
				
				\begin{claim}\label{claim.1}
					The function $\widetilde{\mathscr{V}}_\bI (x,s,v_\bJ,\psi_{\bJ,\bigcdot};0,0,s)$ is strictly increasing in $x$ for any $s\in [0,T]$.\end{claim} 
				To prove the claim, fix an arbitrary initial time $s\in[0,T]$ and two quality levels $0\le x_1<x_2$. Note that $r_\bI(x,t)$ and $h_\bI(x)$ are strictly increasing in $x$, and $\mathbf{1}_{\{x \ge v_\bJ (\mathbf{QL}_{\bJ, \sigma_\bI }, \sigma_\bI  ) \}}$ is non-decreasing in $x$, $\mathbb{P}$ a.s., for any $ \sigma_\bI  \in \mathcal{T}^{\bJ}_s$. Consequently, the definition of $\widetilde{\mathscr{J}}_\bI$ in \eqref{592} and that of $\widetilde{\mathscr{V}}_\bI (x,s,v_\bJ,\psi_{\bJ,\bigcdot};0,0,s)$ in \eqref{value function in thm 1.3} imply
				\begin{align*}
					\widetilde{\mathscr{J}}_\bI \left( \sigma_\bI ;x_1,s,v_\bJ,\psi_{\bJ,\bigcdot};0,0,s\right)
					<\widetilde{\mathscr{J}}_\bI \left( \sigma_\bI ;x_2,s,v_\bJ,\psi_{\bJ,\bigcdot};0,0,s\right)\le \widetilde{\mathscr{V}}_\bI (x_2,s,v_\bJ,\psi_{\bJ,\bigcdot};0,0,s)
				\end{align*}
				for any $ \sigma_\bI  \in \mathcal{T}^{\bJ}_s$, as $\sigma_\bI>s$, $\mathbb{P}$-a.s. Taking the supremum over the space $\mathcal{T}^{\bJ}_s$ on the left-hand side and using \eqref{value function in thm 1.3} again, we obtain $
				\widetilde{\mathscr{V}}_\bI (x_1,s,v_\bJ,\psi_{\bJ,\bigcdot};0,0,s)
				< \widetilde{\mathscr{V}}_\bI (x_2,s,v_\bJ,\psi_{\bJ,\bigcdot};0,0,s)$ which proves the strict monotonicity.
				
				Therefore, it follows that $\widetilde{\mathscr{V}}_\bI (x,s,v_\bJ,\psi_{\bJ,\bigcdot};0,0,s) \in \mathcal{U}_{\bI}$. We thus choose $u^*_\bI (x,s):=\widetilde{\mathscr{V}}_\bI (x,s,v_\bJ,\psi_{\bJ,\bigcdot};0,0,s)$, and the optimal stopping time of \eqref{value function in thm 1.3} is $\tau^*_\bI (x,t)$ taking the form in \eqref{368}. }

			By \eqref{value function in thm 1.2}, we have	\begingroup 
			\begin{align*} 
				\nonumber \widetilde{\mathcal{V}}_\bI(x,t,v_\bJ,\psi_{\bJ,\bigcdot}) \leq\,& \mathbb{E}_{\bI,x,t}\pig[\widetilde{\mathscr{J}}_\bI \left(\tau^*_\bI (x,t);x,t,v_\bJ,\psi_{\bJ,\bigcdot};\mathbf{QL}_{\bJ,\tau_{\bJ,1,t}},\mathbf{ST}_{\bJ,\tau_{\bJ,1,t}},\tau_{\bJ,1,t}\right)\pig] \\
				\nonumber=\,& \mathbb{E}_{\bI,x,t}\!\bigg[\!\mathbb{E}_{\bI,x,t}\!\bigg(\!\int_t^{\tau^*_\bI (x,t)  \land T} e^{-\rho (s-t)} r_\bI(x,s) ds \\
				\nonumber
				&\qquad\qquad\quad\!\!+ \!e^{-\rho  
					[\tau^*_\bI (x,t)-t]}\mathbf{QL}_{\bJ,\tau_{\hspace{0.5pt}\bI}^*(x,t)}\!\mathbf{1}_{\{ \mathbf{ST}_{\bJ,\tau_{\bI}^*(x,t)}=0, \, x \ge v_\bJ (\mathbf{QL}_{\bJ,\tau^*_\bI (x,t)},\tau^*_\bI (x,t)) \}}\!\mathbf{1}_{\{\tau^*_\bI (x,t)  \leq T\}}\hspace{-10pt} \\
				\nonumber
				&\qquad\qquad\quad\!\!+ e^{-\rho (T-t)}\mathbf{1}_{\{\tau^*_\bI (x,t)  > T\}}h_\bI (x)\,\bigg|\,\mathbf{QL}_{\bJ,\tau_{\bJ,1,t}},\mathbf{ST}_{\bJ,\tau_{\bJ,1,t}},\tau_{\bJ,1,t}\bigg)\bigg] \\ \nonumber
				=\,&\widetilde{J}_\bI (\tau_{\hspace{0.5pt}\bI}^*(x,t);x,t,v_\bJ,\psi_{\bJ,\bigcdot} ) \leq\, \mathcal{V}_\bI (x,t,v_\bJ,\psi_{\bJ,\bigcdot}).
				\hspace{-15pt}
			\end{align*}\endgroup
			The second equality is due to the tower property and the definition in \eqref{def obj fun of app.asymA}. The last equality follows from definition \eqref{def.finite.objective function} and the fact that $\tau^*(x,t)$ takes form in \eqref{368}. The last inequality follows from the optimality in Problem \ref{pblm:type A, threshold}. Recall that $\mathcal{V}_\bI (x,t,v_\bJ,\psi_{\bJ,\bigcdot}) \leq \widetilde{\mathcal{V}}_\bI(x,t,v_\bJ,\psi_{\bJ,\bigcdot})$. One obtains that $\tau^*_\bI (x,t)$ is indeed the optimal stopping time for Problem \ref{pblm:type A, stopping time}. 
		\end{proof}
		
		{\color{black}	\subsection{Population Dynamic Evolution and the Mean Field Nash Equilibrium}\label{sec. Mean Field Nash Equilibrium}
			
			In this section, we consider a mean field game based on the formulations in \Cref{sec. Dynamic Matching Problem with Finite Agents}. We begin by introducing the relevant population distributions. Let $\bI,\bJ \in \{\bA,\bB\}$ with $\bI \neq \bJ$. Suppose that the acceptance threshold $v_\bJ \in \mathcal{U}_{\bJ}$ of type-$\bJ$ agents and the joint probability distribution flow $\psi_{\bJ,\bigcdot} \in \mathscr{S}_\bJ$ on $X_\bJ \times \{0,1\}$ are fixed.
			
			\begin{enumerate}[(1).]
				\item \textit{Distribution of agents:} We define the joint distribution of type-$\bI$ agents' quality levels and matching statuses at time $t$ as a measure $\pi_{\bI,t}$ on $X_{\bI} \times \{0,1\}$ by:
				\begin{equation}\label{def.pi}
					\pi_{\bI,t}(E) := \mathbb{P}\pig( (\mathbf{ql}_{\bI},\mathbf{st}_{\bI,t}) \in E \pig),
				\end{equation}
				for every $E \in \mathcal{B} (X_{\bI} \times \{0,1\})$ and $t\in[0,T]$. Here, $\mathbf{ql}_{\bI}$ is the quality level drawn from the distribution $\mu_{\bI,0}$. Based on the matching mechanism in \Cref{ass.matching rule}, the status $\mathbf{st}_{\bI,\bigcdot}$ of the type-$\bI$ agent $(x, 0)$ evolves according to \eqref{sde.matching_status} under the acceptance threshold $u_\bI \in \mathcal{U}_{\bI}$. Specifically, we denote the quality level distribution of unmatched type-$\bI$ agents at time $t$ by 
				\begin{equation}\label{def.muI}
					\mu_{\bI,t}(A) := \pi_{\bI,t}(A \times \{0\}).
				\end{equation}
				Note that $\mu_{\bI,t}$ is a defective probability measure \cite{asmussen2010ruin} on $X_{\bI}$ since its total mass decreases over time as agents become matched. Initially, all agents are unmatched, so $\mu_{\bI,0}(X_{\bI}) = 1$.
				
				\item \textit{Unmatched probability:} 
				The total mass of unmatched type-$\bI$ agents at time $t$ is defined as:
				\begingroup 
				\begin{equation}\label{def. unmatched F_I}
					F_{\bI}(t) := \mu_{\bI,t}(X_{\bI}) = \mathbb{P}\left( \mathbf{st}_{\bI,t} = 0 \right).
				\end{equation}
				\endgroup 
				Clearly, $F_{\bI}(0) = 1$, and $F_{\bI}(t)$ is non-increasing in $t$.
				
				\item \textit{Consistency of quality distribution (first marginal):} 
				Since the quality level $\mathbf{ql}_{\bI}$ of the representative agent is determined at $t=0$ and remains constant throughout the horizon, the first marginal distribution of $\pi_{\bI,t}$ must coincide with the initial distribution $\mu_{\bI,0}$ for all $t\in[0,T]$. That is, for any $t \in [0,T]$ and $A \in \mathcal{B}(X_{\bI})$:
				\begin{equation*}
					\pi_{\bI,t}(A \times \{0,1\}) = \mathbb{P}(\mathbf{ql}_{\bI} \in A) = \mu_{\bI,0}(A).
				\end{equation*}
				The conservation of mass is expressed as
				$\mu_{\bI,t}(A) + \pi_{\bI,t}(A \times \{1\}) = \mu_{\bI,0}(A)$. 
			\end{enumerate}
			
			We now characterize the distribution $\mu_{\bI,t}$ of quality levels among unmatched type-$\bI$ agents at time $t$ using the matching time defined in \eqref{def match time of A}. In particular, Lemma~\ref{prop.fxt3} establishes that $\mu_{\bI,t}$ is absolutely continuous with respect to the initial distribution $\mu_{\bI,0}$.
			\begin{lemma}\label{prop.fxt3} 
				Let $\bI, \bJ\in \{\bA,\bB\}$ with $\bI \neq \bJ$ and $t\in[0,T]$. Fix acceptance thresholds $u_\bI \in \mathcal{U}_{\bI}$, $v_\bJ \in \mathcal{U}_{\bJ}$, and a joint probability distribution flow $\psi_{\bJ,\bigcdot} \in \mathscr{S}_\bJ$ on $X_\bJ \times \{0,1\}$. Then the defective probability distribution $\mu_{\bI ,t}(\bigcdot):= \pi_{\bI,t}(\bigcdot \times \{0\})$ defined in \eqref{def.muI} is absolutely continuous with respect to $\mu_{\bI ,0}$ such that $ 
				\mathbb{P}_{\bI,\bigcdot,0}(\tau_{\hspace{0.5pt}\bI} (\bigcdot,0)> t) = \frac{d\mu_{\bI ,t}}{d\mu_{\bI ,0}}(\bigcdot)$ where $\tau_{\hspace{0.5pt}\bI} (x,0)$ is the matching time of the type-$\bI$ agent $(x,0)$ as defined in \eqref{def match time of A}, corresponding to the acceptance
				thresholds $u_\bI$ and $v_\bJ$ fixed above.
			\end{lemma}
			\begin{proof} \sloppy Consider a bounded measurable test function $\Phi:X_\bI\times\{0,1\} \to \mathbb{R}$ such that $\Phi(x,k) := \phi(x)\mathbf{1}_{\{k=0\}}$ where $\phi$ is a bounded measurable function on $X_\bI$. For any $t \in [0,T]$, we have $
				\mathbb{E}\left[\Phi(\mathbf{ql}_{\bI},\mathbf{st}_{\bI,t})\right] = \int_{X_\bI} \phi(x) d\mu_{\bI,t}(x).$
				On the other hand, by conditioning on $\mathbf{ql}_{\bI}$, we have
				$
				\int_{X_\bI} \!\phi(x) d\mu_{\bI,t}\!(x)
				=  \mathbb{E}\left[\mathbb{E}\left(\Phi(\mathbf{ql}_{\bI},\mathbf{st}_{\bI,t})|\mathbf{ql}_{\bI}\right)\right] 
				= \mathbb{E}\left[\phi(\mathbf{ql}_{\bI})\mathbb{P}(\mathbf{st}_{\bI,t} = 0|\mathbf{ql}_{\bI})\right].$ Note that $\mathbb{P}(\mathbf{st}_{\bI,t} = 0\,|\,\mathbf{ql}_{\bI}) = \mathbb{P}(\tau_{\hspace{0.5pt}\bI}(\mathbf{ql}_{\bI},0) > t\,|\,\mathbf{ql}_\bI)$,
				where $\tau_{\hspace{0.5pt}\bI}(\mathbf{ql}_\bI,0)$ is the matching time of the type-$\bI$ agent $(\mathbf{ql}_\bI,0)$. Since $\mathbf{ql}_\bI$ follows the distribution $\mu_{\bI ,0}$, the expectation can be expressed as:\vspace{-10pt}
				$$
				\int_{X_\bI} \hspace{-3pt}\phi(x) d\mu_{\bI,t}(x)=\mathbb{E}\left[\phi(\mathbf{ql}_{\bI})\mathbb{P}(\tau_{\hspace{0.5pt}\bI}(\mathbf{ql}_{\bI},0) > t\,|\,\mathbf{ql}_\bI)\right] 
				= \int_{X_\bI} \hspace{-3pt} \phi(x) \mathbb{P}_{\bI,x,0}(\tau_{\hspace{0.5pt}\bI}(x,0) > t) d\mu_{\bI ,0}(x).
				$$
				Since this holds for every bounded measurable test function $\phi$, the lemma is established. 
			\end{proof}
			
			\begin{lemma}\label{prop.Ft}
				Under the same setting and assumptions as in \Cref{prop.fxt3}, the unmatched proportion function $F_{\bI}(t)$ defined in \eqref{def. unmatched F_I} equals $\mathbb{P}(\tau_{\hspace{0.5pt}\bI} (\mathbf{ql}_{\bI},0)> t)$ for any $t\in [0,T]$.
			\end{lemma}
			\begin{proof} \sloppy
				By \eqref{def. unmatched F_I}, Lemma \ref{prop.fxt3} and the definition in \eqref{def match time of A}, we have $F_\bI (t)= \mu_{\bI,t}(X_\bI) = \int_{X_\bI}\mathbb{P}_{\bI,x,0}(\tau_{\hspace{0.5pt}\bI} (x,0)> t)d\mu_{\bI ,0}(x) = \mathbb{E}\big[\mathbb{P}(\tau_{\hspace{0.5pt}\bI} (\mathbf{ql}_{\bI},0)> t\,|\,\mathbf{ql}_{\bI})\big] = \mathbb{P}(\tau_{\hspace{0.5pt}\bI} (\mathbf{ql}_{\bI},0)> t)$. 
			\end{proof}
			
			With the characterization and definition of $\mu_{\bI,t}$, the joint distribution $\pi_{\bI,t}$ is uniquely determined via
			\[
			\pi_{\bI,t}(A \times \{0\}) = \mu_{\bI,t}(A), \qquad
			\pi_{\bI,t}(A \times \{1\}) = \mu_{\bI,0}(A) - \mu_{\bI,t}(A),
			\]
			for any Borel set $A \subseteq X_{\bI}$ and $t\in[0,T]$. Given $(v_\bB, \psi_{\bB,\bigcdot})$, one can solve for the optimal acceptance threshold $u_\bA$ of type‑$\bA$ agents, together with the induced population distribution $\pi_{\bA,\bigcdot}$. Similarly, given $(v_\bA, \psi_{\bA,\bigcdot})$, one obtains the optimal acceptance threshold $u_\bB$ of type‑$\bB$ agents and the corresponding distribution $\pi_{\bB,\bigcdot}$. If the optimal controls are well-defined, then the above procedure defines a mapping
			\[
			(v_\bA, v_\bB, \psi_{\bA,\bigcdot}, \psi_{\bB,\bigcdot}) \;\longmapsto\; (u_\bA, u_\bB, \pi_{\bA,\bigcdot}, \pi_{\bB,\bigcdot}).
			\]
			The mean field Nash equilibrium is then defined as a fixed point of the mapping defined above. We denote such a fixed point by $(u_\bA^*, u_\bB^*, \pi_{\bA,\bigcdot}^*, \pi_{\bB,\bigcdot}^*)$. \Cref{fig:structure:Nash_equilibrium} illustrates the interaction between type-$\bA$ and type-$\bB$ agents in the dynamic mean field Nash game.

			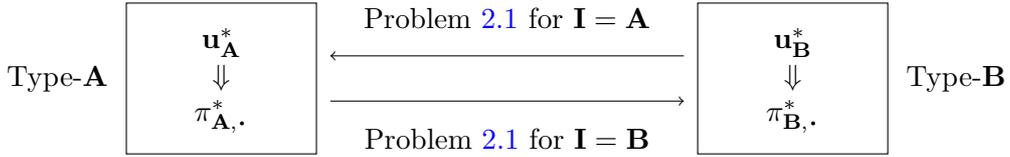
\begin{figure}[h]
				\centering
				\begin{tikzpicture}[
					scale=1,
					transform shape,
					box/.style={
						draw,
						rectangle,
						inner sep=2pt,
						outer sep=0pt,
						align=center
					}
					]
					
					\node[box] (leftbox) at (0,0)
					{$\mathbf{u^*_\bA \;\Longrightarrow\; \pi^*_{\bA,\bigcdot}}$};
					
					\node (psiB) at (0,1.1)
					{$(v_{\bB},\psi_{\bB,\bigcdot})$};
					
					\node[box] (rightbox) at (5,1.1)
					{$\mathbf{u^*_\bB \;\Longrightarrow\; \pi^*_{\bB,\bigcdot}}$};
					
					\node (psiA) at (5,0)
					{$(v_{\bA},\psi_{\bA,\bigcdot})$};
					
					\draw[thick,->]
					([yshift=-2pt]psiB.south) -- ([yshift=2pt]leftbox.north);
					
					\draw[thick,->]
					([yshift=2pt]psiA.north) -- ([yshift=-2pt]rightbox.south);
					
					\draw[thick,<->]
					([xshift=6pt]psiB.east |- rightbox.center) --
					([xshift=-2pt]rightbox.west |- rightbox.center);
					
					\draw[thick,<->]
					([xshift=2pt]leftbox.east |- leftbox.center) --
					([xshift=-5.5pt]psiA.west |- leftbox.center);
					
					\node[below=0.3cm of leftbox, align=center]
					{Problem \ref{pblm:type A, threshold} for $\bI=\bA$};
					
					\node[below=1.3cm of rightbox, align=center]
					{Problem \ref{pblm:type A, threshold} for $\bI=\bB$};
					
				\end{tikzpicture}
				\caption{Dynamic mean field Nash game}
				\label{fig:structure:Nash_equilibrium}
		\end{figure}}
		
		The main goal of this article is to solve the following problem:
		\begin{problem}[\bf Mean Field Nash Equilibrium]\label{def Nash}  
			We aim to find a quadruple $(u^*_\bA ,u^*_\bB ,\pi_{\bA,\bigcdot}^*,\pi_{\bB,\bigcdot}^*)$ such that for any $\bI,\bJ \in \{\bA,\bB\}$ with $\bI \neq \bJ$: 
			\begin{itemize}
				\item The control $u^*_\bI$ solves Problem \ref{pblm:type A, threshold} under the pair $(u^*_\bJ,\pi_{\bJ,\bigcdot}^*)$, that is, $J_\bI (u^*_\bI;x,0,u^*_\bJ,\pi_{\bJ,\bigcdot}^*  ) \geq J_\bI (u_\bI;x,0,u^*_\bJ,\pi_{\bJ,\bigcdot}^*)$ for any $u_\bI \in \mathcal{U}_\bI$ and $x\in X_\bI$;
				\item The distribution $\pi_{\bI,\bigcdot}^*$ is the joint law of the quality levels and statuses of type-$\bI$ agents, where the status evolves according to the dynamics under the optimal acceptance threshold $u^*_\bI$ and the pair $(u^*_\bJ,\pi_{\bJ,\bigcdot}^*)$.
			\end{itemize}
			Such a quadruple is called a mean field Nash equilibrium of the mean field matching game associated with Problem \ref{pblm:type A, threshold} for $\bI \in \{\bA,\bB\}$. 
		\end{problem} \smallskip
		\sloppy {\color{black} For any given distribution flow $(\psi_{\bA,\bigcdot},\psi_{\bB,\bigcdot}) \in  \mathscr{S}_\bA \times\mathscr{S}_\bB$, the sample space $\Omega$ in $(\Omega,\mathcal F^{\psi_{\bA},\psi_{\bB}} ,\mathbb F^{\psi_{\bA},\psi_{\bB}}, \mathbb P^{\psi_{\bA},\psi_{\bB}} )$ does not vary with $(\psi_{\bA,\bigcdot},\psi_{\bB,\bigcdot})$. For a Nash equilibrium $(u^*_\bA ,u^*_\bB ,\pi_{\bA,\bigcdot}^*,\pi_{\bB,\bigcdot}^*)$, the corresponding augmented filtered probability space is $
			(\Omega,\mathcal F^{\pi_{\bA}^*,\pi_{\bB}^*},\mathbb{F}^{\pi_{\bA}^*,\pi_{\bB}^*},\mathbb P^{\pi_{\bA}^*,\pi_{\bB}^*} )$. We continue to denote the probability measure and expectation at equilibrium by $\mathbb P=\mathbb P^{\pi_{\bA}^*,\pi_{\bB}^*}$ and $\mathbb E$, respectively, if the equilibrium exists. }

		{\color{black}\subsection{Discussions}\label{sec. discussions}
			In this section, we compare the framework developed in the previous sections with two related formulations: the graphon game approach and the finite-player matching game. 
		}
		
		\subsubsection{Graphon Games Formulation}\label{rmk:extended_graphon}

		In the formulation of Problem \ref{def Nash}, type-$\bI$ agents do not interact strategically directly with each other within type; their game is solely against type-$\bJ$ agents (and vice versa). Agents of the same type are influenced by each other only indirectly through the opposite type. Viewed this way, the model is an extended block graphon game with a nonstandard structure, so approaches or results in the existing literature on graphon games are not directly applicable.
		
		Specifically, we consider a continuum of nodes labeled by $[0,1]$, with one agent per node. Partition $[0,1]$ into $\bA \cup \bB$ where $\bA=[0,\frac{\lambda_\bA}{\lambda_\bA+\lambda_\bB})$, and $\bB=[\frac{\lambda_\bA}{\lambda_\bA+\lambda_\bB},1]$ for $\lambda_\bA$ and $\lambda_\bB>0$. Let $\pi_{t}$ be the joint label-state-control probability law on $[0,1]\times \mathbb{R}_{\ge0}\times\{0,1\}\times \mathbb{R}_{\ge 0}$ at time $t$, with uniform first marginal on $[0,1]$. Let $\pi_{\bigcdot,\beta}(dy,dk,du)$ be its disintegration given a label $\beta$. Define the block graphon $G(\alpha,\beta):=\mathbf{1}_{\{\alpha \in \bA\}}\mathbf{1}_{\{\beta \in \bB\}}
		+\mathbf{1}_{\{\alpha \in \bB\}}\mathbf{1}_{\{\beta \in \bA\}}$. For an unmatched agent with label $\alpha$, the graphon–weighted opposite–state-control measure at arrival time $t$ is defined by
		\begingroup \begin{equation}\label{eq:Wpi2}
			(G\pi_t)(\alpha)(dy,dk,du):=\int_0^1 G(\alpha,\beta)\,\pi_{t,\beta}(dy,dk,du)\,d\beta.
		\end{equation}\endgroup 
		{\color{black} We construct a suitable space that depends on the given joint label-state-control probability law $\pi_{\bigcdot}$. For any such $\pi_{\bigcdot}$, there exists a complete probability space $(\widetilde{\Omega},\widetilde{\mathcal F},\widetilde{\mathbb P})$ supporting a uniformly distributed random variable $\alpha: \widetilde{\Omega} \to [0,1]$ representing the representative unmatched agent's label; a Poisson process with rate $\lambda_{\bA}+\lambda_{\bB}$ and arrival times $\{\tau_n\}_{n\in\mathbb{N}}$; and  
			a sequence of random triples $(\mathbb{QL}_{\tau_n},\mathbb{ST}_{\tau_n},\mathbbm{u}_{\tau_n})$, sampled at each arrival time $\tau_n$ from the conditional distribution  
			\[
			\mu_{\tau_n}(\alpha)(dy,dk,dv):=\frac{(G\pi_{\tau_n})(\alpha)(dy,dk,dv)}{(G\pi_{\tau_n})(\alpha)\big( \mathbb{R}_{\ge0}\times\{0,1\}\times \mathbb{R}_{\ge 0}\big)},
			\]  
			where $(G\pi_t)(\alpha)$ is defined in \eqref{eq:Wpi2}. By definition of $G(\alpha,\beta)$, this distribution guarantees that the sampled triple at each arrival time corresponds to an agent of the opposite type relative to $\alpha$. Denote by $\widetilde{\mathbb{E}}$ the expectation on this space. For the representative agent $\alpha$, we let the acceptance threshold be $u_{\alpha}:[0,T]\to \mathbb{R}_{\ge 0}$.
		}

		Let $\mathcal{QL}(\alpha)$ be the quality level of the agent $\alpha$. According to the \emph{FCFS} matching policy, the agent $\alpha$ matches with an opposite agent at time 
		\begingroup \begin{equation*} 
			\tau(\alpha,u_{\alpha},\pi_{\bigcdot},\omega):=\inf\left\{s\in\{\tau_1(\omega),\tau_2(\omega),\ldots\} 
			: 
			\begin{aligned}
				&\mathbb{ST}_{\,s}(\omega)=0,\ \mathbb{QL}_{\,s}(\omega)\geq u_{\alpha(\omega)}(s),\\
				&\mathcal{QL}(\alpha(\omega))\ge \mathbbm{u}_{\,s}(\omega)
			\end{aligned}\right\}.
		\end{equation*}\endgroup
		Given the joint label-state-control law $\pi_{\bigcdot}$, the unmatched agent $\alpha$ selects the threshold control $u_{\alpha}$ to maximize the following objective function: $
		\nonumber J(u_{\alpha};\alpha,\pi_{\bigcdot})
		=\widetilde{\mathbb{E}}\left[\int_0^{\tau\wedge T} e^{-\rho s}\,r\big(\alpha,s\big)\,ds
		+e^{-\rho \tau}\,\mathbb{QL}_{\tau}\,\mathbf{1}_{\{\tau\le T\}}
		+e^{-\rho T }h\big(\alpha\big)\,\mathbf{1}_{\{\tau>T\}}\right],$ where $r (\alpha, s):=r_{\bI}(\mathcal{QL}(\alpha),s)$ and $h(\alpha):=h_{\bI}(\mathcal{QL}(\alpha))$ with $\alpha \in \bI=\bA$ or $\bB$. In this cost, $\tau$ is evaluated at $(\alpha,u_{\alpha},\pi_{\bigcdot} )$ and $\mathbb{QL}_{\bigcdot}$ is sampled from the distribution $\int_0^\infty\sum_{k=0,1}\mu_{\bigcdot}(\alpha)(\bigcdot,dk,dv)$.

		Moreover, the discounted lump sum reward of matching $\widetilde{\mathbb{E}}\pig[e^{-\rho \tau}\,\mathbb{QL}_{\tau}\,\mathbf{1}_{\{\tau\le T\}}\pig]$ can be written in the ratio of dual pairings in terms of $(G\pi_t)(\alpha)(\cdot)$. By the  tower property, we have
		$\widetilde{\mathbb{E}}\pig[e^{-\rho \tau}\,\mathbb{QL}_{\tau}\,\mathbf{1}_{\{\tau\le T\}}\pig] =  \widetilde{\mathbb{E}} \pig[ \widetilde{\mathbb{E}} \pig(\,\mathbb{QL}_{\tau}\,\pig|\tau \pig) e^{-\rho \tau} \mathbf{1}_{\{\tau\le T\}} \pig]$. Conditioning on match occurring at time $\tau = s$, the quality level of agent $\alpha$ and the sampled quality level $\mathbb{QL}_{\tau}$ must be greater than or equal to the corresponding acceptance threshold, and the sampled agent must be unmatched. That is,
		$ \widetilde{\mathbb{E}} \pig[\,\mathbb{QL}_{\tau}\,\pig|\tau = s \pig] =  \widetilde{\mathbb{E}} \pig[\,\mathbb{QL}_{s}\,\mathbf{1}_{\{ \mathbb{ST}_{s} =0, \mathbb{QL}_{s} \geq u_{\alpha}(s), \mathcal{QL}(\alpha)\geq \mathbbm{u}_{s}\}}\pig|\tau = s \pig]$. The disintegration measure of $(\mathbb{QL}_{\,s},\mathbb{ST}_{\,s},\mathbbm{u}_{\,s})$ given the matching occurs at time $\tau=s$ is the normalization of \eqref{eq:Wpi2} on the set $\{\omega: \mathbb{ST}_{s}(\omega) =0, \mathbb{QL}_{s}(\omega) \geq u_{\alpha(\omega)}(s), \mathcal{QL}(\alpha(\omega))\geq \mathbbm{u}_{s}(\omega),\,\tau=s\}$. That is	
		\begingroup 
		\begin{align}
			(\overline{G\pi_t})(\alpha)(dy,dk,dv):=\frac{\int_0^1 G(\alpha,\beta)\mathbf{1}_{\{k=0\}}\,
				\mathbf{1}_{\{y\geq u_{\alpha}(t)\}}\,
				\mathbf{1}_{\{\mathcal{QL}(\alpha)\geq v\}}\,\pi_{t,\beta}(dy,dk,dv)\,d\beta}{\int_0^1 \int_0^{\infty} \sum_{k=0}^1 \int_0^{\infty}G(\alpha,\beta)\mathbf{1}_{\{k=0\}}\,
				\mathbf{1}_{\{y\geq u_{\alpha}(t)\}}\,
				\mathbf{1}_{\{\mathcal{QL}(\alpha)\geq v\}}\,\pi_{t,\beta}(dy,dk,dv)\,d\beta}.
			\label{785}
		\end{align}\endgroup
		Then the disintegration expectation is given by $
		\widetilde{\mathbb{E}} \pig[\,\mathbb{QL}_{s}\,\mathbf{1}_{\{ \mathbb{ST}_{s} =0, \mathbb{QL}_{s} \geq u_{\alpha}(s), \mathcal{QL}(\alpha)\geq \mathbbm{u}_{s}\}}\pig|\tau = s \pig] = \left\langle(\overline{G\pi_{s}})(\alpha), \chi \right\rangle =\frac{\pig\langle(G\pi_{s})(\alpha), \chi \cdot \varphi_{u_{\alpha},\alpha,{s}}\pigr\rangle}{\pig\langle(G\pi_{s})(\alpha), \varphi_{u_{\alpha},\alpha,{s}}\pigr\rangle},$
		where $\chi(y) := y$ is the identity function, and $\varphi_{u_{\alpha},\alpha,t}(y,k,v)
		:=\mathbf{1}_{\{k=0\}}\,
		\mathbf{1}_{\{y\geq  u_{\alpha}(t)\}}\,
		\mathbf{1}_{\{\mathcal{QL}(\alpha)\geq v\}}$ for any admissible threshold $u_{\alpha}$, $t \in [0,T]$, and $(y,k,v) \in \mathbb{R}_{\ge0}\times\{0,1\}\times \mathbb{R}_{\ge 0}$. The objective functional can be written as
		\begingroup \begin{align*}
			\nonumber J(u_{\alpha};\alpha,\pi_{\bigcdot})
			&= \widetilde{\mathbb{E}} \Bigg[\hspace{-1pt}\int_0^{\tau\wedge T} \hspace{-5pt}e^{-\rho s}\,r\big(\alpha,s\big)\,ds
			+e^{-\rho \tau}\hspace{-2pt}\frac{\pig\langle(G\pi_{\tau})(\alpha), \chi \cdot \varphi_{u_{\alpha},\alpha,{\tau}}\pigr\rangle}{\pig\langle(G\pi_{\tau})(\alpha), \varphi_{u_{\alpha},\alpha,{\tau}}\pigr\rangle}\hspace{-2pt}\mathbf{1}_{\{\tau\le T\}}
			+e^{-\rho T }h\big(\alpha\big)\,\mathbf{1}_{\{\tau>T\}}\hspace{-1pt}\Bigg].
		\end{align*}\endgroup

		The graphon equilibrium is the pair $(\pi_{\bigcdot}^*,\{ u^*_{\alpha}\}_{\alpha\in[0,1]})$ such that \textup{(i).} $
		J(  u^*_{\alpha};\alpha,\pi_{\bigcdot}^*)\ \ge\ J(  u_{\alpha};\alpha,\pi_{\bigcdot}^*)$ for a.e. $\alpha\in[0,1]$
		and all admissible $u_{\alpha}$; and \textup{(ii).} $\pi_{\bigcdot}^*$ is the joint law of the labels, quality
		levels, statuses and thresholds $(\alpha,\mathcal{QL}(\alpha),\mathcal{ST}^{*}(\alpha,\bigcdot),u^*_{\alpha}(\bigcdot))$ where $\mathcal{ST}^{*}(\alpha,t) := \mathbf{1}_{\{\tau(\alpha,u^*_\alpha,\pi_{\bigcdot}^*) \leq t\}}$ (the presence of $\pi_{\bigcdot}^*$ in $\mathcal{ST}^{*}(\alpha,t)$ implies that the equilibrium is attained at the fixed point). 
		
		This formulation is equivalent to the mean field Nash game in \Cref{def Nash}: disintegrating $\pi_{\bigcdot}^*$ over $\bA$ and $\bB$ recovers the type–wise laws and the same optimal acceptance thresholds. There are, however, two key differences between this formulation and the standard graphon games in the literature. First, our formulation is an extended graphon game in which interaction depends on the joint law of states and controls: the individual control problem of the representative agent depends simultaneously on the states and controls of other agents in a non-separable manner. Second, the presence of the matching stopping time $\tau$ in the objective functional complicates our problem. Due to the dependence between $\tau$ and the state-control sampling of the opposite type, the lump sum reward of matching $\mathbb{QL}_{\,\tau}$ can be expressed only as a ratio of two dual pairings. These features fall outside the standard graphon-game framework, so existing results do not apply.

		\subsubsection{Finite‑player Games Formulation}\label{rmk:finite‑player}

		{\color{black} Let $n_\bA,n_\bB\in \mathbb{N}$, and consider a market with $n_\bA$ type‑$\bA$ agents and $n_\bB$ type‑$\bB$ agents with index sets $\mathbf{A}$ and $\mathbf{B}$, respectively. Let $(\widehat \Omega,\widehat{\mathcal{F}},\widehat{\mathbb{F}},\widehat{\mathbb{P}})$ be a complete filtered probability space on which we work, with $\widehat{\mathbb{F}}=\{\widehat{\mathcal{F}}_{t}\}_{t\ge 0}$ and corresponding expectation $\widehat{\mathbb{E}}$. We assume that each agent meets any other agent (not necessarily on the opposite side) according to a Poisson process with intensity $\lambda>0$. Then the intensities that any type-$\bA$ agent meets any other agent of type-$\bA$ and type-$\bB$ are given by $\frac{\lambda (n_\bA-1)}{n_\bA + n_\bB -1}$ and  $\frac{\lambda n_\bB}{n_\bA + n_\bB -1}$, respectively. For a type-$\bB$ agent, the corresponding intensities are $\frac{\lambda n_\bA}{n_\bA + n_\bB -1}$ and $\frac{\lambda (n_\bB-1)}{n_\bA + n_\bB -1}$. Let $N^{\text{mkt}}_{\bigcdot}$ be the Poisson process counting all cross-type meetings, with intensity $\frac{\lambda n_\bA n_\bB}{n_\bA + n_\bB -1}$, and let $\{T^{\text{mkt}}_n\}$ be its arrival times. For any agent $i \in \bA \cup \bB$, let $N_t(i)$ be the Poisson process counting the number of meetings of agent $i$ with agents of the opposite type in the time interval $[0,T]$, and denote its arrival times by $\{T_n(i)\}$. It is clear that $\{T_n(i)\} \subseteq \{T^{\text{mkt}}_n\}$ and $\cup_{a \in \bA}\{T_n(a)\} = \cup_{b \in \bB}\{T_n(b)\} = \{T^{\text{mkt}}_n\}$, $\widehat{\mathbb{P}}$-a.s. Note that for any $a\in \bA$ and meeting time $T_n(a)$, there exists $k \in \mathbb{N}$ and $b \in \bB$ such that $T_n(a) = T_k(b)$. This shows that the meetings are bilaterally coherent.
			
			For any given stopping time $s \in \{T^{\text{mkt}}_n\}$, let $a_s$ and $b_s$ be the uniformly distributed indices random variables taking values in $\bA$ and $\bB$, respectively, representing the indices of the agents involved in the meeting at time $s$. For $\bI \in\{ \bA, \bB\}$, we define two maps $\mathbf{ql}_{\bI}: \bI \to X_{\bI}$ and $\mathbf{st}_{\bI,\bigcdot}: \bI\times[0,T] \to \{0,1\}$ such that $\mathbf{ql}_{\bI}(i)$ and $\mathbf{st}_{\bI,t}(i)$ respectively denote the quality level and matching status of the agent $i$ at time $t\in[0,T]$.
			
			For given threshold strategies of all agents, i.e., $u_{\bB}(b,\bigcdot)$ for every $b\in\bB$ and $u_{\bA}(a,\bigcdot)$ for every $a\in\bA$, the matching status of all agents can be characterized by the following coupled dynamics
			\begin{equation}\label{sde.matching_statusb}
				\mathbf{st}_{\bB,t}(b) = \sum_{i=1}^{N_t(b)} \mathbf{1}_{\{\mathbf{st}_{\bA,T_i(b)^-}(a_{T_i(b)})=0\}} \mathbf{1}_{\{\mathbf{st}_{\bB,T_i(b)^-}(b) = 0\}} \mathbf{1}_{\{\mathbf{ql}_{\bB}(b) \geq u_{\bA}(a_{T_i(b)},T_i(b))\}} \mathbf{1}_{\{\mathbf{ql}_{\bA}(a_{T_i(b)}) \geq u_{\bB}(b,T_i(b))\}},
				\hspace{-10pt}
			\end{equation}
			\begin{equation}\label{sde.matching_statusa}
				\mathbf{st}_{\bA,t}(a) = \sum_{i=1}^{N_t(a)} \mathbf{1}_{\{\mathbf{st}_{\bB,T_i(a)^-}(b_{T_i(a)})=0\}} \mathbf{1}_{\{\mathbf{st}_{\bA,T_i(a)^-}(a) = 0\}} \mathbf{1}_{\{\mathbf{ql}_{\bA}(a) \geq u_{\bB}(b_{T_i(a)},T_i(a))\}} \mathbf{1}_{\{\mathbf{ql}_{\bB}(b_{T_i(a)}) \geq u_{\bA}(a,T_i(a))\}},
				\hspace{-10pt}
			\end{equation}
			with $\mathbf{st}_{\bB,0}(b) = \mathbf{st}_{\bA,0}(a) = 0$, for all $a \in \bA$, $b \in \bB$. The dynamics indicate that any two agents {\color{black}with different types become matched} if they are both unmatched and each agent’s quality level exceeds the other’s acceptance threshold.

			For a given type‑$\bA$ agent $a\in \bA$, taking as given the threshold strategies of all other agents, i.e., $u_{\bB}(b,\bigcdot)$ for every $b\in\bB$ and $u_{\bA}(a',\bigcdot)$ for every $a'\in\bA\setminus\{a\}$, the agent $a$ chooses a threshold function $u_\bA(a,\cdot):[0,T]\rightarrow [0,\infty)$ to maximize the objective functional
			\begingroup \begin{align}\label{prob:finite_player}
				&\hspace{-5pt}J_{\bA}\pig(u_{\bA}(a,\bigcdot);a,\{u_{\bA}(a',\bigcdot)\}_{a'\neq a},u_{\bB}\pig)\nonumber\\ 
				&\hspace{-5pt}:= \widehat{\mathbb{E}}\bigg[\!\!\int_0^{\tau_{\bA}(a) \land T} \hspace{-8pt}e^{-\rho s} r_{\bA}(\mathbf{ql}_{\bA}(a),s) \, ds 
				+ e^{-\rho \tau_{\bA}(a)} \mathbf{ql}_{\bB}(b_{\tau_{\bA}(a)}) \mathbf{1}_{\{\tau_{\bA}(a) \leq T\}} + e^{-\rho T} \mathbf{1}_{\{\tau_{\bA}(a) > T\}} h_{\bA}(\mathbf{ql}_{\bA}(a))\bigg],
				\hspace{-5pt}
			\end{align}\endgroup
			where $\tau_{\bA}(a) := \inf\Big\{s\in \{T_n(a)\}: \mathbf{st}_{\bB,s^-}(b_{s})=0,\; \mathbf{ql}_{\bA}(a) \geq u_{\bB}(b_{s},s),\; \mathbf{ql}_{\bB}(b_{s}) \geq u_\bA(a,s)\Big\}.$ The status dynamics of the agent $a$ is given by \eqref{sde.matching_statusa}. A similar objective functional can be defined for agents on the opposite side.

			As noted above, the evolutions of the status processes $\mathbf{st}_{\bI,\bigcdot}$ of agents are coupled and must be determined jointly. One can verify that $\mathbf{st}_{\bA,\bigcdot}(a)$ and $\mathbf{st}_{\bB,\bigcdot}(b)$ are adapted to the filtration $\widehat{\mathbb{F}}$. A pair of strategies $(u^*_{\bA},u^*_{\bB})$ constitutes a Nash equilibrium of this finite‑player matching game if, for every agent $a\in\bA$ (resp. $b\in\bB$), given that all other agents on both sides adopt $(u^*_{\bA},u^*_{\bB})$, the strategy $u^*_{\bA}(a,\bigcdot)$ (resp. $u^*_{\bB}(b,\bigcdot)$) is optimal for problem \eqref{prob:finite_player}.
			
			When the number of agents is large but finite, analyzing Nash equilibria in the above setting becomes challenging for two main reasons. First, each agent’s decision and state have a non-negligible effect on the pool of available partners, making the system fully coupled. Second, without the law of large numbers, the empirical joint distribution of quality level and matching status of agents in each type is itself a random measure that evolves stochastically. Its evolution is governed by stochastic integro-differential equations.
			
			To overcome these difficulties, we instead analyze a mean‑field setting, obtained as the limit $n_\bA,n_\bB\to\infty$ while keeping the relative population sizes fixed. This approach is widely used to analyze large‑population matching problems \cite{perthame2018career,li2025repositioning,KanoriaSaban2021}. In this limit, the law of large numbers applies, and the impact of any individual agent on the aggregate system becomes negligible. As a result, three key differences emerge relative to the finite-player setting:
			
			\begin{enumerate}[(1).]
				\item The evolution of the joint distribution of quality and status is governed by a controlled deterministic forward equation.
				\item In the continuum limit, the probability that two specific individuals meet is zero; instead, meetings are described probabilistically through the current distribution of quality levels and status on the opposite side of the market \cite{duffie2007existence,duffie2025continuous,duffie2018dynamic}.  
				\item Because each agent has a negligible effect on the population distribution, we can study the decision problem of a representative agent given the distribution of agents on the opposite side, as in Section \ref{sec. Dynamic Matching Problem with Finite Agents}. 
			\end{enumerate}
			These features make the analysis of large‑population two‑sided matching problems tractable. In the mean-field limit, each type-$\bA$ agent experiences his own Poisson stream of meeting opportunities with rate $\lambda_{\bB}$, and each type-$\bB$ agent faces one with rate $\lambda_{\bA}$. Heuristically, this corresponds to $\frac{\lambda n_\bB}{n_\bA+n_\bB -1} \to \lambda_{\bB}$ and $\frac{\lambda n_\bA}{n_\bA+n_\bB -1} \to \lambda_{\bA}$, as $n_\bA, n_\bB \to \infty$ with $\frac{n_\bA}{n_\bB}\to \Lambda$ for some $0<\Lambda<\infty$. These intensities satisfy the coherence condition $\frac{\lambda_{\bA}}{\lambda_{\bB}} = \Lambda$.
			
			Finally, we note that the connection between the finite‑player system and its mean‑field limit can be rigorously established via propagation‑of‑chaos arguments in many related contexts, such as \cite{andreis2018mckean,locherbach2025strong}. Extending such results to our matching model would be an interesting direction for future research. In the present paper, however, we focus exclusively on the formulation, well‑posedness, and economic insights of the mean‑field equilibrium.
			
		}
		
		\subsection{Applications}\label{sec. applications}
		
		Our model applies to a variety of real-world markets, including labor markets, over-the-counter (OTC) markets for second-hand houses, and initial public offering (IPO) markets.\vspace{5pt}
		
		\paragraph{Labor markets:} The labor market consists of job seekers (group $\bA$) and hiring managers of firms (group $\bB$). A job seeker is characterized by his/her ability level $x \in \mathbb{R}_{\geq 0}$, and a hiring manager by the level or ranking $y \in \mathbb{R}_{\geq 0}$ of the job vacancy offered. The market operates repeatedly over successive periods of horizon $[0,T]$ (e.g., $T=1$ year). At the start of each period, both job seekers and hiring managers enter the market to search for matches.
		
		During the search period, a job seeker receives running utility $r_\bA(x,t)$, such as income from part-time work, job-seeking assistance, or other temporary benefits. The term $r_\bB (y,t)$ corresponds to the opportunity cost or return the company gains from idle capital while waiting to hire, such as from investments or reallocating internal resources.
		
		If a hiring manager successfully hires (matches) a job seeker with ability $x$ at time $\tau_\bB \leq T$, the manager receives a lump sum utility $x$, interpreted as the present discounted value of the total profits generated by the job seeker over their tenure (e.g., over a 30-year contract). Conversely, if a job seeker is hired (matched) by a firm with ability $y$ at time $\tau_\bA \leq T$, the job seeker obtains a lump sum utility $y$, interpreted as the present discounted value of the total benefits from the job (e.g., lifetime earnings discounted to the present).
		
		Agents who remain unmatched at the terminal time $T$ exit the current market and receive terminal utilities that reflect their expected outcomes in the subsequent period. For a job seeker, the terminal utility $h_\bA$ represents the average utility from participating in the next period's labor market or from self-employment alternatives. For a hiring manager, the terminal utility $h_\bB$ represents the average value derived from finding an employee in the next period's market.
		
		The dynamic asymmetric selection problem applies to this repeated labor market as follows: Hiring managers (Type-$\bB$ agents) post job vacancies with a predetermined hiring threshold $u_\bB$, sequentially interview job seekers according to the exogenous meeting process $\{N_{\bA,t}\}_{t\ge0}$, and follow a first-acceptance policy, hiring the first job seeker whose ability $x$ satisfies $x\ge u_\bB$. Job seekers (Type-$\bA$ agents) are sequentially interviewed by firms according to an exogenous meeting process $\{N_{\bB,t}\}_{t\ge0}$ and adopt an optimal stopping strategy, accepting the most recent job offer from a firm offering a position of level $y$ whenever the offer exceeds the job seeker’s reservation utility (dynamically determined based on the value of continued search), after which the job search terminates.
		
		\paragraph{OTC market for second-hand houses:} The OTC market of second-hand houses involves buyers (group $\bA$) and sellers (group $\bB$). Each buyer is identified by his/her budget or purchasing power $x \in \mathbb{R}_{\geq 0}$; while each seller is characterized by the quality level $y \in \mathbb{R}_{\geq 0}$ of his/her house. The term $r_\bA (x,t)$ corresponds to the risk-free interest (or opportunity cost) on the buyer's idle capital while searching for a house to purchase. The term $r_\bB (y,t)$ represents the rental income the seller receives while the property is still unsold. If the buyer fails to purchase a suitable property by the terminal time, they exit the market and receive a terminal utility represented by the function $h_\bA$, which reflects the buyer’s fallback options, such as investing their unused funds elsewhere or continuing to rent. It quantifies the opportunity value retained by the buyer in the absence of a transaction. If the seller fails to sell their property by the terminal time, then they exit the market and receive a terminal utility given by the function $h_\bB $ which represents the value of alternative outcomes—such as continuing to rent out the property.\vspace{5pt}
		
		\section{Main Results}\label{sec. main results}
		In this section, we present the well-posedness results of Problems \ref{pblm:type A, threshold} and \ref{def Nash}. They are established by studying the corresponding HJB–FP system. 
		
		\subsection{Derivation of HJB and FP Equations}
		
		In this section, we outline the formal derivation of the HJB and FP equations associated with the control problems presented in Problem~\ref{pblm:type A, threshold}, which does not follow directly from existing methods, and no immediate or classical HJB formulation applies. By governing the systems corresponding to Problem~\ref{pblm:type A, threshold} of type-$\bA$ and type-$\bB$ agents, we obtain the fully coupled HJB–FP system \eqref{HJBx}-\eqref{boundary} that characterizes the equilibrium of the mean field game formulated in Problem~\ref{def Nash}. The global well-posedness of the system is rigorously established in Appendices~\ref{sec:existence}–\ref{sec. unique}, with Theorems~\ref{maintheorem} and \ref{thm unique} providing the main results. Furthermore, Theorem~\ref{thm.verification} verifies that the obtained solution indeed constitutes a mean field Nash equilibrium of Problem~\ref{def Nash}. {\color{black}As we work at equilibrium throughout, the exogenous thresholds and distributions $(v_{\mathbf{A}}, v_{\mathbf{B}}, \psi_{\mathbf{A},\bigcdot}, \psi_{\mathbf{B},\bigcdot})$ (as introduced in \Cref{sec Problem Settings}) coincide with the optimal ones 
			$(u_{\mathbf{A}}, u_{\mathbf{B}}, \pi_{\mathbf{A},\bigcdot}, \pi_{\mathbf{B},\bigcdot})$ at equilibrium. Therefore, we use only the latter throughout this section.}

		\subsubsection{Evolution of Defective Probability Density Functions}\label{sec.markov} 
		
		We first characterize the matching event. For $\bI, \bJ\in \{\bA,\bB\}$ with $\bI \neq \bJ$, we suppose that type-$\bI$ and type-$\bJ$ agents adopt threshold strategies $u_{\bI}\in \mathcal{U}_{\bI}$ and $u_{\bJ} \in \mathcal{U}_{\bJ}$ respectively. According to the matching mechanism in \Cref{ass.matching rule}, the time-dependent matching region of the representative type-$\bI$ agent $(x,t)$ meeting a random type-$\bJ$ agent is defined by\begingroup 
		\begin{equation}\label{def: matching region}
			\mathcal{O}_{\bI}(x,t):=\big\{y\in X_\bJ:u_{\bI} (x,t) \leq y \quad\text{and}\quad x\geq u_{\bJ}(y,t)\big\}.
		\end{equation}\endgroup
		This set has the following interpretation: if the representative type-$\bI$ agent $(x,t)$ meets an unmatched type-$\bJ$ agent $(y,t)$, then a match occurs if and only if $y \in \mathcal{O}_{\bI}(x,t)$.  With a proper flow of defective probability distribution $\mu_{\bJ,\bigcdot}$ of quality level of unmatched type-$\bJ$ agents, the conditional probability that the representative type-$\bI$ agent $(x,t)$ meets and matches with a randomly sampled type-$\bJ$ agent $(\mathbf{QL}_{\bJ,\bigcdot},\mathbf{ST}_{\bJ,\bigcdot})$ at time $t$ is given by
		{\color{black}\begingroup \begin{align}\label{645}
				\mathbb{P}_{\bI,x,t}\pig(\mathbf{QL}_{\bJ,t} \in \mathcal{O}_{\bI}(x,t) \text{ and } \mathbf{ST}_{\bJ,t} = 0 \,\pig|\, \text{$\exists \, \tau^\bJ \in \mathcal{T}^\bJ$ s.t. $\tau^\bJ = t$}\pig) 
				= \mu_{\bJ,t} \pig(\mathcal{O}_{\bI}(x,t)\pig)
			\end{align}\endgroup}
		for $(x,t) \in X_\bI \times [0,T]$. For $(x,t) \in X_\bI \times [0,T)$, we take $h > 0$ small enough such that $t+h<T$. We have\begingroup 
		\begin{equation*}
			\mathbb{P}_{\bI,x,0}(\tau_{\hspace{0.5pt}\bI}(x,0) >t) - \mathbb{P}_{\bI,x,0}(\tau_{\hspace{0.5pt}\bI}(x,0) >t+h)
			= \, \mathbb{P}_{\bI,x,0}(\tau_{\hspace{0.5pt}\bI}(x,0) >t) \mathbb{P}_{\bI,x,0}\pig(\tau_{\hspace{0.5pt}\bI}(x,0) \leq t+h \,\pig| \,\tau_{\hspace{0.5pt}\bI}(x,0) > t \pig).
		\end{equation*}\endgroup
		By definitions in \eqref{def.PIxt} and \eqref{def match time of A}, we have, for any $(x,t) \in X_\bI \times [0,T]$,
		\begingroup \begin{equation}\label{property.memoryless}
			\mathbb{P}_{\bI,x,0}\pig(\tau_{\hspace{0.5pt}\bI}(x,0) \in \bigcdot \,\pig|\, \tau_{\hspace{0.5pt}\bI}(x,0) > t \pig) = \mathbb{P}_{\bI,x,t}\pig(\tau_{\hspace{0.5pt}\bI}(x,t) \in \bigcdot\pig).
		\end{equation}\endgroup
		In addition, by the matching mechanism in \Cref{ass.matching rule} and \cite[Definition 2.1.2]{ross1995stochastic}, and by the law of total probability, we have
		{\color{black}\begingroup \begin{align*}
				&\mathbb{P}_{\bI,x,t}\pig(\tau_{\hspace{0.5pt}\bI}(x,t) \leq t+h \pig) \\
				&=\, \pig[\!\lambda_\bJ h + o(h)\!\pig]\!\pig[
				\mathbb{P}_{\bI,x,t}\pig(\mathbf{QL}_{\bJ,t} \in \mathcal{O}_{\bI}(x,t)\, \text{ and }\, \mathbf{ST}_{\bJ,t} = 0 \,\pig|\, \text{$\exists \, \tau^\bJ \in \mathcal{T}^\bJ$ s.t. $\tau^\bJ = t$}\pig) + o(1)\pig]\hspace{-20pt} \nonumber\\
				&=\,  \lambda_\bJ h \mathbb{P}_{\bI,x,t}\pig(\mathbf{QL}_{\bJ,t} \in \mathcal{O}_{\bI}(x,t)\, \text{ and }\, \mathbf{ST}_{\bJ,t} = 0  \,\pig|\, \text{$\exists \, \tau^\bJ \in \mathcal{T}^\bJ$ s.t. $\tau^\bJ = t$}\pig) + o(h).
			\end{align*}\endgroup}%
		Therefore, using \eqref{645} gives $\mathbb{P}_{\bI,x,0}(\tau_{\hspace{0.5pt}\bI}(x,0) >t) - \mathbb{P}_{\bI,x,0}(\tau_{\hspace{0.5pt}\bI}(x,0) >t+h)
		= \mathbb{P}_{\bI,x,0}(\tau_{\hspace{0.5pt}\bI}(x,0) >t)\lambda_\bJ \mu_{\bJ,t} \pig(\mathcal{O}_{\bI}(x,t)\pig) \hspace{0.5pt} h + o(h).$ When $\tau_\bI(x,t)\le T$, it is $\mathcal{T}^\bJ$-valued by \eqref{def match time of A}, so its distribution has no atom on $[0,T]$. Suppose, for contradiction, that we have $\mathbb{P}_{\bI,x,t}(\tau_\bI(x,t) = s ) > 0$  for some $s\in[t,T]$. By the tower property, $0 < \mathbb{P}_{\bI,x,t}(\tau_\bI(x,t) = s ) = \sum_{k=1}^\infty \mathbb{P}_{\bI,x,t}(\tau_\bI(x,t) = s\,|\,\tau_\bI(x,t) = \tau_{\bJ,k}) \mathbb{P}_{\bI,x,t}(\tau_\bI(x,t) = \tau_{\bJ,k}) = \sum_{k=1}^\infty \mathbb{P}_{\bI,x,t}(\tau_{\bJ,k} = s) \mathbb{P}_{\bI,x,t}(\tau_\bI(x,t) = \tau_{\bJ,k})$. It would imply that at least one of the arrival times has mass in its distribution, which is obviously not true. Thus, by $\mathbb{P}_{\bI,x,0}(\tau_{\hspace{0.5pt}\bI}(x,0) > t) = \mathbb{P}_{\bI,x,0}(\tau_{\hspace{0.5pt}\bI}(x,0) \geq t) = \mathbb{P}_{\bI,x,0}(\tau_{\hspace{0.5pt}\bI}(x,0) > t^-)$, we obtain that $\mathbb{P}_{\bI,x,0}(\tau_{\hspace{0.5pt}\bI}(x,0) >t)$ is continuous with respect to $t$. Hence, we have
		\begingroup \begin{equation}\label{eq.Ptau}
			-\frac{\p \mathbb{P}_{\bI,x,0}(\tau_{\hspace{0.5pt}\bI}(x,0) >t)}{\p t} = \mathbb{P}_{\bI,x,0}(\tau_{\hspace{0.5pt}\bI}(x,0) >t)\lambda_\bJ \mu_{\bJ,t} \pig(\mathcal{O}_{\bI}(x,t)\pig), \quad \text{for $t\in[0,T)$},
		\end{equation}\endgroup
		with the initial condition $\mathbb{P}_{\bI,x,0}(\tau_{\hspace{0.5pt}\bI}(x,0) >0)=1$.
		By \Cref{prop.fxt3}, for any bounded measurable test function $\phi$ on $X_\bI$, we have
		\begingroup \begin{equation}\label{def.fxt.2}
			\int_0^\infty \phi(x) d\mu_{\bI,t}(x) = \int_0^\infty \phi(x) \mathbb{P}_{\bI,x,0}(\tau_{\hspace{0.5pt}\bI}(x,0) > t) d\mu_{\bI ,0}(x).
		\end{equation}\endgroup
		Differentiating with respect to  $t$ on both sides of \eqref{def.fxt.2}, together with \eqref{eq.Ptau}, we have 
		\begingroup \begin{align}\label{def.fxt.3}
			\frac{\partial }{\partial t}
			\int_0^\infty \hspace{-5pt} \phi(x) d \mu_{\bI,t}(x) 
			=  \hspace{-1pt}-\hspace{-3pt}\int_0^\infty \hspace{-5pt}\phi(x) \lambda_\bJ \mu_{\bJ,t} \pig(\mathcal{O}_{\bI}(x,t)\pig) d\mu_{\bI ,t}(x).
		\end{align} \endgroup
		As $\phi$ is arbitrary, equation \eqref{def.fxt.3} can be interpreted as $
		\frac{d \frac{\partial \mu_{\bI,t}}{\partial t}}{d \mu_{\bI ,t}} (x)=- \lambda_\bJ \mu_{\bJ,t}\pig(\mathcal{O}_{\bI}(x,t)\pig),$ with the initial condition $\lim_{t \downarrow 0}\mu_{\bI,t} (\bigcdot)=\mu_{\bI,0} (\bigcdot)$, provided that $\mu_{\bI,t} $ is regular enough.
		
		At the end of this subsection, we establish a lemma showing that $\mu_{\bI,t} (X_\bI)=F_\bI (t)$ defined in \eqref{def. unmatched F_I} is continuous and strictly positive for any $t\in[0,T]$. This result implies that, whenever the market closes, there remains a strictly positive proportion of agents who are unmatched. 
		\begin{lemma}\label{lem:FAB>0}
			For any $\bI \in \{\bA,\bB\}$, $F_\bI (t)$ is continuous with respect to $t$ and satisfies $F_\bI (t) >0$ for all $t \in [0,T]$.
		\end{lemma}
		\begin{proof} 
			Let $\bI, \bJ\in \{\bA,\bB\}$ with $\bI \neq \bJ$. By the continuity of $\mathbb{P}(\tau_{\hspace{0.5pt}\bI}(x,0) >t)$ in $t$ given by \eqref{eq.Ptau}, \Cref{prop.Ft} implies that $F_\bI (t)=\mathbb{P}(\tau_{\hspace{0.5pt}\bI} (\mathbf{ql}_{\bI},0)> t)$ is continuous. By \eqref{eq.Ptau}, we have $
			\mathbb{P}(\tau_{\hspace{0.5pt}\bI}(x,0) >t)= \exp\left[-\lambda_\bJ  \int_0^t\mu_{\bJ,s} \pig(\mathcal{O}_{\bI}(x,s)\pig)ds\right] >0,$ for $t\in[0,T)$ and $x \geq 0$.   
		\end{proof}

		\subsubsection{Derivation of HJB Equations}\label{sec.HJB}

		In this section, we formally derive the HJB equation associated with Problem~\ref{pblm:type A, threshold}, aiming to provide intuition for how the equation arises. Let $\bI, \bJ\in \{\bA,\bB\}$ with $\bI \neq \bJ$. Given a fixed threshold function $u_{\bJ} \in \mathcal{U}_{\bJ}$ and a joint distribution $\pi_{\bJ,\bigcdot} \in \mathscr{S}_\bJ$ for type-$\bJ$ agents with $\mu_{\bJ,\bigcdot}(\bigcdot)=\pi_{\bJ,\bigcdot}(\bigcdot,0)$. Suppose that they are regular enough. For convenience, we denote the corresponding value function $\mathcal{V}_\bI(x,t,u_\bJ,\pi_{\bJ,\bigcdot})$ in \eqref{def.V_I.threshold} simply by $\mathcal{V}_\bI(x,t)$. We apply the first jump analysis, a standard argument widely used in the study of renewal processes, to derive the HJB equation. Specifically, fix $(x,t) \in \mathbb{R}_{\geq 0} \times [0,T)$ and choose $h > 0$ small enough such that $t+h < T$. There are three possible outcomes within $[t, t+h]$: (i) a meeting occurs and results in a match; (ii) a meeting occurs but no match is formed; or (iii) no meeting takes place. Using the tower property by conditioning on the events (i)-(iii) within $[t, t+h]$, and letting $h \to 0^+$, we formally derive the HJB equation:
		\begingroup \begin{equation}\label{HJB_maximize.VA}\begin{aligned}
				\rho \mathcal{V}_\bI (x,t)-\frac{\partial \mathcal{V}_\bI (x,t) }{\partial t}
				=\lambda_\bJ \sup_{z\geq 0}\left( 
				\int_{z}^{\infty}\mathbf{1}_{\{u_\bJ(y,t)\leq x\}}
				\big[y-\mathcal{V}_\bI (x,t)\big]
				d\mu_{\bJ,t} (y)\right)+r_\bI (x,t)  
			\end{aligned}
		\end{equation}\endgroup
		on $\mathbb{R}_{\geq 0} \times [0,T)$, with the terminal condition $\mathcal{V}_\bI (x,T)=h_\bI (x)$ on $\mathbb{R}_{\geq 0}.$ We further reduce the above equation. Fix $(x,t) \in \mathbb{R}_{\geq 0} \times [0,T]$ and define
		\begingroup \begin{align}		\label{1204} 
			H(z):=
			\int_{z}^{\infty} \mathbf{1}_{\{u_\bJ(y,t)\leq x\}}
			\big[y-\mathcal{V}_\bI (x,t)\big] d\mu_{\bJ,t} (y).
		\end{align}\endgroup 
		Since for any $z\ge 0$, we have $\int_{z}^{\mathcal{V}_\bI (x,t)} \mathbf{1}_{\{u_\bJ(y,t)\leq x\}}\big[y-\mathcal{V}_\bI (x,t)\big] d\mu_{\bJ,t} (y)\le 0.$ It follows that
		\begingroup \begin{align*}
			H(z)=\,& \int_{z}^{\mathcal{V}_\bI (x,t)} \mathbf{1}_{\{u_\bJ(y,t)\leq x\}}\big[y-\mathcal{V}_\bI (x,t)\big] d\mu_{\bJ,t} (y)+H(\mathcal{V}_\bI (x,t))\le H(\mathcal{V}_\bI (x,t)).
		\end{align*}\endgroup
		Hence, a maximizer (may not unique) of $H$ is attained at
		\begingroup \begin{equation}\label{optimalcontrol.A}
			z^* = \mathcal{V}_\bI (x,t)
		\end{equation}\endgroup
		which simplifies the HJB equation \eqref{HJB_maximize.VA} to
		\begingroup 
		\begin{equation}\label{eq HJB VA}
			\left\{\begin{aligned}
				&\rho \mathcal{V}_\bI (x,t)-\frac{\partial \mathcal{V}_\bI (x,t) }{\partial t}=\lambda_\bJ 
				\int_{\mathcal{V}_\bI (x,t)}^{\infty}
				\mathbf{1}_{\{u_\bJ(y,t)\leq x\}} 
				\big[y-\mathcal{V}_\bI (x,t)\big] d\mu_{\bJ,t} (y)  +r_\bI (x,t), && \text{on $\mathbb{R}_{\geq 0} \times [0,T)$}, \\[5pt]
				&\mathcal{V}_\bI (x,T)=h_\bI (x)
				&& \text{on $\mathbb{R}_{\geq 0}$.} 
			\end{aligned}
			\right.
		\end{equation}\endgroup
		Note that the value of the Hamiltonian in \eqref{HJB_maximize.VA} is independent of the particular choice of maximizer $z^*$, as long as $z^*$ attains the maximum of $H$. In other words, \eqref{HJB_maximize.VA} and \eqref{eq HJB VA} are equivalent.

		\subsection{Existence and Uniqueness of Mean Field Nash Equilibrium of Problem \ref{def Nash}}\label{sec.equilibrium}
		Now we state the fully-coupled forward backward HJB-FP system associated with the equilibrium in \Cref{def Nash} and the global well-posedness result. To ensure solvability, we assume that the initial distributions are absolutely continuous:
		\begin{assumption}[\bf Initial Distributions]\label{condition_initial} 
			For any $\bI \in \{\bA,\bB\}$, the initial distribution $\mu_{\bI,0}$ given in \Cref{sec. Probabilistic Setting and Meeting Mechanism} is absolutely continuous with respect to the Lebesgue measure with density $f_{\bI,0} \in L^1(\mathbb{R}_{\geq 0}; \mathbb{R}_{\geq 0})$. Moreover, there exist positive constants $\nu>0$ and $ C_\bA, C_\bB >1$ such that   $f_{\bI ,0}(z) \le \frac{ C_\bI }{1+z^{2+\nu}}$ for all $z \in \mathbb{R}_{\geq 0}$.  
		\end{assumption} 
		Therefore, by \Cref{prop.fxt3}, we have $\mu_{\bI,t} \ll \mu_{\bI,0}$ and $f_\bI(x,t) := \mathbb{P}(\tau_{\hspace{0.5pt}\bI}(x,0) > t) f_{\bI ,0}(x)$ is the density function of $\mu_{\bI,t}$. Hence, formally, equation \eqref{def.fxt.3} can be expressed as
		{\color{black}\begingroup \begin{align}
				\frac{\partial f_{\bI}(x,t)}{\partial t}=- \lambda_\bJ f_{\bI}(x,t)  
				\int_{u_{\bI} (x,t)}^{\infty}
				\mathbf{1}_{\{u_\bJ(y,t)\leq x\}} 
				f_{\bJ}(y,t)dy.
				\label{fp.fA0}
			\end{align}\endgroup}%
		Putting $u_\bI=u^*_\bI=\mathcal{V}_\bI$ in \eqref{fp.fA0}, we see that the joint distribution $\pi^*_{\bI,\bigcdot}$ constructed from the solution of this equation is independent of the choice of optimal control, once the objective functional is maximized. To enhance the presentation of the system of equations, we denote by \(\mathcal{V}^{-1}_\bI(\bigcdot,t)\) the inverse function of \(\mathcal{V}_\bI(\bigcdot,t)\) for each \(t\in[0,T]\), whenever it exists. 
		The existence of this inverse will be established in Theorem \ref{maintheorem}. If we couple \eqref{eq HJB VA} and \eqref{fp.fA0} with $u_\bI=u^*_\bI=\mathcal{V}_\bI$ for $\bI \in \{\bA,\bB\}$, then we aim to solve $(V_\bA,V_\bB,f_\bA,f_\bB)$ for the following fully coupled HJB-FP system over $\mathbb{R}_{\geq 0} \times [0,T]$:
		\begin{align}
			\label{HJBx}
			\rho V_\bA (x,t) - \frac{\partial V_\bA (x,t)}{\partial t}
			=  \mathbf{1}_{\{V_\bA (x,t) \leq V_\bB ^{-1}(x,t)\}}
			\lambda_\bB
			\int_{V_\bA (x,t)}^{V_\bB ^{-1}(x,t)} \big[y - V_\bA (x,t)\big] f_\bB (y,t) \, dy 
			+ r_\bA (x,t); \end{align}
		\begin{align}\label{HJBy}
			\rho V_\bB (y,t) - \frac{\partial V_\bB (y,t)}{\partial t}
			=  \mathbf{1}_{\{V_\bB (y,t) \leq V_\bA^{-1}(y,t)\}}
			\lambda_\bA
			\int_{V_\bB (y,t)}^{V_\bA ^{-1}(y,t)} \big[x - V_\bB (y,t)\big] f_\bA (x,t) \, dx 
			+ r_\bB (y,t); \end{align}
		\begin{align}\label{FPx}
			\frac{\partial f_\bA (x,t)}{\partial t}
			= 
			\left(-\lambda_\bB   f_\bA (x,t) \int_{V_\bA (x,t)}^{V_\bB ^{-1}(x,t)} f_\bB (y,t) \, dy \right)\wedge 0;\end{align}
		\begin{align}\label{FPy}
			\frac{\partial f_\bB (y,t)}{\partial t}
			= 
			\left( -\lambda_\bA    f_\bB (y,t) \int_{V_\bB (y,t)}^{V_\bA ^{-1}(y,t)} f_\bA (x,t) \, dx \right)\wedge 0; \end{align}
		\begin{align}\label{boundary}
			V_\bA (x,T) = h_\bA (x),\quad
			V_\bB (y,T) = h_\bB (y),\quad
			f_\bA (x,0) = f_{\bA ,0}(x),\quad
			f_\bB (y,0) = f_{\bB ,0}(y).\vspace{-40pt}
		\end{align} \normalsize 
		Consider the following mild assumptions on coefficient functions: 
		\begin{assumption}[\bf Running Costs]\label{condition_running} 
			For each $\bI \in \{\bA, \bB\}$, the running cost $r_{\bI}:\mathbb{R}_{\geq 0}\times[0,T] \to \mathbb{R}_{\geq 0}$ satisfies $r_{\bI}(0,t) = 0$ for all $t \in [0,T]$ and is bi-Lipschitz in the following sense: there exist positive constants $\gamma_\bI$ and $\Gamma_\bI$ such that for all $0\le z_2\leq z_1$ and $t\in[0,T]$, it holds that $
			\gamma_\bI  ( z_1-z_2) \le   r_{\hspace{1pt}\bI} (z_1,t)-r_{\hspace{1pt}\bI} (z_2,t)  \le \Gamma_\bI  ( z_1-z_2)$.
		\end{assumption}
		
		\begin{assumption}[\bf Terminal Costs]\label{condition_terminal} 
			For each $\bI \in \{\bA, \bB\}$, the terminal cost $h_\bI:\mathbb{R}_{\geq 0} \to \mathbb{R}_{\geq 0}$ satisfies $h_\bI(0) = 0$ and is bi-Lipschitz in the following sense: there exist positive constants $l_\bI$ and $L_\bI$ such that for all $0\le z_2\leq z_1$, it holds that $l_\bI  (z_1-z_2) \le   h_\bI (z_1)-h_\bI (z_2)  \le L_\bI  ( z_1-z_2)
			$.
		\end{assumption}
		\noindent These conditions serve technical purposes. In particular, Assumption~\ref{condition_initial} ensures compactness of the space for solving the FP equations, while Assumptions~\ref{condition_running}–\ref{condition_terminal} guarantee that the solutions to the HJB equations are strictly increasing, ensuring the existence of their inverses. For $\bI \in\{\bA,\bB\}$ with $\bJ\neq\bI$, we define the sets 
		\begingroup \begin{equation}\label{def.HI}
			\mathcal{H}_{\bI} :=		\left\{
			V \in C\big([0, T]; C(\mathbb{R}_{\geq 0}; \mathbb{R}_{\geq 0})\big) \;\middle|\;
			\begin{aligned}
				&k_{\bI} (z_1 - z_2) \leq V(z_1,t) - V(z_2,t) \leq K_{\bI} (z_1 - z_2),\,\, V(0,t)=0\\
				& \text{for all }t \in [0, T ] \text{ and } 0 \le z_2\le z_1; 
				\text{and $\sup_{t\in[0, T]}\sup_{z\ge 0}\frac{|V(z,t)|}{1+z} \le \Pi_{\bI}$}
			\end{aligned}
			\right\}
		\end{equation}\endgroup
		and 
		\begingroup \begin{equation}\label{def.AI}
			\scalebox{0.93}{$\mathcal{A}_{\bI} := \left\{
				f  \in C\big([0, T]; L^1(\mathbb{R}_{\geq 0};  \mathbb{R}_{\geq 0})\big) \,\middle|\,
				\displaystyle\sup_{t\in[0, T]}\int_0^{\infty}(1+z)f(z,t)dz \leq  \frac{2^{1+\nu}C_\bI}{\nu}
				\hspace{5pt} \text{and} 
				\sup\limits_{t\in [0, T]}\displaystyle\int^\infty_0f(z,t)dz\leq 1
				\right\}$},
		\end{equation}\endgroup
		where $
		k_\bI  := \frac{\gamma_\bI }{\rho+\lambda_\bJ } \wedge l_\bI$, $K_\bI=\left(\frac{\lambda_{\bJ}C_\bJ }{\rho k_\bJ }+ \frac{\Gamma_\bI }{\rho}\right) \vee L_\bI$ and $\Pi_\bI :=\left[\frac{\lambda_\bJ }{\rho}\left(\frac{2^{1+\nu}C_\bJ }{\nu} + 
		\Gamma_\bI\right)\right]\vee L_\bI.$ With the above definitions, we state the following result:
		\begin{theorem}[\bf Global Existence of the Fully Coupled System]\label{maintheorem}
			Suppose that Assumptions  \ref{condition_initial}-\ref{condition_terminal} hold, then the fully coupled HJB-FP system \eqref{HJBx}-\eqref{boundary} admits a solution in $ \mathcal{H}_{\bA}\times\mathcal{H}_{\bB}\times\mathcal{A}_{\bA}\times\mathcal{A}_{\bB}$.
		\end{theorem}
		The proof is presented in Appendix \ref{sec:existence}, with the main ideas outlined in Appendix \ref{sec. Main result and Sketch of Proof}. Although the proof is based on fixed-point arguments, it departs from standard settings due to the system’s two-layered coupling: a forward–backward structure and dynamic interactions between two distinct agent types. Additionally, the presence of non-local terms further complicates the analysis, necessitating the development of novel estimates to control the coupled dynamics.

		Due to the coupled forward–backward structure of the system, additional conditions on the coefficient functions are necessary to guarantee uniqueness. Note that these conditions are not required for existence. We present two sufficient conditions: (i) a bounded time horizon condition, under which $T$ cannot be too large relative to model parameters; and (ii) a structural condition on the coefficients that is independent of $T$.
		
		\begin{theorem}[\bf Uniqueness of the Fully Coupled System]\label{thm unique}
			Suppose that Assumptions \ref{condition_initial}-\ref{condition_terminal} hold. In addition, we assume that one of the following conditions is satisfied:
			\begin{enumerate}[label=\textup{\textbf{(C.\Roman*)}}, leftmargin=*]
				\item\label{condition.uniqueness.loc}
				\hspace*{\fill}
				$		M_3(\lambda_\bA  \vee \lambda_\bB )	
				e^{M_2 T}
				\left( \dfrac{1-e^{-\rho T}}{\rho} \right) 
				\left( \dfrac{e^{\left(\lambda_\bA M_\bA +\lambda_\bB M_\bB \right)T}-1 }{\lambda_\bA M_\bA +\lambda_\bB M_\bB} \right)<1$,
				\hspace*{\fill}
				
				\item\label{condition.uniqueness.global}
				\hspace*{\fill}
				$	2\pig[(\lambda_\bA  \vee \lambda_\bB )M_3                  \pigr]^{1/2}
				<\rho- M_2  -(\lambda_\bA M_\bA +\lambda_\bB M_\bB )$.
				\hspace*{\fill}
			\end{enumerate}
			\sloppy Then the system \eqref{HJBx}–\eqref{boundary} admits a unique solution in the space $\mathcal{H}_{\bA}\times\mathcal{H}_{\bB}\times\mathcal{A}_{\bA}\times\mathcal{A}_{\bB}$. Here $M_\bI := 2 \vee \frac{(1+C_\bI )(1+k_\bI )}{(k_\bI )^2}$, $M_2:=	\left(\lambda_\bA + \lambda_\bB \tfrac{2^{1+\nu} C_\bA}{\nu}\right)
			\vee
			\left(\lambda_\bB + \lambda_\bA \tfrac{2^{1+\nu} C_\bB}{\nu}\right)$ and 
			$
			M_3:=  \frac{2^{3+\nu}C_\bA C_\bB }{\nu}
			\left[\frac{\lambda_\bA+\lambda_\bB(1+k_\bA)}{k_\bA}
			\vee  \frac{\lambda_\bB+\lambda_\bA(1+k_\bB)}{k_\bB}\right]$ where $\bI \in\{\bA,\bB\}$.  
			
		\end{theorem}  \noindent The proof of this theorem can be found in Appendix \ref{sec. unique}.

		\sloppy In the following, we present the verification theorem for Problems \ref{pblm:type A, threshold} and \ref{def Nash}, whose proof can be found in Appendix \ref{sec. thm.verification}. 
		
		\begin{theorem}[\bf Verification Theorem of Control Problems]\label{thm.verification}
			Suppose that Assumptions \ref{condition_initial}-\ref{condition_terminal} hold. Let $(V_\bA^*,V_\bB^*,f_\bA^*,f_\bB^*) \in \mathcal{H}_{\bA}\times\mathcal{H}_{\bB}\times\mathcal{A}_{\bA}\times\mathcal{A}_{\bB}$ be the solution to the fully coupled HJB-FP system \eqref{HJBx}-\eqref{boundary}, obtained by \Cref{maintheorem}. For each $\bI,\bJ \in \{\bA,\bB\}$ with $\bI \neq \bJ$, we define the joint probability measure $\pi_{\bI,\bigcdot}^*$ on $X_\bI \times \{0,1\}$ by 
			\begingroup\begin{equation}
				\pi_{\bI,t}^*(\bigcdot \times \{0\})
				:= f_{\bI}^{*}(\bigcdot,t)\,d \mathcal{L}^1, 
				\hspace{15pt} \textup{and} \hspace{15pt}
				\pi_{\bI,t}^*(\bigcdot\times \{1\})
				:= \mu_{\bI,0}(\bigcdot)\;-f_{\bI}^{*}(\bigcdot,t)\,d \mathcal{L}^1.
				\label{3343} 
			\end{equation}\endgroup
			Then, the following statements are true:

			\begin{enumerate}[(1).]
				\item For Problem \ref{pblm:type A, threshold} of type-$\bI$ agents, suppose the threshold function $u_\bJ=V_\bJ^*$ and the joint distribution $\pi_{\bJ,\bigcdot}^*$ of quality levels and statuses  for type-$\bJ$ agents are given. Then the value function defined in \eqref{def.V_I.threshold} is given by $\mathcal{V}_\bI (x,t,V_\bJ^*,\pi_{\bJ,\bigcdot}^*)=V_\bI^*(x,t)$. Upon maximization of the objective functional, the resulting joint distribution of quality levels and statuses for type-$\bI$ agents is uniquely given by $\pi_{\bI,\bigcdot}^*$. Moreover, the optimal control (threshold function) is given by $u_\bI=V_\bI^*$;
				
				\item The quadruple $(V_\bA^*, V_\bB^*, \pi_{\bA,\bigcdot}^*, \pi_{\bB,\bigcdot}^*)$ constitutes the mean field Nash equilibrium solving Problem \ref{def Nash}. 
			\end{enumerate}

		\end{theorem}
		\begin{remark}
			The optimal control in statement (1) may not be unique. One of the choices is $u_\bI=V_\bI^*$. However, once the objective functional is maximized, the resulting joint distribution is uniquely determined. 
		\end{remark}

		\subsection{Properties of Matching Equilibrium and Interpretations }

		Let $(V_\bA^*, V_\bB^*, f_\bA^*, f_\bB^*)\in  \mathcal{H}_{\bA}\times\mathcal{H}_{\bB}\times\mathcal{A}_{\bA}\times\mathcal{A}_{\bB}$ be the solution to the fully coupled HJB-FP system \eqref{HJBx}–\eqref{boundary} obtained in \Cref{maintheorem}, from which we can construct the mean field Nash  equilibrium of Problem \ref{def Nash} by \Cref{thm.verification}. We describe qualitative features of agent behavior in the market under equilibrium, derived from primal analysis. These findings also admit natural economic interpretations consistent with real-world matching markets.

		First, for $\bI \in\{\bA,\bB\}$, we would like to highlight from our construction of the solution that the value function $V_\bI^* \in \mathcal{H}_{\bI}$, which also represents the agents' thresholds, is strictly increasing in the quality level $x$ (or $y$). It reflects the natural economic intuition that higher-quality agents derive higher value from participation in the market. This monotonicity is also consistent with the economic intuition: higher-quality agents demand higher-quality partners. They have stronger fallback options (outside options), so the opportunity cost of matching with a low-quality partner is higher. Moreover, the definition of $\mathcal{H}_{\bI}$ in \eqref{def.HI} implies that agents with zero quality level derive no value from participating in the market. That is, their expected utility is zero at all times, and they are effectively excluded from the matching process. 
		
		Second, in \Cref{prop.kk>1}, we find that when the running utility during the search period and the terminal reward are sufficiently high, agents have no incentive to match.

		\begin{proposition}[\bf Condition for No Match to Occur]\label{prop.kk>1}
			\sloppy	Suppose that Assumptions \ref{condition_initial}-\ref{condition_terminal} hold. For each $\bI,\bJ \in \{\bA,\bB\}$ with $\bI \neq \bJ$, if 
			\begingroup\begin{align}\label{ass. no match}
				\left[ e^{-\rho(T-t)} l_\bA 
				+ \left( 1-e^{-\rho(T-t)} \right) \dfrac{\gamma_\bA}{\rho}\right]\left[e^{-\rho(T-t)}l_\bB
				+ \left( 1-e^{-\rho(T-t)} \right)\dfrac{\gamma_\bB}{\rho} \right]  > 1,
			\end{align}\endgroup for all $t\in[0,T]$, then $V_\bI^* (x,t) > (V_\bJ^*) ^{-1}(x,t)$ for any $(x,t) \in \mathbb{R}_{> 0} \times [0,T]$. Hence the optimal matching region of the type-$\bI$ agent $(x,s)$ meeting a type-$\bJ$ agent, that is, $\mathcal{O}_{\bI}^*(x,s):=\big\{y\in X_\bJ:V_{\bI}^* (x,s) \leq y \leq (V_{\bJ}^*) ^{-1}(x,s)\big\}, $ is empty and no successful match occurs, for any $(x,t) \in \mathbb{R}_{> 0} \times [0,T]$.
		\end{proposition}        
		In \Cref{prop.kk>1}, the time-weighted average of the least reservation utility:
		\begingroup	\begin{align}\label{def time-weighted average of reservation utility}
			e^{-\rho(T-t)}l_\bI 
			+ \left( 1-e^{-\rho(T-t)}\right) \dfrac{\gamma_\bI}{\rho}
		\end{align}\endgroup
		represents the lower bound of the discounted average of terminal and running utilities, or equivalently, the expected payoff if unmatched. If \eqref{ass. no match} holds, then the lump sum reward from matching is less attractive than the combined benefits \eqref{def time-weighted average of reservation utility} of staying unmatched, thus agents in both types may rationally prefer to remain unmatched. More importantly, for $\bI,\bJ \in \{\bA,\bB\}$ with $\bI \neq \bJ$, if type-$\bI$ agents receive a relatively low time-weighted average of reservation utility, type-$\bJ$ agents must receive a relatively high one in order for no match to occur at all. This compensates for not gaining the lump sum reward. In other words, if one group has a strong incentive to match due to poor outside options, the absence of matching must come from the other side having particularly attractive alternatives, underscoring the mutual, reciprocal nature of the matching dynamics.

		Referring to the labor market context in \Cref{sec. applications}, suppose job seekers face a low time-weighted average of reservation utility. Under normal circumstances, this would push them to accept job offers quickly with lower threshold functions. However, if hiring firms experience high time-weighted average of reservation utility from leaving positions temporarily unfilled, then these firms may raise their hiring thresholds and delay matches. As a result, no hiring occurs. Thus, equilibrium with no matching can arise from a mismatch in incentives, where one side’s willingness is offset by the other’s disinterest. 
		
		If the assumptions of \Cref{prop.kk>1} hold, then the integral terms in the HJB equations \eqref{HJBx}–\eqref{HJBy} vanish. In this case, on $\mathbb{R}_{\ge 0} \times [0,T]$, we have $
		V_\bI^* (x,t) = e^{-\rho(T-t)}h_\bI (x) + 
		\int^T_t e^{-\rho(s-t)} r_\bI( x,s)ds.$
		
		\begin{proof}[Proof of \Cref{prop.kk>1}]
			We prove by contradiction. Suppose there exists $(x,t) \in \mathbb{R}_{\ge 0} \times [0,T]$ such that $V_\bA^* (x,t) \leq (V_\bB^*)^{-1}(x,t)$. By the optimality of value function, Assumptions \ref{condition_running} and \ref{condition_terminal}, we have $
			V_\bI^* (x,t) 
			\geq e^{-\rho(T-t)}h_\bI (x) + \int^T_t e^{-\rho(s-t)} r_\bI( x,s)ds 
			\geq \left[e^{-\rho(T-t)}l_\bI 
			+ \left( 1-e^{-\rho(T-t)}\right) \dfrac{\gamma_\bI}{\rho}\right]x,$
			for $\bI\in\{\bA ,\bB \}$. Hence, for example, we have 
			\begingroup
			$$
			\left[e^{-\rho(T-t)}l_\bA 
			+ \left( 1-e^{-\rho(T-t)}\right) \dfrac{\gamma_\bA}{\rho}\right] x \leq V_\bA^* (x,t)\leq (V_\bB^*) ^{-1}(x,t) 
			\leq 	\left[e^{-\rho(T-t)}l_\bB 
			+ \left( 1-e^{-\rho(T-t)}\right) \dfrac{\gamma_\bB}{\rho}\right]^{-1}x. 
			$$\endgroup
			This contradicts \eqref{ass. no match}. Therefore, the proposition is proved as a similar result for $\bI=\bB$ also holds. 
		\end{proof}

		Third, in contrast to the degenerate case in \Cref{prop.kk>1}, \Cref{prop matching happens} provides a sufficient condition for the existence of non-empty matching regions: it requires that agents' thresholds are not excessively high.

		\begin{proposition}[\bf Condition for Nonempty Matching Region]\label{prop matching happens}
			Suppose that Assumptions \ref{condition_initial}-\ref{condition_terminal} hold. Recall the constants defined below \eqref{def.AI}. For each $\bI,\bJ \in \{\bA,\bB\}$ with $\bI \neq \bJ$, if $K_\bA K_\bB  \leq 1$, then $V_\bI^* (x,t) \leq (V_\bJ^*) ^{-1}(x,t)$ for any $(x,t) \in \mathbb{R}_{\ge 0} \times [0,T]$. Hence, the optimal matching region of the type-$\bI$ agent $(x,t)$ meeting a type-$\bJ$ agent, that is, $\mathcal{O}_{\bI}^*(x,t):=\big\{y\in X_\bJ:V_{\bI}^* (x,t) \leq y \leq (V_{\bJ}^*) ^{-1}(x,t)\big\}, $ is not empty, for any $(x,t) \in \mathbb{R}_{\ge 0} \times [0,T]$.
		\end{proposition}	
		\begin{proof} 
			We suppose that there is $(x,t) \in \mathbb{R}_{\ge 0} \times [0,T]$ such that $V_\bA^* (x,t) > (V_\bB^*) ^{-1}(x,t)$. Lemma \ref{BilipschitzVA} implies that $
			\frac{1}{K_\bB } x \leq  (V_\bB^*) ^{-1}(x,t) < V_\bA^* (x,t) \leq K_\bA x. $
			This contradicts $K_\bA K_\bB  \leq 1$. Thus, the proof is completed.
		\end{proof}
		
		In the proof of \Cref{prop matching happens}, we have used $V_\bI^* (x,t) \leq K_\bI x$, where the constant $K_\bI$ represents the upper bound of the reservation value per unit of quality of type-$\bI$ agents. It also corresponds to the optimal threshold per unit of quality due to \Cref{thm.verification}. Roughly speaking, the condition $K_\bA K_\bB \leq 1$ means that when agents’ threshold functions per unit of quality are not excessively high, matches can occur.

		For example, in the case of the labor market, the constants $K_\bA$ and $K_\bB$ represent the \emph{selectivity indices} (threshold per unit of quality) for job seekers (type-$\bA$) and firms (type-$\bB$), respectively. A higher value of $K_\bI$ indicates that type-$\bI$ agents are more selective and impose more stringent acceptance criteria. To examine how parameters influence the values of $K_\bA$ and $K_\bB$, consider the case where quality levels are normalized ($x=y=1$). For instance, $K_\bA=\left(\frac{\lambda_{\bB}C_\bB }{\rho k_\bB }+\frac{\Gamma_\bA }{\rho}\right) \vee L_\bA$ depends on three key elements:
		\begin{enumerate}[(1).]
			\item the best-scenario expected gain (running utility) when waiting for a match, measured by $\frac{\Gamma_\bA}{\rho}$, i.e., highest possible income $\Gamma_\bA$ during job searching (e.g. income from part-time
			work or job-seeking support) discounted by the impatience level $\rho$;
			\item the best-scenario fallback option (terminal utility) if unmatched, measured by $L_\bA$ (e.g., self-employment, unemployment insurance);  
			\item the \emph{leverage ratio} of job seekers from market conditions, measured by $\frac{\lambda_\bB C_\bB}{\rho k_\bB}$. A higher value of this ratio implies that the job seekers expect high matching payoffs in the remaining time of the job market.
			
			\begin{enumerate}[(a).]
				\item To understand this, we unpack the expression $k_\bB  = \frac{\gamma_\bB }{\rho+\lambda_\bA } \wedge l_\bB $ as follows: (I). $\frac{\gamma_\bB}{\rho + \lambda_\bA}$ represents the present value of lowest vacancy utility (leaving positions temporarily unfilled, e.g., internal labor reallocation, delayed onboarding, or short-term capital
				gains), discounted by the impatience level $\rho$ and adjusted by the arrival frequency of job seekers $\lambda_\bA$
				; and (II). $l_\bB$ measures the gain from expected outcomes in the subsequent period (participating in the next job market or starting a business). The expression $k_\bB = \frac{\gamma_\bB}{\rho + \lambda_\bA} \wedge l_\bB$ can be interpreted as a measure of the differentiation level of the value functions (or hiring thresholds) across firms. A higher value of $\lambda_\bA$ implies a greater meeting rate. Due to the presence of overlap in the matching region (referring to the discussion in \Cref{mfr.fig:gA_x}), lower-quality firms gain increased opportunities to match with job seekers whose ability exceeds that of those hired by higher-quality firms. As a result, the degree of differentiation among firms of different tiers diminishes.
				
				From the job seeker’s perspective, a reduction in the differentiation power of hiring thresholds across firm tiers implies that a job seeker has a greater chance of securing a position that exceeds their intrinsic ability level. This, in turn, raises their expectation of finding a satisfactory job in the remaining time of the job market.

				\item  Moreover, $\lambda_\bB$ reflects the amount of job positions in the market, or equivalently, the arrival frequency of job opportunities for job seekers. The constant $C_\bB$ captures the concentration of job quality, where a high $C_\bB$ indicates that a large proportion of jobs possess similar quality levels, leading to intense competition among firms offering positions at similar levels. 
			\end{enumerate}
			The leverage ratio $\frac{\lambda_\bB C_\bB}{\rho k_\bB}$ measures the abundance and concentration of job opportunities relative to firms' hiring thresholds. A higher value of this ratio implies that the job seekers have a higher likelihood of being employed and higher matching payoffs in the remaining time of the labor market.
		\end{enumerate}
		A similar interpretation holds for $K_\bB$.

		Therefore, the condition $K_\bA K_\bB \leq 1$ serves as a safeguard against market stagnation, ensuring that the job market remains fluid and active. Economically, it guarantees that the selectivity indices of job seekers and firms are sufficiently moderate so that their matching regions are not empty.
		
		Crucially, the condition $K_\bA K_\bB \leq 1$ functions as a market balancing mechanism. If one side becomes more selective, say, job seekers raise their threshold and $K_\bA$ increases, then the condition can still be satisfied if firms become relatively less selective ($K_\bB$ decreases), and vice versa. This is a mutual and reciprocal constraint: each side’s selectivity directly influences, and is influenced by, the other.

		For any $\bI,\bJ \in \{\bA,\bB\}$ with $\bI \neq \bJ$, the following proposition derives the probability and expectation of the quality level of the type-$\bJ$ agent matched with the representative type-$\bI$ agent $(x,0)$, conditional on the type-$\bI$ agent $(x,0)$ got matched before closing time $T$.

		\begin{proposition}

			Let $\bI, \bJ \in \{\bA, \bB\}$ with $\bI \neq \bJ$. For any $x\in X_{\bI}$ such that $f_{\bI,0}(x)>0$ and $\mathbb{P}_{\bI,x,0}(\tau_\bI(x,0) \leq T)>0$, and any Borel set $B \in \mathcal{B}(X_{\bJ})$, we have
			\begingroup
			\begin{align}\label{prop.Qtau in B}
				&\nonumber\mathbb{P}_{\bI,x,0}\left(\mathbf{QL}_{\bJ,\tau_\bI(x,0)} \in B \,\pig|\, \tau_\bI(x,0) \leq T\right) \\
				&=\lambda_\bJ \int_B \frac{\int_0^T  \mathbb{P}_{\bI,x,0}(\tau_\bI(x,0) >s)\mathbb{P}_{\bJ,y,0}(\tau_\bJ(y,0) >s) \mathbf{1}_{\{V_{\bI}^* (x,s) \leq y \leq (V_{\bJ}^*) ^{-1}(x,s)\}} ds}{\mathbb{P}_{\bI,x,0}(\tau_\bI(x,0) \leq T)} d\mu_{\bJ,0}(y),\hspace{15pt} \text{and}
			\end{align}
			\begin{equation}\label{prop.exp.Qtau in B}
				\hspace{-5pt}\mathbb{E}_{\bI,x,0}\!\left(\mathbf{QL}_{\bJ,\tau_\bI(x,0)} \pig| \tau_\bI(x,0) \leq T\!\right)
				\!=  \!\lambda_\bJ\frac{\int_0^T \! \mathbb{P}_{\bI,x,0}(\tau_\bI(x,0) >s) \!
					\int_{V_{\bI}^* (x,s)}^{(V_{\bJ}^*) ^{-1}(x,s)}\! y \mathbb{P}_{\bJ,y,0}(\tau_\bJ(y,0) \!>\!s)  d\mu_{\bJ,0}(y)ds}{\mathbb{P}_{\bI,x,0}(\tau_\bI(x,0) \leq T)},\hspace{-10pt}
			\end{equation}\endgroup
			where
			\begingroup \begin{equation}\label{prop.taux0>s}
				\hspace{-10pt}	\mathbb{P}_{\bI,x,0}(\tau_\bI(x,0) >r) = \exp\left[-\lambda_\bJ \int_0^r  \mathbf{1}_{\{V_\bI^* (x,s) \leq (V_\bJ^*) ^{-1}(x,s)\}} \int_{V_\bI^* (x,s)}^{(V_\bJ^* )^{-1}(x,s)} \mathbb{P}_{\bJ,y,0}(\tau_\bJ(y,0) >s) d\mu_{\bJ,0} (y) ds\right].\hspace{-10pt}
			\end{equation}  \endgroup
		\end{proposition}
		\begin{proof}
			Without loss of generality, take $\mathbf I = \mathbf A$. We assume $f_{\mathbf A,0}(x)>0$ to exclude quality levels absent from the market. Then $\mathbb{P}_{\mathbf A,x,0}\bigl(\tau_{\mathbf A}(x,0)\le T\bigr)>0$ implies that the set $\bigl\{(y,t)\in X_{\bB} \times [0,T]: V_{\mathbf A}^*(x,t)\le y \le (V_{\mathbf B}^*)^{-1}(x,t)\bigr\}$ has positive measure due to Lemma \ref{lem eq. of P1 P2}. Let $\eta_{\bA,x,0}(\bigcdot,\bigcdot)$ denote the joint distribution of $\mathbf{QL}_{\bB,\tau_\bA(x,0)}$ and the matching time $\tau_\bA(x,0)$ of type-$\bA$ agent $(x,0)$. For any bounded measurable test functions $\phi(y,t)$ and $\Phi(y,t):=\phi(y,t)\mathbf{1}_{\{t \leq T\}}$, we have
			\begingroup \begin{align}\label{E phi eta.1}
				\mathbb{E}_{\bA,x,0}\pig[\Phi(\mathbf{QL}_{\bB,\tau_\bA(x,0)},\tau_\bA(x,0))\pig]
				= \int_0^T \int_0^\infty \phi(y,t) \eta_{\bA,x,0}(dy,dt). 
			\end{align}\endgroup
			Since $\mathbb{P}\pig(V_\bA^*(x,t) \leq \mathbf{QL}_{\bB,t} \leq (V_\bB^*) ^{-1}(x,t),\,
			\mathbf{ST}_{\bB,t} =0 \,\big|\, \tau_\bA(x,0) = t \pig) = 1$, we have
			\begingroup\begin{align*}
			&\mathbb{E}_{\bA,x,0}\pig[\Phi(\mathbf{QL}_{\bB,t},t) \mathbf{1}_{\{V_\bA^*(x,t) \leq \mathbf{QL}_{\bB,t} \leq (V_\bB^*) ^{-1}(x,t),\,
					\mathbf{ST}_{\bB,t} =0\}}\,\big|\,\tau_\bA(x,0) = t\pig] \\
				& = \mathbb{E}_{\bA,x,0}\pig[\Phi(\mathbf{QL}_{\bB,t},t)\,\big|\,\tau_\bA(x,0) = t, \, V_\bA^*(x,t) \leq \mathbf{QL}_{\bB,t} \leq (V_\bB^*) ^{-1}(x,t),\,
				\mathbf{ST}_{\bB,t} =0\pig] \\
				& = \int_{V_\bA^* (x,t)}^{(V_\bB^*) ^{-1}(x,t)}  \frac{\Phi(y,t)}{\mu_{\bB,t}\pig(\big[V_\bA^* (x,t), (V_\bB^*) ^{-1}(x,t)\big]\pig)}
				d\mu_{\bB,t}(y),
			\end{align*}\endgroup
			where the last equality follows from the fact that $\{\tau_\bA(x,0) = t\}$ is conditionally independent of $\mathbf{QL}_{\bB,t}$ given $\{ V_\bA^*(x,t) \leq \mathbf{QL}_{\bB,t} \leq (V_\bB^*) ^{-1}(x,t),\,
			\mathbf{ST}_{\bB,t} =0{\color{black}, \text{$\exists \, \tau^\bB \in \mathcal{T}^\bB$ s.t. $\tau^\bB = t$}} \}$.
			Hence, by the continuity of $\mu_{\bB,\bigcdot}$, $V_\bA^*$ and $V_\bB^*$ and \eqref{eq.Ptau}, we obtain
			\begingroup \begin{align*}
				&\hspace{-10pt}\mathbb{E}_{\bA,x,0}\pig[\Phi(\mathbf{QL}_{\bB,\tau_\bA(x,0)},\tau_\bA(x,0))\pig] \\
				=&\,  \mathbb{E}_{\bA,x,0}\pig[\Phi(\mathbf{QL}_{\bB,\tau_\bA(x,0)},\tau_\bA(x,0)) \mathbf{1}_{\{V_\bA^*(x,\tau_\bA(x,0)) \leq \mathbf{QL}_{\bB,\tau_\bA(x,0)} \leq (V_\bB^*) ^{-1}(x,\tau_\bA(x,0)),\,
					\mathbf{ST}_{\bB,\tau_\bA(x,0)} =0\}}\pig] \\ 
				=&\,  \int_0^T \int_{V_\bA^* (x,t)}^{(V_\bB^*) ^{-1}(x,t)} 
				\frac{ \phi(y,t)}{\mu_{\bB,t}\pig(\big[V_\bA^* (x,t), (V_\bB^*) ^{-1}(x,t)\big]\pig)}
				\left(-\frac{\partial \mathbb{P}_{\bA,x,0}(\tau_\bA(x,0) >t)}{\partial t}\right) d\mu_{\bB,t}(y)dt\\
				=&\,  \int_0^T \int_0^\infty \phi(y,t) \mathbb{P}_{\bA,x,0}(\tau_\bA(x,0) >t)\lambda_\bB \mathbf{1}_{\{V_{\bA}^* (x,t) \leq y \leq (V_{\bB}^*) ^{-1}(x,t)\}} d\mu_{\bB,t}(y) dt.
			\end{align*}\endgroup
			By \Cref{prop.fxt3}, it follows that
			\begingroup \begin{align}\label{E phi eta.2}
				\nonumber 	&\hspace{-10pt}\mathbb{E}_{\bA,x,0}\pig[\Phi(\mathbf{QL}_{\bB,\tau_\bA(x,0)},\tau_\bA(x,0))\pig] \\
				=&\,  \int_0^T \int_0^\infty \phi(y,t) \mathbb{P}_{\bA,x,0}(\tau_\bA(x,0) >t)
				\lambda_\bB \mathbf{1}_{\{V_{\bA}^* (x,t) \leq y \leq (V_{\bB}^*) ^{-1}(x,t)\}} 
				\mathbb{P}_{\bB,y,0}(\tau_\bB(y,0) >t)
				d\mu_{\bB,0}(y) dt.
			\end{align}\endgroup
			Equating \eqref{E phi eta.1} and \eqref{E phi eta.2}, we have $
			\eta_{\bA,x,0}(dy,dt) 
			= \mathbb{P}_{\bA,x,0}(\tau_\bA(x,0) >t)
			\lambda_\bB 
			\mathbf{1}_{\{V_{\bA}^* (x,t) \leq y \leq (V_{\bB}^*) ^{-1}(x,t)\}} 
			\mathbb{P}_{\bB,y,0}(\tau_\bB(y,0) >t)
			d\mu_{\bB,0}(y) dt$ as $\phi$ is arbitrary. We can check that, again by \eqref{eq.Ptau}, $
			\mathbb{P}_{\bA,x,0}(\tau_\bA(x,0) \leq T) = \int_0^T \int_0^\infty \eta_{\bA,x,0}(dy,dt).$ Therefore, the distribution of $\mathbf{QL}_{\bB,\tau_\bA(x,0)}$ under the condition that $\tau_\bA(x,0) \leq T$ is given by
			\begingroup\begin{equation*}
				\hspace{-5pt}\frac{\int_0^T \eta_{\bA,x,0}(dy,dt)}{\int_0^T \!\!\int_0^\infty\!\! \eta_{\bA,x,0}(dy,dt)} 
				\!=\! \frac{\int_0^T\! \mathbb{P}_{\bA,x,0}(\tau_\bA(x,0) >s)\mathbb{P}_{\bB,y,0}(\tau_\bB(y,0) >s)\lambda_\bB 
					\mathbf{1}_{\{V_{\bA}^* (x,s) \leq y \leq (V_{\bB}^*) ^{-1}(x,s)\}}ds}{\mathbb{P}_{\bA,x,0}(\tau_\bA(x,0) \leq T)}  d\mu_{\bB,0}(y)\hspace{-10pt}
			\end{equation*}\endgroup
			which concludes \eqref{prop.Qtau in B} and \eqref{prop.exp.Qtau in B}. Moreover, by \eqref{eq.Ptau} again, we have $
			\mathbb{P}_{\bA,x,0}(\tau_\bA(x,0) >r) = \exp\pig[-\lambda_\bB \int_0^r  \mathbf{1}_{\{V_\bA^* (x,s) \leq (V_\bB^*) ^{-1}(x,s)\}} \mu_{\bB,s} ([V_\bA^* (x,s),(V_\bB^* )^{-1}(x,s)]) ds\pig]$ for $r\in [0,T]$. By \Cref{prop.fxt3}, we obtain \eqref{prop.taux0>s}.
		\end{proof}

		By \Cref{prop.fxt3} and \Cref{condition_initial}, we have $f_\bI(z,t) := \mathbb{P}_{\bI,z,0}(\tau_{\hspace{0.5pt}\bI}(z,0) > t) f_{\bI ,0}(z)$. Thus, the density of the probability distribution \eqref{prop.Qtau in B} is
		\begingroup \begin{align}\label{prop.Qtau in B.2}
			g_\bI(x,y) := \lambda_\bJ \frac{\int_0^T  f_\bI(x,s)f_\bJ(y,s) \mathbf{1}_{\{V_{\bI}^* (x,s) \leq y \leq (V_{\bJ}^*) ^{-1}(x,s)\}} ds}{f_{\bI ,0}(x)-f_\bI(x,T)}.
		\end{align}\endgroup
		We will show some numerical simulations for the quantities in \eqref{prop.Qtau in B.2} in the next section.
		
		\section{Numerical Results of Matching Model}\label{sec. Numerical Results of the Matching Model}
		This section presents numerical experiments for solutions to the fully coupled HJB–FP system \eqref{HJBx}–\eqref{boundary}. We also provide economic interpretations of the results in the context of the labor market introduced in \Cref{sec. applications}, illustrating how the model captures and explains observed matching behaviors. These findings demonstrate the model's applicability and its ability to accurately predict empirical phenomena.

		\subsection{Numerical Parameters}\label{sec. Numerical Parameters}
		We apply our model to the frictionless labor market and present numerical simulations. A detailed discussion of the model’s application and the economic interpretation of its parameters is provided in \Cref{sec. applications}. The market consists of job seekers (type-$\bA$) and hiring managers of firms (type-$\bB$). Below, we summarize the parameter configuration used in the simulations:
		\begin{enumerate}[(1).]
			\item The time horizon is set to $T=1$ year;
			
			\item The intensities of the Poisson processes are $\lambda_\bA =20$ and $\lambda_\bB =26$, representing the meeting rates of hiring managers and job seekers, respectively;
			
			\item The discount rate is $\rho=0.04$. The $30$-year continuous-time annuity factor is $\frac{1-e^{-30 \rho}}{\rho} \approx 17.47$;
			
			\item By the human capital approach, we identify the quality levels of job seekers with the present value of $30$ years of market wages (in thousands of USD). We take $[0,x_\bA]=[0,7000]$, so the lump sum payoff of a match corresponds to the present value of $30$ years of earnings. We posit that the profit a job seeker generates for a firm is positively correlated with their income. Given the practical difficulty in obtaining direct data on job seekers' contribution to corporate profits, we use their income as a proxy measure. Furthermore, this approach ensures that the measures of $x$ and $y$ are on a comparable scale;
			
			\item Firms’ job offers provide a lump sum payoff of 30 years' salaries in the range $[\mu,x_\bB]=[\mu, 7000]$ thousand USD, where $\mu > 0$ is a parameter determined later (e.g., a weighted average of minimum wages across U.S. states);
			
			\item The running utility of job seekers is $r_\bA(x,t) =0.23(x/17.47)= 0.013x$, where $ 0.23$ is the net replacement rate in unemployment \cite{oecd2025} in the U.S. So $r_\bA$ reflects the flow value of unemployment benefits relative to the present value of lifetime earnings. For firms, we take $r_\bB(y,t)=0.05y$, capturing the return on idle capital associated with unfilled vacancies (slightly above the discount rate);

			\item The terminal utilities are $h_\bA (x)=0.6x$ and $h_\bB (y)=1.1y$, respectively. We can justify the choice of the terminal functions based on the following rationale. For a job seeker who fails to secure employment within the current annual labor market cycle and proceeds to the next cycle, their competitiveness in the labor market tends to diminish due to factors such as increasing age. Consequently, the expected value of the job seeker who participates in the next period of match generally declines. This value can be represented by $h_\bA (x)$. Conversely, from the firm's perspective, technological advancements are expected to enhance the productivity of workers, enabling them to generate greater profits for the company, leading to a corresponding increase in the expected value captured by $h_\bB (y)$.
		\end{enumerate}
		
		To model the initial distribution of quality levels of job seekers, we adopt the right-handed Pareto log-normal distribution, widely used to represent income distributions (see \cite{hosseini2019retirement,reed2004double}). The corresponding probability density is: 
		\begingroup \begin{equation}\label{pdf.pareto_lognormal}
			f_{\mathbf{A},0}(x) = 
			\alpha x^{-\alpha-1} 
			\exp\!\left(\alpha \nu + \frac{\alpha^2 \tau^2}{2}\right) 
			\Phi\!\left(\frac{\log(x)-\nu-\alpha \tau^2}{\tau}\right)
			\mathbf{1}_{\{x>0\}},
		\end{equation}\endgroup
		where $\alpha$, $\nu$, $\tau$ are some real parameters and $\Phi$ denotes the standard normal cumulative distribution function.

		An alternative choice to model the distribution of incomes is the generalized Pareto distribution:
		\begingroup \begin{equation}\label{pdf.generalized_pareto}
			f_{\mathbf{B},0}(y) =
			\frac{1}{\sigma}
			\left[ 1 + \frac{1}{\beta} \left( \frac{y-\mu}{\sigma} \right) \right]^{-\beta-1}
			\mathbf{1}_{\{y\geq \mu\}} .
		\end{equation}\endgroup

		In the presence of a government-mandated minimum wage, the generalized Pareto distribution provides a more realistic representation of the income distribution. We assume that the salaries offered by employers in the U.S. follow this distribution. The starting point is the published quantile values of weekly earnings for full-time wage and salary workers in the third quarter of 2024 (see \cite[Table 5]{USBureau2024Table5}). These raw figures correspond to the 10th, 25th, 50th, 75th, and 90th percentiles of weekly earnings. To compute the present value of total earnings over a 30-year working horizon, we assume that nominal wage growth exactly offsets inflation, so real salaries remain constant over time. Total earnings are then obtained by converting the weekly earnings to annual earnings and then multiplying them by the present value annuity factor for a 30-year horizon:
		
		\begin{table}[htbp]
\centering
\caption{Present value of 30-year earnings for full-time wage and salary workers, derived from 2024 Q3 weekly earnings quantiles in the U.S.}\medskip
\label{tab:quantiles}

\begin{tabular}{@{}l@{\quad}ccccc@{}} 
\hline\smallskip
Probability & 0.10 & 0.25 & 0.50 & 0.75 & 0.90 \\
\hline\smallskip
Quantile (USD in thousands)
& $\le 551.43$ & $\le 717.67$ & $\le 1058.34$ & $\le 1687.90$ & $\le 2627.23$ \\
\hline 
\end{tabular}

\vspace{2pt}
{\footnotesize Source: \cite[Table~5]{USBureau2024Table5}.}
\end{table}

		To calibrate the distributions, we match model-implied quantiles to the five empirical quantiles in Table~\ref{tab:quantiles}. Parameters for the Pareto–lognormal and generalized Pareto specifications are chosen by minimizing the relative root mean squared error (RRMSE; see \cite{jadon2024comprehensive}) between model and data across these five points, i.e., the RMSE of quantile gaps normalized by the scale of the model quantiles. The calibration yields the following estimates. For the Pareto--lognormal model~\eqref{pdf.pareto_lognormal}, the tail parameter is $\alpha=1.8644$, the location parameter is $\nu=6.5492$, and the variance parameter is $\tau=0.44209$. For the generalized Pareto model~\eqref{pdf.generalized_pareto}, the tail parameter is $\beta=8.6348$, the location parameter is $\mu=459.4388$, and the scale parameter is $\sigma=835.2216$.
		
		The RRMSEs for the fitted distributions $f_{\bA,0}$ and $f_{\bB,0}$ are $0.510\%$ and $0.326\%$ respectively. The fitted location parameter for the generalized Pareto distribution \eqref{pdf.generalized_pareto} is $\mu = 459.44$. This corresponds to an hourly minimum wage of USD $\frac{459.44 \times 1000}{17.47 \times 52 \times 40} \approx 12.64$, where the denominator converts the annual amount to an hourly rate ($17.47$ is the annuity factor, $52$ is the number of weeks per year, and $40$ is the standard weekly working hours). This value is broadly consistent with the current U.S. data, where state minimum wages typically range from USD 7 to 16 per hour.

		\subsection{Interpretation of Results on Labor Market Dynamics}\label{sec. Interpretation of Results}
		In this section, we approximate the fully coupled HJB–FP system \eqref{HJBx}–\eqref{boundary} on a bounded domain via a fixed-point iteration with a standard finite-difference discretization, using the parameters in \Cref{sec. Numerical Parameters}. The scheme ran for over $600$ iterations until the total relative error of value functions and density functions between successive iterates fell below $10^{-4}$. 
		
		\Cref{mfr.fig:VAVB_xy} shows the value functions (which also represent the agents' optimal thresholds) $V_\bA$ and $V_\bB$ respectively, plotted against their quality levels $x$ and $y$ at selected times. The monotonicity in $x$ (or $y$) is consistent with our theoretical results.  
		
		However, the relationship between $V_\bA$ (or $V_\bB$) and time $t$ exhibits distinct patterns across agent types. For job seekers (see the left panel of \Cref{mfr.fig:VAVB_xy}), the value functions (optimal thresholds) $V_\bA$ for low-ability job seekers show negligible variation over time. This is because their probability of securing a job is very low, causing their value functions to be predominantly determined by the running utility. In contrast, job seekers of medium to high ability exhibit a discernible declining trend in their value functions (and optimal thresholds) as time progresses. This reflects the effect of time discounting and reduced future opportunities. The longer one waits, the fewer opportunities remain, so the thresholds for accepting a match drop accordingly. These findings are consistent with those in \cite[Section II]{mccall1970economics}.  
		
		For hiring managers in firms (see the right panel of \Cref{mfr.fig:VAVB_xy}), the value functions (and optimal thresholds) for low to mid-level firms also demonstrate a decreasing trend over time, which is similarly driven by time discounting and the contraction of future opportunities. However, high-level firms exhibit a notably different pattern: their value function $V_\bB$ remains largely constant over time. This occurs because these high-tier firms anticipate the ability to recruit superior job seekers in the next period's labor market, and the prestige or acceptance criteria associated with high-level positions do not diminish with the passage of time within a single market cycle.

		\begin{figure}[h]\centering
			\begin{minipage}[t]{0.4\linewidth}
				\centering
				\includegraphics[width=\linewidth]{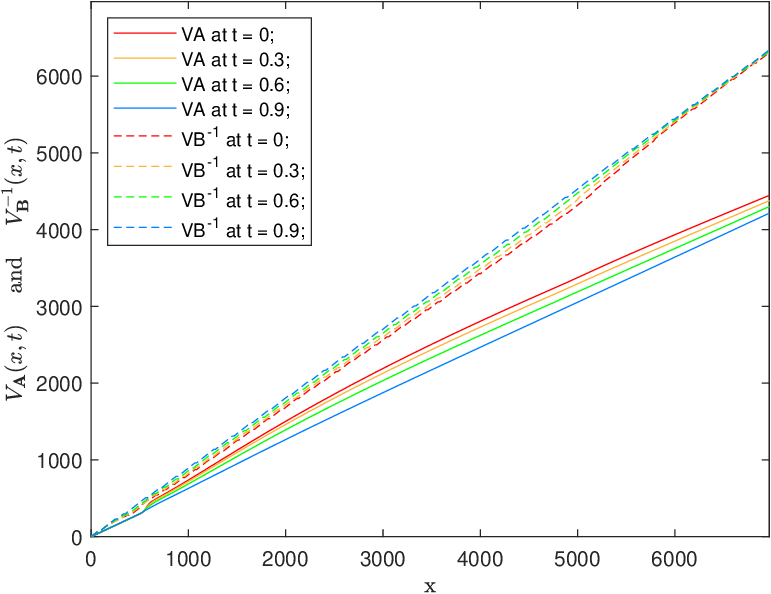}
			\end{minipage}
			\begin{minipage}[t]{0.4\linewidth}
				\centering
				\includegraphics[width=\linewidth]{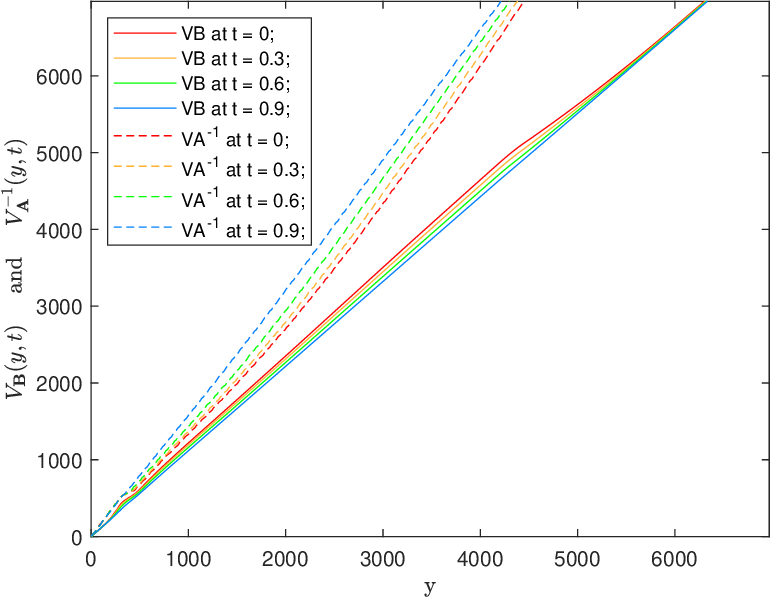}
			\end{minipage}
			\caption{Value functions (optimal thresholds) $V_\bA $, $V_\bB $ and their inverses against quality levels $x$ and $y$, for selected time $t$}
			\label{mfr.fig:VAVB_xy}
		\end{figure}
		
		\Cref{mfr.fig:fAfB_xy} indicates that the remaining unmatched agents become increasingly concentrated in the lower-quality region, particularly as $t \to 1^-$. This is because high-quality agents tend to match earlier due to their higher attractiveness and greater chances of satisfying both sides' acceptance thresholds. Consequently, the remaining pool of unmatched agents becomes increasingly skewed toward lower quality. We also note that the densities have higher values for small $x$ or $y$. Low-quality agents struggle to meet acceptance thresholds. They are more likely to remain unmatched until the very end of the market period. Lower-quality agents face a higher risk of exclusion.
		
		\begin{figure}[h]\centering
			\begin{minipage}[t]{0.4\linewidth}
				\centering
				\includegraphics[width=\linewidth]{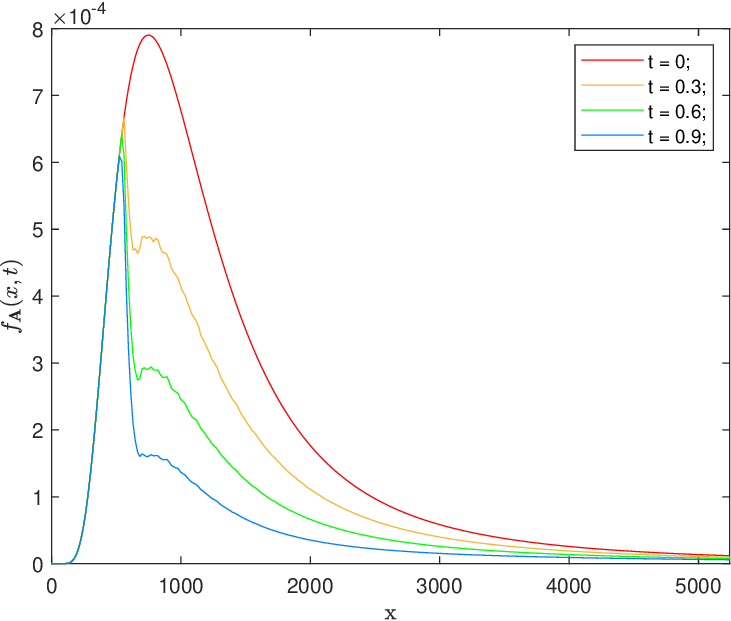}
			\end{minipage}
			\begin{minipage}[t]{0.4\linewidth}
				\centering
				\includegraphics[width=\linewidth]{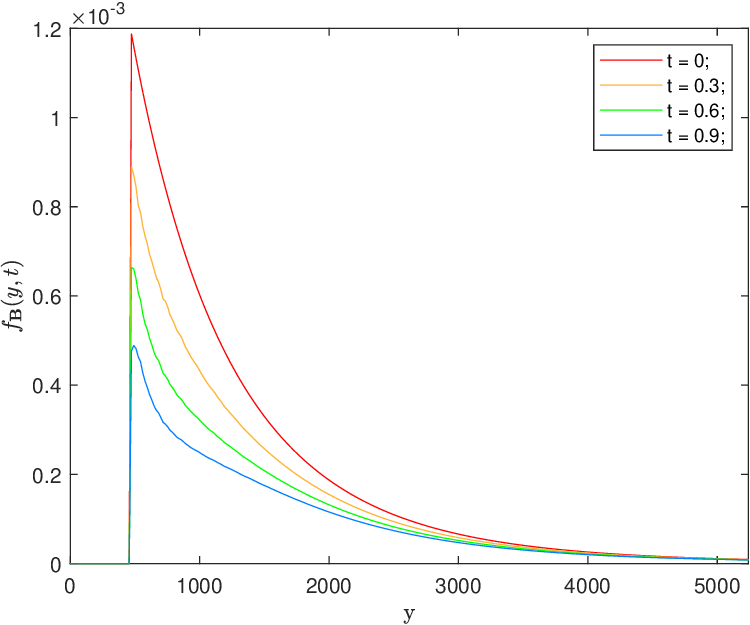}
			\end{minipage}
			\caption{Defective probability densities $f_\bA $ and $f_\bB $ against quality levels $x$ and $y$, for selected times}
			\label{mfr.fig:fAfB_xy}
		\end{figure}

		Figures \ref{mfr.fig:FAFB_t} and \ref{mfr.fig:fAfB-fAfB} illustrate the dynamics of the unmatched rate with respect to agent quality levels $x$ and $y$, as well as time $t$. For job seekers, the left panel of Figure \ref{mfr.fig:fAfB-fAfB} highlights a non-monotonic relationship between quality and matching success. Mid-tier job seekers tend to match more quickly, achieving the highest employment rates. In contrast, both low- and high-quality job seekers experience delayed matching due to distinct frictions: lower-quality job seekers face limited appeal to employers, while higher-quality job seekers are hindered by their greater selectiveness and higher reservation thresholds. These mismatches in expectations contribute to persistent underemployment or unemployment among job seekers at both ends of the ability distribution. Job seekers who fail to match within the current period enter the following year’s labor market, and a subset of these individuals may turn to entrepreneurship. This observation aligns with empirical findings by \cite{poschke2013becomes}, who showed that individuals at both extremes of the ability spectrum are more likely to become entrepreneurs.
		
		The pattern is markedly different for firms, as shown in the right panel of Figure \ref{mfr.fig:fAfB-fAfB}. From the firm’s perspective, low-tier job positions are filled more quickly, while job positions are filled more slowly as firms' rank increases. This indicates that low-tier positions prioritize speed of hiring—for example, retail salespersons—whereas high-tier positions, such as university professors, emphasize finding high-quality candidates and are less concerned with filling the vacancy within the immediate hiring cycle.

		\begin{figure}[h]\centering
			\begin{minipage}[t]{0.4\linewidth}
				\centering
				\includegraphics[width=\linewidth]{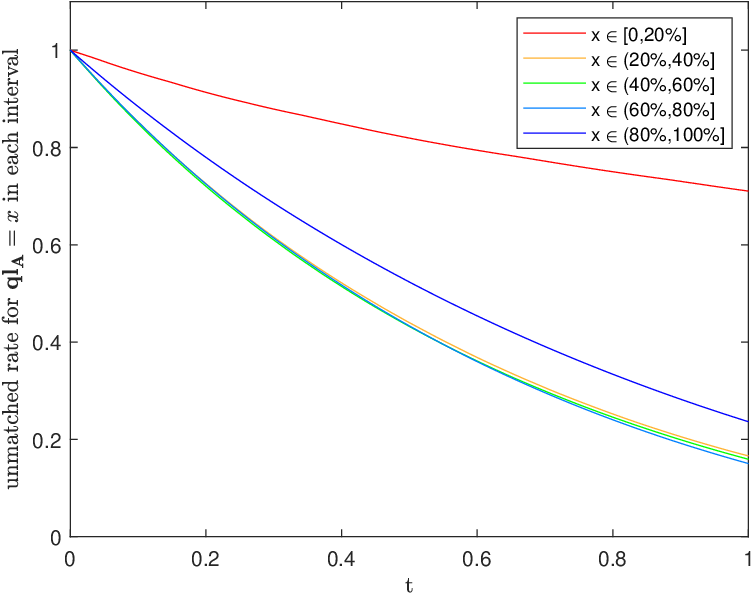}
			\end{minipage}
			\begin{minipage}[t]{0.4\linewidth}
				\centering
				\includegraphics[width=\linewidth]{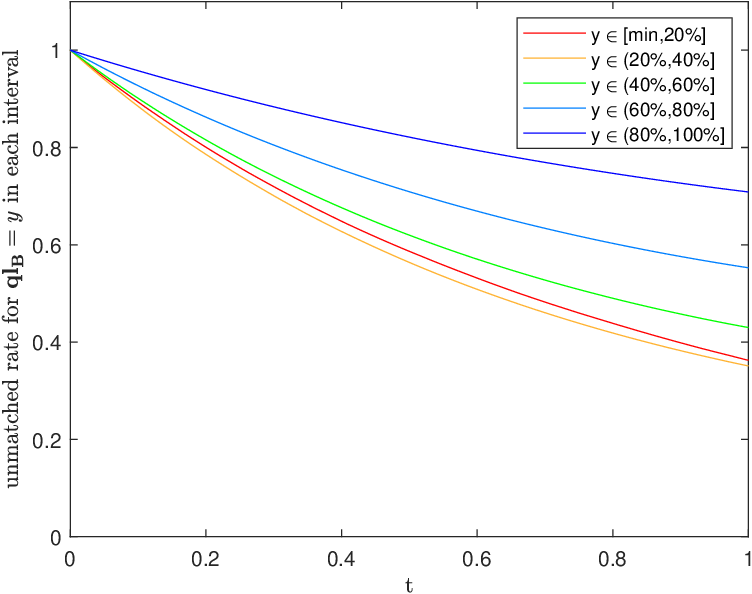}
			\end{minipage}
			\caption{Unmatched rates $F_\bA $ and $F_\bB $ against time $t$, for selected percentile bands of quality levels}
			\label{mfr.fig:FAFB_t}
		\end{figure}
		
		\begin{figure}[h]\centering
			\begin{minipage}[t]{0.4\linewidth}
				\centering
				\includegraphics[width=\linewidth]{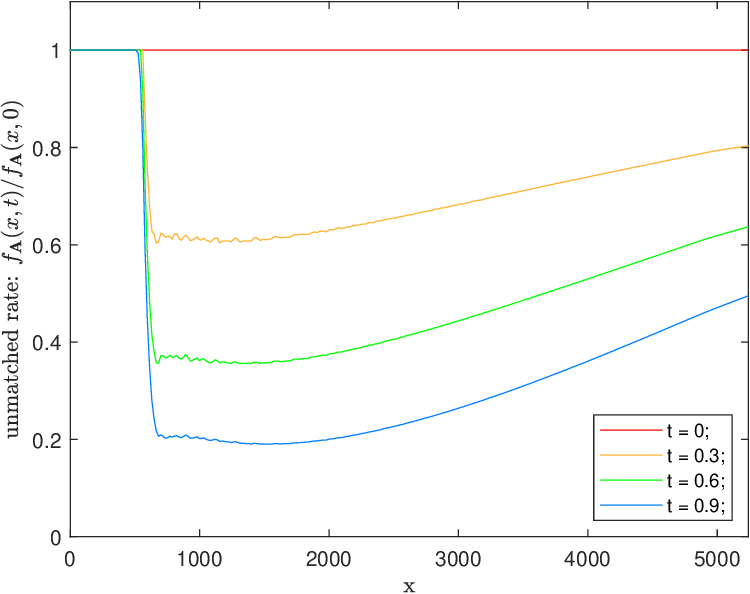}
			\end{minipage}
			\begin{minipage}[t]{0.4\linewidth}
				\centering
				\includegraphics[width=\linewidth]{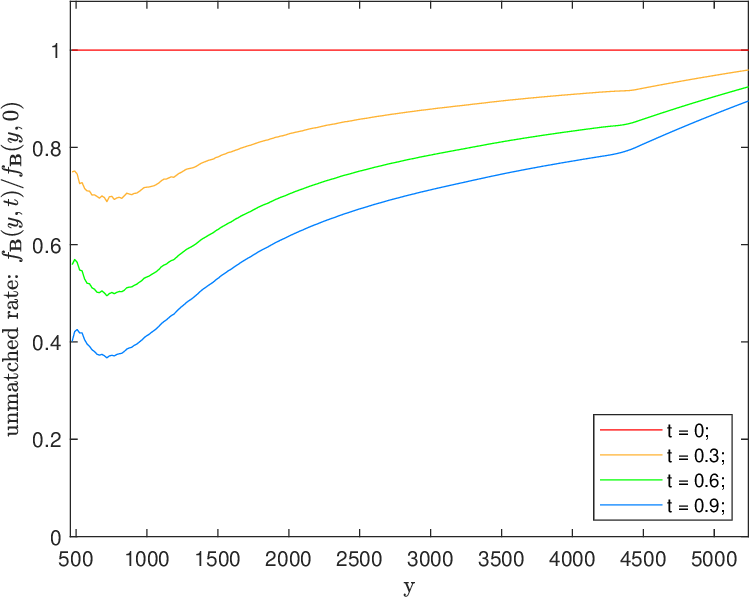}
			\end{minipage}
			\caption{Ratios of defective probability densities $f_\bA $ and $f_\bB$ to their initial values against quality levels $x$ and $y$, for selected times}
			\label{mfr.fig:fAfB-fAfB}
		\end{figure}

		For any $\bI,\bJ \in \{\bA,\bB\}$ with $\bI \neq \bJ$, we use \eqref{prop.Qtau in B.2} to plot the graphs of probability densities of the quality level of the matched partner of the representative type-$\bI$ agent $(x,0)$, given the condition that the type-$\bI$ agent $(x,0)$ is matched before closing time $T$. In \Cref{mfr.fig:gA_x}, as the percentile band increases, the mode (peak) of the probability density shifts toward higher quality values of matched partners. Agents with higher quality tend to match with partners who also have higher quality. For higher quality percentile bands (toward the right of both panels), the densities become broader and flatter, which reflects greater variability in the quality of matched partners for high-quality agents. This suggests that the matching outcomes are less predictable and more dispersed for top-quality agents. It is because there are fewer agents at the upper end, making exact matches harder and the distribution more spread out. Higher-quality agents may need to compromise more or wait longer for suitable matches.

		Moreover, the probability densities for adjacent percentile bands of agent quality are not strictly separated; instead, the curves intersect, meaning agents in one band may match with partners who possibly match with the next band. Markets do not perfectly or totally sort agents by quality—there’s always some mixing at the ``boundaries'' where agents of intermediate quality can plausibly match with partners in adjacent bands. This reflects realistic dynamics where not all agents can wait indefinitely or be perfectly selective due to market pressures or time constraints. Agents may adjust their standards based on the market situation (e.g., impatience, the scarcity of partners), allowing for some matches outside their ``ideal'' band. This overlap is a hallmark of realistic matching models, unlike traditional models with perfectly sorted pairings (e.g., \cite{gale1962college,mccall1970economics}). Compared to the strictly monotonic nature of both running utility and terminal utility, the overlap in matching outcomes provides an opportunity for ``upward mobility," meaning that a job seeker may achieve a better outcome than another with higher inherent ability. This phenomenon is a natural result of the labor market as an imperfect information market and represents one of its notable characteristics.

		\begin{figure}[h]\centering
			\begin{minipage}[t]{0.4\linewidth}
				\centering
				\includegraphics[width=\linewidth]{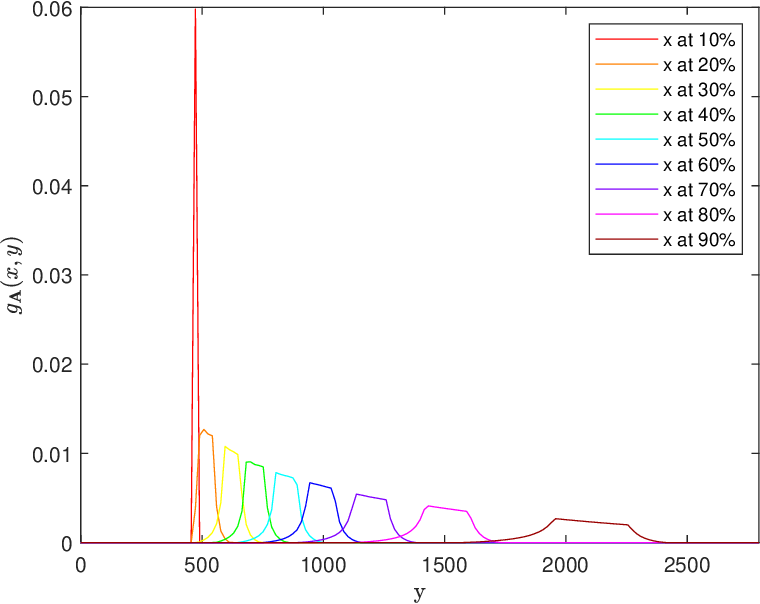}
			\end{minipage}
			\begin{minipage}[t]{0.4\linewidth}
				\centering
				\includegraphics[width=\linewidth]{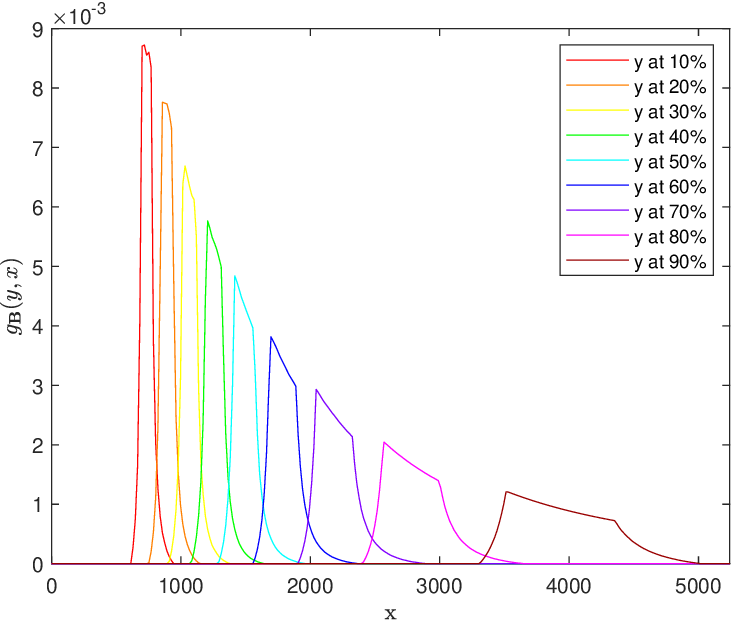}
			\end{minipage}
			\caption{Probability densities of the quality level of the matched partners of type-$\bA$ and type-$\bB$ agents, for selected percentile bands of their own quality level, see \eqref{prop.Qtau in B.2}}
			\label{mfr.fig:gA_x}
		\end{figure}

		\begin{figure}[h]\centering
			\begin{minipage}[t]{0.4\linewidth}
				\centering
				\includegraphics[width=\linewidth]{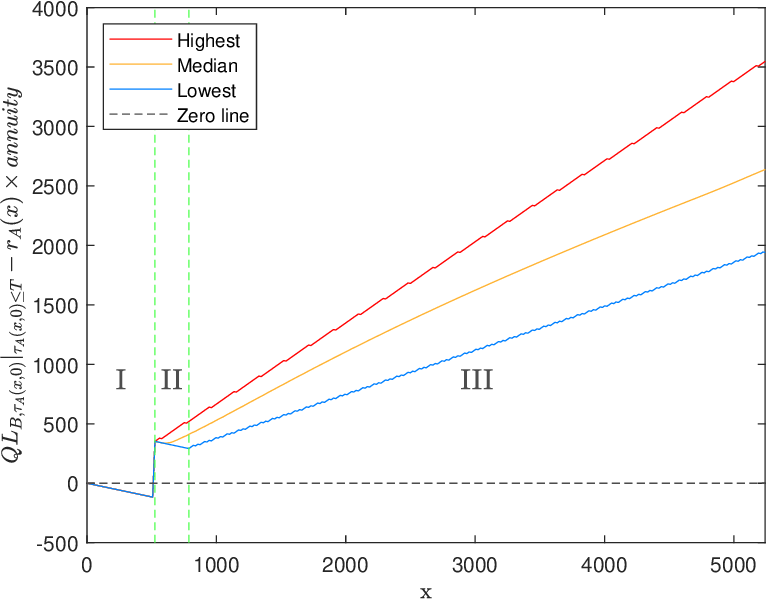}
			\end{minipage}
			\begin{minipage}[t]{0.4\linewidth}
				\centering
				\includegraphics[width=\linewidth]{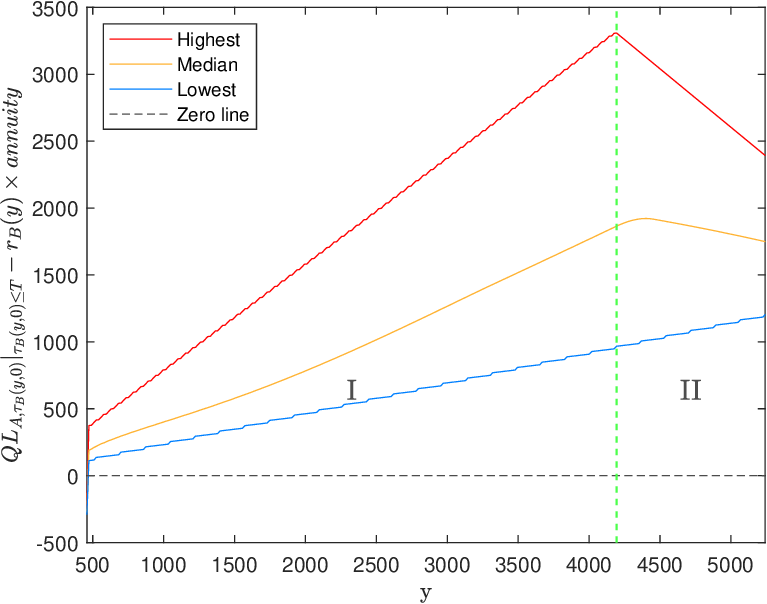}
			\end{minipage}
			\caption{Three different scenarios of net gains from matching of type-$\bA$ and type-$\bB$ agents against their own quality levels.
			}
			\label{mfr.fig:expgA_minus_kappaA_x}
		\end{figure}

		{\color{black} We now examine the net gain from matching relative to remaining unmatched, and how it varies across agents’ quality levels. For each party, it is measured by the lump‑sum payoff of 30 years' salaries or profit delivered by the matched partner, minus the present value of 30 years' income {\color{black}(from running payoff)} while remaining unmatched. For a job seeker with quality $x \ge 0$, the present value of 30 years' unmatched income (such as unemployment benefits) is $r_\bA(x,t) \times 17.47 = 0.23x$, see items (3) and (6) in \Cref{sec. Numerical Parameters}. If the job seeker instead finds a job offering a lump‑sum payoff of 30 years' salaries $y$, then the net gain of matching is $y - 0.23x$. For a firm of quality $y$, leaving a position vacant for 30 years yields a return on idle capital of $r_\mathbf{B}(y,t) \times 17.47 \approx 0.87y$. Hiring a worker of quality $x$ instead generates total profit $x$, so the net gain from filling the position is $x - 0.87y$, equivalently, the opportunity loss of leaving it vacant.
			
			Figure \ref{mfr.fig:expgA_minus_kappaA_x} reports the lowest, median, and highest net gain over the distribution of matched partners, plotted against quality level $x$. If \(\mathbb{P}_{\bI,x,0}(\tau_\bI(x,0) \leq T)=0\), then no match occurs and we set the matched outcome to \(y_{\max}(x)=y_{\text{med}}(x)=y_{\min}(x)=0\). Otherwise, let \(y_{\max}(x)\), \(y_{\text{med}}(x)\) and \(y_{\min}(x)\) denote, respectively, the highest, median, and lowest matched firm qualities in the conditional distribution \(\mathbb{P}\bigl(\mathbf{QL}_{\bB,\tau_\bA(x,0)} \in \bigcdot \;|\; \tau_\bA(x,0) \leq T=1\bigr)\) given by \eqref{prop.Qtau in B}. These correspond to the best, median and worst matches available to a job seeker of quality \(x\). For job seekers (left panel), the red, orange, and blue curves correspond to the best, median, and worst net gains for job seekers, namely, \(y_{\max}(x) - 0.23x\), \(y_{\text{med}}(x) - 0.23x\), and \(y_{\min}(x) - 0.23x\). The right panel has the analogous interpretation for firms, with the roles of $x$ and $y$ reversed.
			
			For job seekers (left panel of Figure \ref{mfr.fig:expgA_minus_kappaA_x}), we identify three distinct regions based on the behavior of the curves. In region I, the net gain is negative and decreases with $x$ since job seekers' ability is too low to meet the hiring thresholds of any firm, so they do not match. The unmatched income $0.23x$ grows with $x$ while the payoff from matching remains zero. At a critical threshold, the curves jump upward: above this point, feasible matches begin to appear, and their payoff from matching exceeds the unemployment benefit. 
			
			Beyond this threshold (region II), the lowest and median net gains decline slightly as $x$ increases, while the highest net gain continues to rise. This decline reflects the presence of a government-mandated minimum wage. For job seekers with relatively low ability in this region, many feasible matches still offer wages close to the minimum wage{\color{black}, observed from \Cref{mfr.fig:gA_x}}. As $x$ rises slightly, unmatched income grows, but typical salaries are still concentrated near the minimum wage. This causes the median and lowest net gains to edge downward. In contrast, the best possible matched income rises with ability $x$ at a faster rate than the unmatched income, so the highest net gain keeps increasing. 
			
			In region III, all three curves increase with $x$, indicating that higher‑ability job seekers enjoy progressively larger gains from matching. The opportunity loss of remaining unmatched rises with ability, and they tend to secure better matches.
			
			The pattern for firms in the right panel of Figure \ref{mfr.fig:expgA_minus_kappaA_x} differs from that for job seekers. From $y=\mu= 459.4388$ up to $y \approx 4200$ (region I), all three net gains are positive and increasing {\color{black} because the presence of the minimum quality $\mu$ removes the very-low-quality region that appears for job seekers. Moreover, } higher-rank firms attract better job seekers on average, so the lump sum payoff from hiring rises with firm quality. Beyond $y \approx 4200$ (region II), the lowest net gain continues to rise, whereas the median and highest decline. These very high-quality firms recruit only high-ability job seekers, but there is an upper bound ($x_\bA=7000$) on job-seeker quality. As the return on idle capital $0.87y$ keeps increasing as $y$, yet the quality of the best attainable job seekers no longer increases, causing the median and highest net gains to fall. By contrast, the worst attainable match still has room to improve as $y$ increases, so the lowest net gain continues to rise.
			
			Overall, the loss from remaining unmatched relative to matching generally increases with quality. The main exceptions arise for very low-quality job seekers, who are often unable to match at all, and for very high-quality firms, for which the upper bound on job-seeker quality limits further gains from matching.}
		
		\section{Conclusion}\label{sec con}
		
		This article advances the theory of dynamic, two-sided matching markets by integrating micro-level individual decisions with macro-level equilibrium outcomes through a novel mean field game framework. Unlike classical matching models that primarily characterize static stable matchings or focus on unilateral agent behavior, our approach captures the fully coupled strategic interactions between two agent populations whose acceptance criteria and distributions evolve jointly over time. We rigorously establish the existence and uniqueness of mean field Nash equilibrium by analyzing a coupled forward-backward HJB-FP system with nonlocal and dual-population features, overcoming significant technical challenges in decoupling and compactness.
		
		Our model provides rich insights into how individual acceptance thresholds adapt dynamically, reflecting temporal pressures and strategic considerations in realistic labor markets and other matching environments. Numerical results illustrate patterns such as declining value functions and thresholds over time, concentration of unmatched agents among lower-quality individuals, non-monotonic matching success by agent quality, and imperfect sorting that enables upward mobility. These findings deepen the understanding of dynamic matching phenomena by linking micro-level strategy to emergent aggregate patterns.
		
		Overall, this work offers a comprehensive theoretical and computational framework for studying dynamic two-sided markets with heterogeneous agents, bridging classical matching theory, stochastic control, and mean field game approaches. It opens promising avenues for future research in diverse applications, including labor markets, marriage markets, platform economies, and beyond.

\paragraph{Acknowledgments}
E. Bayraktar is supported in part by the NSF Grant  DMS-2507940 and in part by the Susan M. Smith Professorship. Bohan Li is supported by the National Natural Science Foundation of China
			under grant No. 12501661.  H. M. Tai is partially supported by Australian Research Council Discovery Project DP240100781.

\paragraph{Author Contributions}
The author names are listed in an alphabetical order of their surnames. All authors have contributed equally to the paper.
		
		\setlength{\bibsep}{0ex}
		\addcontentsline{toc}{section}{References}
		\bibliographystyle{abbrv}
		\bibliography{ref}

		\begin{appendix}
	 		\section{Proof of \Cref{maintheorem}: Global Existence of System \eqref{HJBx}-\eqref{boundary}}\label{sec:existence}
			\normalsize
			
			In this section, we establish the global-in-time existence of solution to the system \eqref{HJBx}-\eqref{boundary} under Assumptions \ref{condition_initial}, \ref{condition_running} and \ref{condition_terminal}. First, using inequality $(1+z)^{2+\nu} \leq 2^{1+\nu}(1+z^{2+\nu})$, together with Assumption~\ref{condition_initial}, we have the following lemma
			
			\begin{lemma}\label{lemma.integrablepdf}
				Suppose that \Cref{condition_initial} holds. Then  $
				\int_0^{\infty} (1+z)f_{\bI ,0}(z) dz\leq  \frac{2^{1+\nu}C_\bI }{\nu}$ for $\bI\in\{\bA ,\bB \}$.
			\end{lemma}

			Next, we define the space used to construct solutions via the Schauder fixed point theorem. Let $\mathcal{M}^1_f(\mathbb{R}_{\ge 0})$ be the set of finite signed Borel measures on $\mathbb{R}_{\ge 0}$, endowed with the Kantorovich-Rubinstein norm:
			\begingroup $$
			\| \mu \|_{\textup{KR}}\hspace{-2pt}:=\hspace{-1pt}| \mu(\mathbb{R}_{\ge 0})\hspace{-1pt}|
			\hspace{-1pt}+\hspace{-1pt}\sup \left\{ \int_0^\infty \hspace{-7pt} f(x)d\mu(x)\hspace{-1pt}:\hspace{-1pt} f\hspace{-2pt}\in\hspace{-2pt} \text{Lip}_1(\mathbb{R}_{\ge 0}) \text{ and }  
			f(0)=0\right\}\hspace{-2pt},  \text{ for any $\mu \in \mathcal{M}^1_f(\mathbb{R}_{\ge 0})$,}  
			$$\endgroup
			where $\text{Lip}_1(\mathbb{R}_{\ge 0})$ denotes the set of Lipschitz continuous functions from $\mathbb{R}_{\ge 0}$ to $\mathbb{R}$ with Lipschitz constant at most one. Define $C\big([0,T];\mathcal{M}^1_f(\mathbb{R}_{\ge 0})\big)$ as the space of continuous measure-valued functions equipped with the uniform norm $
			\| \mathbf{\mu} \|_{\infty,\textup{KR}}:=\sup_{t\in[0,T]} \| \mu(\bigcdot,t) \|_{\textup{KR}}, $ for any $\mu \in C\big([0,T];\mathcal{M}^1_f(\mathbb{R}_{\ge 0})\big)$. It is standard that $C\big([0,T];\mathcal{M}^1_f(\mathbb{R}_{\ge 0})\big)$ is a Banach space, see \cite{carmona2018probabilistic1,villani2008optimal}.
			
			Let $\mathcal{P}_1(\mathbb{R}_{\ge 0})$ denote the space of probability measures on $\mathbb{R}_{\ge 0}$ with finite first moment, equipped with the 1-Wasserstein metric $W_1$ defined by $
			W_1(\mu,\nu) := \inf_{\pi \in \pi(\mu,\nu)}
			\left[ \int_{\mathbb{R}_{\ge 0} \times \mathbb{R}_{\ge 0}} |x-y|\pi(dx,dy)\right]$ for any $\mu,\nu \in \mathcal{P}_1(\mathbb{R}_{\ge 0})$, where $\pi(\mu, \nu)$ is the set of joint probability measures on $\mathbb{R}_{\ge 0} \times \mathbb{R}_{\ge 0}$ with marginals $\mu$ and $\nu$. Note that $\mathcal{P}_1(\mathbb{R}_{\ge 0}) \subseteq \mathcal{M}^1_f(\mathbb{R}_{\ge 0})$. By the Kantorovich-Rubinstein duality \cite[Proposition 2.6.6]{figalli2021invitation}, the metric induced by $\| \bigcdot\|_{\textup{KR}}$ norm on $\mathcal{P}_1(\mathbb{R}_{\ge 0})$ coincides with the 1-Wasserstein distance $W_1$.

			For $\bI\in\{\bA ,\bB \}$, we define the set
			\begingroup \begin{equation}\label{KA}
				\mathcal{K}_{\bI} :=\left\{
				\mu \in \mathcal{P}_1(\mathbb{R}_{\ge 0})
				\;\middle|\;
				\begin{aligned}
					& \text{there exist } p \in [0,1]
					\text{ and } g \in L^1(\mathbb{R}_{> 0};\mathbb{R}_{\geq 0}) \text{ such that}\\
					& 
					\mu = p \delta_{\textup{D}} 
					+ g\,d\mathcal{L}^1
					\hspace{3pt}\text{and} \hspace{3pt}
					\sup_{x > 0} \big(1 + x^{2+\nu}\big) g(x) \le C_{\bI}.
				\end{aligned}
				\right\},
			\end{equation}\endgroup
			where $\delta_{\textup{D}}$ is the Dirac delta measure centered at zero. Denote by $\overline{\mathcal{K}_{\bI}}$ the closure of $\mathcal{K}_{\bI}$ under the topology generated by the $\| \bigcdot\|_{\textup{KR}}$ norm.  The following lemma summarizes key properties of $\overline{\mathcal{K}_{\bI}}$.
			
			\begin{lemma}\label{lemma.closureofK}
				Let $\bI\in\{\bA ,\bB \}$. For any $\mu \in \overline{\mathcal{K}_{\bI}}$ and $0\leq a\leq b$, the following hold:
				\begin{enumerate}[label=\textbf{(\arabic*)},leftmargin=*]
					\setcounter{enumi}{0}
					
					\item\hspace{-5pt}\textbf{.}\label{lemma.closureofK.compact} $\overline{\mathcal{K}_{\bI}} $ is compact in $\mathcal{P}_1(\mathbb{R}_{\ge 0})$ under the topology induced by the 1-Wasserstein metric $W_1$;
					
					\item\hspace{-5pt}\textbf{.}\label{lemma.absolutelycontinuous} there exist $p\in[0,1]$ and $g\in L^1(\mathbb{R}_{> 0};\mathbb{R}_{\geq 0})$ such that $\textup{supp}(g\,d\mathcal{L}^1)\subseteq\mathbb{R}_{> 0}$ and $\mu = p \delta_{\textup{D}} + g\,d\mathcal{L}^1$;
					
					\item\hspace{-5pt}\textbf{.}\label{lemma.closureofK.integrability} $\int_{[0,\infty)} (1+x) d\mu \leq  \frac{2^{1+\nu}C_\bI}{\nu}+1$,  $\int_{a}^{b}  d\mu \leq C_{\bI} (b-a)$, and  $\int_{a}^{b} (1+x)  d\mu \leq 2C_{\bI} (b-a)$.

				\end{enumerate}
				Throughout, $\int_{[0,\infty)}$ or $\int_{\mathbb{R}_{\geq 0}}$ denotes integration over the half-line including the origin, whereas $\int_0^\infty$	excludes the atom at the origin.	
			\end{lemma}
			
			\begin{proof}
				\noindent{\bf Proof of item \ref{lemma.closureofK.compact}:} 
				Consider the set $
				\mathcal{K}_{\bI,0}:=\left\{\mu\in\mathcal{P}_1(\mathbb{R}_{\ge 0}):\int_0^{\infty}x^{1+\frac{\nu}2}d\mu
				\le \frac{2^{2+\nu}C_{\bI}}{\nu}\right\}.$ Let $a>0$, the inequality $
				\sup_{\mu \in \mathcal{K}_{\bI,0}}\int_a^\infty d\mu 
				< \sup_{\mu \in \mathcal{K}_{\bI,0}}
				\int_a^\infty \frac{x^{1+\nu/2}}{a^{1+\nu/2}} d\mu 
				\leq \frac{2^{2+\nu}C_{\bI} }{\nu a^{1+\nu/2}}$ implies that $\mathcal{K}_{\bI,0} $ is tight. By Prokhorov's theorem, any sequence $\{\mu_{n}\}_{n\in\mathbb{N}}\subseteq \mathcal{K}_{\bI,0} $ contains a subsequence $\{\mu_{n_k}\}_{k\in \mathbb{N}}$ which weakly converges to $\mu_*\in \mathcal{P}_1(\mathbb{R}_{\ge 0})$. By Skorokhod's representation theorem, there exist $L^1$ random variables $X_k$ and $X_*$ such that $\mathcal{L}(X_k)=\mu_{n_k}$, $\mathcal{L}(X_*)=\mu_*$ and $X_k\rightarrow X_*$ $\mathbb{P}$-a.s. Since $
				\sup_k\mathbb{E}\pig(|X_k|^{1+\frac{\nu}{2}}\pig)\le \frac{2^{2+\nu}C_{\bI} }{\nu}<\infty$, Markov's inequality implies that $\{X_k\}_{k\in \mathbb{N}}$ is uniformly integrable. Hence, by Vitali convergence theorem, we have
				$
				\lim_{k\rightarrow \infty} \mathbb{E}|X_k-X_*|=0$ which implies $\mu_{n_k} \rightarrow \mu_*$ in 1-Wasserstein metric as $k\to \infty$.  Moreover, Fatou's lemma preserves
				the moment bound in the limit: $
				\int_0^{\infty}x^{1+\frac{\nu}{2}}d\mu_*
				\le \frac{2^{2+\nu}C_{\bI} }{\nu}$,
				leading to $\mu_*\in \mathcal{K}_{\bI,0}$. Therefore, $\mathcal{K}_{\bI,0}$ is compact. For any $\mu \in \mathcal{K}_\bI$, it holds that $
				\int_0^\infty x^{1+\frac{\nu}{2}}d\mu                                      
				\leq C_\bI \int_0^\infty \frac{x^{1+\frac{\nu}{2}}}{1+x^{2+\nu}}dx                
				\leq                                               2^{1+\nu}C_\bI \int_0^\infty \frac{(1+x)^{1+\frac{\nu}{2}}}{(1+x)^{2+\nu}}dx 
				\leq                                             \frac{2^{2+\nu}C_\bI }{\nu} $ implying that $\mu \in \mathcal{K}_{\bI,0}$. Hence, $\mathcal{K}_\bI$ is a subset of $\mathcal{K}_{\bI,0}$. Together with the compactness of $\mathcal{K}_{\bI,0}$, we can see that $\overline{\mathcal{K}_\bI} $ is compact.

				\noindent{\bf Proof of item \ref{lemma.absolutelycontinuous}:} For any $\mu \in \overline{\mathcal{K}_\bI}\subseteq \mathcal{K}_{\bI,0}$, we consider a sequence of measures $\{\mu_{n}\}_{n\in\mathbb{N}}\subseteq \mathcal{K}_\bI$ such that $\lim_{n\to \infty}W_1(\mu_{n},\mu)=0$. By \cite[Theorem 7.12]{villani2008optimal}, $\mu_{n}$ converges to $\mu$ weakly. 	Let $\epsilon>0$, for any Lebesgue null set $\mathcal{N}$ on $(0,\infty)$, there exists an open set $U \supseteq \mathcal{N}$ and $\mathcal{L}^1(U) < \epsilon/C_\bI$. By Portmanteau's theorem, we have
				$
				\mu(U) \leq \liminf_{n \to \infty} \mu_{n}(U) \leq  C_\bI \int_U dx  < \epsilon. 
				$
				Therefore, it holds that $\mu(\mathcal{N}) \leq \mu(U) < \epsilon$ which forces $\mu(\mathcal{N})=0$, leading to the fact that $\mu$ is absolutely continuous in $(0,\infty)$. The Lebesgue decomposition theorem yields that $\mu=\mu^\textup{s}+\mu^{\textup{ac}}$, where $\mu^{\textup{ac}} \ll \mathcal{L}^1$ and $\mu^\textup{s} \perp \mathcal{L}^1$. By Radon-Nikodym theorem, there exist $p\in[0,1]$ and $g\in L^1(\mathbb{R}_{> 0};\mathbb{R}_{\geq 0})$ such that $\textup{supp}(g\,d\mathcal{L}^1)\subseteq \mathbb{R}_{> 0}$ and therefore $\mu = p \delta_{\textup{D}} + g\,d\mathcal{L}^1$. Without loss of generality, we assume $g(0)=0$.

				\noindent{\bf Proof of item \ref{lemma.closureofK.integrability}:} 	Let $\mu \in \overline{\mathcal{K}_{\bI}} \subseteq \mathcal{K}_{\bI,0}$ and let a sequence $\{\mu_k\}_{k\in \mathbb{N}} \subseteq \mathcal{K}_{\bI}$ such that $\lim_{k \to \infty}W_1(\mu,\mu_k)=0$. As $(1+x) \in\text{Lip}_1(\mathbb{R}_{\ge 0}) $, it follows that $
				\int_{[0,\infty)} (1+x)d \mu   
				\leq  \int_{[0,\infty)} (1+x) d\mu_k + \sup_{\varphi\in\text{Lip}_1(\mathbb{R}_{\ge 0})}
				\int_{[0,\infty)} \varphi(x) d\left(\mu-\mu_k\right).$ By similar arguments as in Lemma \ref{lemma.integrablepdf}, we have $
				\int_{[0,\infty)} (1+x)d\mu_k \leq \frac{2^{1+\nu}C_\bI}{\nu}+1. 
				$
				Thus, by the Kantorovich-Rubinstein duality \cite[Proposition 2.6.6]{figalli2021invitation}, we have $
				\int_{[0,\infty)} (1+x)d \mu \leq \frac{2^{1+\nu}C_\bI}{\nu} + 1 + W_1(\mu,\mu_k).$ Taking $k \to \infty$, we deduce the first result in \ref{lemma.closureofK.integrability}. Let $g_k \in L^1(\mathbb{R}_{> 0};\mathbb{R}_{\geq 0})$ and $p_k \in [0,1]$ such that $\mu_k=p_k \delta_{\textup{D}} + g_kd\mathcal{L}^1$ for $k \in\mathbb{N}$. By Portmanteau's theorem and the fact that $\sup_{x> 0}g_k(x) \leq \sup_{x> 0}(1+x^{2+\nu})g_k(x) \le C_\bI$, we obtain $
				\int_{a}^{b} d  \mu  
				\leq       \liminf_{k \to \infty}  
				\int_{a}^{b} d\mu_k
				\leq          C_{\bI} (b-a)  $ for any $0<a\leq b$, leading to the second result in \ref{lemma.closureofK.integrability}. Moreover, 
				for $x \ge 0$, the inequality $
				1+x\leq 2(1+x^{2+\nu})$ implies $\sup_{x> 0}(1+x)g_k(x) \leq \sup_{x> 0}2(1+x^{2+\nu})g_k(x) \le 2 C_\bI$ and hence $
				\int_{a}^{b} (1+x)d\mu_k
				\leq          2 C_{\bI} (b-a).$ Therefore,
				\begingroup \begin{align}
					\int_{a}^{b} (1+x)  d \mu
					\le 2 C_{\bI} (b-a) + \int_{a}^{b} (1+x) d (\mu-\mu_k)
					.\label{1666}
				\end{align} \endgroup
				Suppose that $W_1(\mu_k,\mu)>0$ for each finite $k\in \mathbb{N}$, otherwise the result is established. Define the test function on $\mathbb{R}_{\geq 0}$: $
				\phi_\delta(x)
				= \left(\frac{x+\delta-a}{\delta}\right)(1+a)\,\mathbf{1}_{(a-\delta,\,a]}(x)
				+ (1+x)\,\mathbf{1}_{(a,\,b)}(x)
				+ \left(\frac{b+\delta-x}{\delta}\right)(1+b)\,\mathbf{1}_{[b,\,b+\delta)}(x),$ for some $\delta>0$ to be determined later. If $\delta$ is sufficiently small such that $\delta\in(0,a)$, then the second result in \ref{lemma.closureofK.integrability} yields $
				\pig|  \int_b^{b+\delta} \dfrac{b+\delta-x}{\delta}(1+b) d(\mu-\mu_k) \pig| 
				\le (1+b)\pig[\mu ((b,  b+\delta))+\mu_k((b,  b+\delta))\pig] 
				\le 2 (1+b)C_{\bI}\delta,$ and $	\pig|  \int_{a-\delta}^{a} \frac{x+\delta-a}{\delta}(1+a) d(\mu-\mu_k) \pig|
				\le 2 (1+b)C_{\bI}\delta$ similarly. We take $k$ large enough such that $\delta:=\big[W_1(\mu_k,\mu)\big]^{\frac12}< \frac{a}{2}<1+b$. Since the Kantorovich-Rubinstein duality implies $
				\pig|   \int_0^{\infty} \phi_\delta(x) d (\mu-\mu_k) \pig| 
				\le \left(\frac{1+b}{\delta} \vee 1\right) 
				W_1(\mu_k,\mu)
				\leq \left(1+b\right) 
				\big[W_1(\mu_k,\mu)\big]^{\frac12},$ we deduce
				\begingroup \begin{align*}
					\bigg| \int_{a}^{b} (1+x) d (\mu-\mu_k) \bigg| 
					\le\,&   \bigg|   \int_0^{\infty} \phi_\delta(x) d (\mu-\mu_k) \bigg|
					+\bigg|  \int_{a-\delta}^{a} \dfrac{x+\delta-a}{\delta}(1+a) d (\mu-\mu_k) \bigg| \\
					&+\bigg|  \int_b^{b+\delta} \dfrac{b+\delta-x}{\delta}(1+b) d (\mu-\mu_k) \bigg|\\
					\le\,& (1+b)\big[W_1(\mu_k,\mu)\big]^{\frac12}
					+4(1+b)C_{\bI}\big[W_1(\mu_k,\mu)\big]^{\frac12}
				\end{align*}\endgroup
				which tends to zero as $k\rightarrow \infty$. From \eqref{1666}, we obtain the third result in \ref{lemma.closureofK.integrability}.
			\end{proof}

			For notational convenience, we use the same symbol to denote a probability density function and the corresponding measure; that is, when we write $f$, it denotes both the density function and the measure $f d\mathcal{L}^1$ it induces. For $\bI\in\{\bA ,\bB \}$, we define the set
			\begingroup \begin{equation}\label{DA}
				\mathcal{D}_{\bI}:=\left\{\mu\in C([0,T];\overline{\mathcal{K}_{\bI}}) \,:\,
				\sup_{\substack{s,t \in [0,T],s \neq t }}\frac{W_1\big(\mu(\bigcdot,t),\mu(\bigcdot,s)\big)}{|t-s|}\le c_\bI\right\},
			\end{equation}\endgroup
			where $c_\bA :=\frac{\lambda_\bB 2^{1+\nu} C_\bA }{\nu}$, $c_\bB :=\frac{\lambda_\bA 2^{1+\nu} C_\bB }{\nu}$, and $C([0,T];\overline{\mathcal{K}_{\bI}})$ is equipped with the uniform 1-Wasserstein metric $D_1$ defined by $D_1(\mu,\nu):=\sup\limits_{t\in[0,T]}W_1(\mu(\bigcdot,t),\nu(\bigcdot,t))$.

			\subsection{Sketch of Proof}\label{sec. Main result and Sketch of Proof}
			We outline the proof of \Cref{maintheorem} via the following steps. 
			
			\noindent \textbf{Step 1. Decoupled systems:} For each $\bI \in 
			\{\bA,\bB\}$, we fix any $\mu_{\bI} \in \mathcal{D}_{\bI}$. The structure of $\overline{\mathcal{K}_{\bI}}$ implies that there exist unique $p_{\bI} \in L^{\infty}\big([0,T];[0,1]\big)$ and $g_\bI \in L^{\infty}\big([0,T];L^1(\mathbb{R}_{> 0};\mathbb{R}_{\geq 0})\big)$ such that $\mu_{\bI} = p_{\bI} \delta_{\textup{D}} + g_{\bI} d\mathcal{L}^1$. We then define $(V_{\bA}, V_{\bB}, f_{\bA}, f_{\bB})$ as the solution to the decoupled HJB–FP system:
			\begingroup \begin{align}
				\rho V_\bA(x,t) - \frac{\partial V_\bA(x,t)}{\partial t}
				&= 
				\mathbf{1}_{\{V_\bA (x,t) \leq V_\bB ^{-1}(x,t)\}} 
				\lambda_\bB
				\int_{V_\bA(x,t)}^{V_\bB^{-1}(x,t)}
				\big[y - V_\bA(x,t)\big] g_\bB(y,t) dy
				+ r_\bA(x,t);
				\label{NHJBx} \\[6pt]
				\rho V_\bB(y,t) - \frac{\partial V_\bB(y,t)}{\partial t}
				&= 
				\mathbf{1}_{\{V_\bB (y,t) \leq V_\bA ^{-1}(y,t)\}} 
				\lambda_\bA
				\int_{V_\bB(y,t)}^{V_\bA^{-1}(y,t)}
				\big[x - V_\bB(y,t)\big] g_\bA(x,t) dx
				+ r_\bB(y,t);
				\label{NHJBy} \\[6pt]
				\frac{\partial f_\bA(x,t)}{\partial t}
				&= \Big(
				-\lambda_\bB f_\bA(x,t)
				\int_{V_\bA(x,t)}^{V_\bB^{-1}(x,t)}
				f_\bB(y,t) dy
				\Big)
				\land 0;
				\label{NFPx} \\[6pt]
				\frac{\partial f_\bB(y,t)}{\partial t}
				&= \Big(
				-\lambda_\bA f_\bB(y,t)
				\int_{V_\bB(y,t)}^{V_\bA^{-1}(y,t)}
				f_\bA(x,t) dx
				\Big)
				\land 0,
				\label{NFPy}
			\end{align}\normalsize\endgroup
			for $(x,y,t) \in \mathbb{R}_{>0} \times \mathbb{R}_{>0} \times [0,T]$, with the initial-terminal conditions:
			\begingroup \begin{equation}\label{Nboundary}
				V_\bA (x,T)=h_\bA (x),\quad
				V_\bB (y,T)=h_\bB (y),\quad
				f_\bA (x,0)=f_{\bA ,0}(x),\quad
				f_\bB (y,0)=f_{\bB ,0}(y).
			\end{equation} \endgroup
			With $(V_\bA,V_\bB,f_\bA,f_\bB)$ obtained in the above, we define $(q_{\bA},q_{\bB})$ such that
			\begingroup \begin{align*}
				q_{\bA} (t) := 
				1-\int^\infty_0 f_\bA(x,t) dx
				,\hspace{15pt} \text{and} \hspace{15pt}
				q_{\bB} (t) := 
				1-\int^\infty_0 f_\bB(y,t) dy,
			\end{align*}\endgroup
			and set $\mu^*_{\bI} := q_{\bI} \delta_{\mathrm{D}} + f_{\bI} d\mathcal{L}^1$ for $\bI \in \{\bA,\bB\}$. We define the map $\Phi(\mu_{\bA}, \mu_{\bB}) = (\mu^*_{\bA}, \mu^*_{\bB})$. Note that $p_\bI$ and $q_\bI$ are not parts of the solution to \eqref{NHJBx}-\eqref{NFPy}, see Remark \ref{rmk reason of p and q} for reasons of introducing them. 
			
			\noindent \textbf{Step 2. Well-definedness of $\Phi$:} We fix $(\mu_{\bA} ,\mu_{\bB} )\in \mathcal{D}_{\bA}\times\mathcal{D}_{\bB}$ and proceed in two steps. 
			
			\noindent \textbf{Step 2A. HJB system:} Divide $[0,T]$ into subintervals $\{I_{i}\}_{i=1}^{N_1}$. For each $i=1,2,\cdots,N_1$, we fix $V_{\bB,i}$ in an appropriate function space and subsequently apply the Banach fixed point theorem to obtain a local solution $\overline{V}_{\bA,i}$ to equation \eqref{NHJBx} on $\mathbb{R}_{>0} \times I_i$, provided $|I_i|$ is small enough. Substituting $V_\bA=\overline{V}_{\bA,i}$ in \eqref{NHJBy}, we establish the existence of the local solution $\overline{V}_{\bB,i}$ to equation \eqref{NHJBy} on $\mathbb{R}_{>0} \times I_i$. Therefore, the map $\overline{V}_{\bB,i}:=\Psi_i(V_{\bB,i})$ is well defined on a suitable function space. By showing that $\Psi_i$ is a self-map and is contractive, it ensures a unique local solution $(V_{\bA,i}^*, V_{\bB,i}^*)$ to equations \eqref{NHJBx}-\eqref{NHJBy}. Finally, we paste all the local solutions backward in time, provided that $\max_i|I_i|$ is small, obtaining a global solution $(V^*_{\bA},V^*_{\bB})$ to equations \eqref{NHJBx}-\eqref{NHJBy} subject to terminal conditions in \eqref{Nboundary}.
			
			\noindent \textbf{Step 2B. FP system:} Using $(V_\bA,V_\bB)=(V_{\bA}^*, V_{\bB}^*)$ obtained from Step 2A, the FP system \eqref{NFPx}–\eqref{NFPy} admits a unique global solution $(f_{\bA}^*, f_{\bB}^*)$ by a similar argument as in Step 2A. Thus, the map $\Phi$ is well-defined.
			
			\noindent \textbf{Step 3. Schauder fixed point theorem:} Finally, we verify that: \textbf{(3A).} $\Phi$ is a self map; \textbf{(3B).} $\Phi$ is continuous; and \textbf{(3C).} $\mathcal{D}_\bA \times \mathcal{D}_\bB$ is compact. The Schauder fixed point theorem then yields a global-in-time solution to \eqref{HJBx}–\eqref{boundary} in the claimed space. 
			
			\begin{remark}\label{rmk reason of p and q}
				The presence of $p_\bI$ and $q_\bI$ serves for technical reasons. To solve \eqref{NHJBx}-\eqref{Nboundary}, one should expect that we only need to involve $g_\bI$, $V_\bI$ and $f_\bI$. However, it is hard to find a suitable function space to contain the solutions of the FP system whose solutions are defective probability densities. For example, the typical space should be some bounded subset of $C\big([0, T]; L^1(\mathbb{R}_{> 0};  \mathbb{R}_{\geq 0})\big) $. Unfortunately, we do not have enough compactness for this space in order to apply Schauder fixed point theorem. To this end, we construct probability measures using $g_\bI$ and $f_\bI$ such that these measures lie in a compact subset $\mathcal{D}_{\bI}$ of $C\big([0,T];\mathcal{M}^1_f(\mathbb{R}_{\ge 0})\big)$ (see Lemma \ref{theorem.compact}). Hence, we establish the existence of solution to \eqref{NHJBx}-\eqref{Nboundary} by studying the map $\Phi$ with input $(\mu_\bA,\mu_\bB) \in \mathcal{D}_{\bA} \times \mathcal{D}_{\bB}$ instead of the input $(g_\bA,g_\bB)$.

			\end{remark}

			\subsection{Global Well-posedness of HJB Equations \eqref{NHJBx}-\eqref{NHJBy}}\label{section.wellpose.HJB}
			In this section, we establish Step 2A outlined in Appendix \ref{sec. Main result and Sketch of Proof}, namely the existence of a unique global-in-time solution to the HJB equations \eqref{NHJBx}–\eqref{NHJBy}, subject to the terminal conditions given in \eqref{Nboundary}, for any fixed $(\mu_{\bA}, \mu_{\bB}) \in \mathcal{D}_{\bA} \times \mathcal{D}_{\bB}$. 
			
			We partition the time interval $[0,T]$ into $N_1$ uniform subintervals of length $\delta_1 > 0$. We first solve \eqref{NHJBx}–\eqref{NHJBy} on the final subinterval $[T - \delta_1, T]$ using the given terminal conditions specified in \eqref{Nboundary}. For each preceding subinterval $[T - i\delta_1, T - (i-1)\delta_1]$ with $i = 2,3,\dots,N_1$, the terminal condition is provided by the initial value of the solution on the adjacent next interval. This procedure is applied recursively, stepping backward in time, until the entire interval $[0,T]$ is covered. Finally, the global solution on $[0,T]$ is obtained by concatenating the local solutions constructed on each subinterval. 
			
			For $\bI\in\{\bA ,\bB \}$, we let $\Pi_\bI $, $k_\bI $, $K_\bI$ be positive constants to be determined later. For any $v \in C(\mathbb{R}_{\ge 0}; \mathbb{R}_{\ge 0})$, we define the weighted supremum norm
			$\|v\|_{0,1/w}:=\sup_{z\ge 0}\frac{|v(z)|}{1+z}$; and for any $V \in C\big([0, T];C(\mathbb{R}_{\geq 0};\mathbb{R}_{\geq 0})\big)$, we define the norm
			$\| V \|_{\mathcal{H}}:=\sup_{t\in[0, T]}\| V(\bigcdot, t) \|_{0,1/w}$. We define $\Vert V\Vert_{\mathcal{H},i}:=\sup_{t\in[T-i\delta_1, T-(i-1)\delta_1]}\| V(\bigcdot, t) \|_{0,1/w}$ and the set $\mathcal{H}_{\bI, i}$ as
			\begingroup \begin{equation}\label{def.HIn}
				\scalebox{0.92}{$\left\{
					V \in C\big([T - i\delta_1, T - (i - 1)\delta_1]; C(\mathbb{R}_{\geq 0}; \mathbb{R}_{\geq 0})\big) \;\middle|\;
					\begin{aligned}
						&k_{\bI} (z_1 - z_2) \leq V(z_1,t) - V(z_2,t) \leq K_{\bI} (z_1 - z_2),\, V(0,t)=0\\
						& \text{for all }t \in [T - i\delta_1, T - (i - 1)\delta_1] \text{ and } 0 \le z_2\le z_1; \\
						&\text{and $\|V\|_{\mathcal{H},i} \le \Pi_{\bI}$ \,\,}
					\end{aligned}
					\right\}$}.
			\end{equation}\normalsize\endgroup
			The space $\mathcal{H}_{\bI}$ is defined in \eqref{def.HI}. Throughout this section, we fix $\mu_\bI\in \mathcal{D}_\bI$ and its associated absolutely continuous component $g_\bI \in L^\infty\big([0,T];L^1(\mathbb{R}_{> 0};\mathbb{R}_{\geq 0})\big)$ such that $\mu_{\bI} = p_{\bI} \delta_{\textup{D}} + g_{\bI} d\mathcal{L}^1$ for a unique $p_{\bI} \in L^\infty\big([0,T];[0,1]\big)$. This representation is always possible due to the definition of $\overline{\mathcal{K}_{\bI}}$. For $i=1,2,\cdots,N_1$, we define a map $\Psi_i$ on $\mathcal{H}_{\bB ,i}$ such that $\overline{V}_\bB =\Psi_i(V_\bB )$ where $(\overline{V}_{\bA},\overline{V}_{\bB})$ is the unique solution to the system:
			\begingroup \begin{equation}\label{NNHJBx}
				\rho \overline{V}_\bA (x,t)-\frac{\partial \overline{V}_\bA (x,t) }{\partial t}
				=\mathbf{1}_{\{\overline{V}_\bA (x,t) \leq V_\bB ^{-1}(x,t)\}} 
				\lambda_\bB  \int_{\overline{V}_\bA (x,t)}^{V_\bB ^{-1}(x,t)}\big[y-\overline{V}_\bA (x,t)\big] g_\bB (y,t) d y + r_\bA (x,t);
			\end{equation}
			\begin{equation}\label{NNHJBy}
				\rho \overline{V}_\bB (y,t)-\frac{\partial \overline{V}_\bB (y,t) }{\partial t}
				=\mathbf{1}_{\{\overline{V}_\bB (y,t) 
					\leq \overline{V}_\bA^{-1}(y,t)\}} 
				\lambda_\bA  \int_{\overline{V}_\bB (y,t)}^{\overline{V}_\bA ^{-1}(y,t)}\big[x-\overline{V}_\bB (y,t)\big] g_\bA (x,t) d x + r_\bB (y,t),
			\end{equation}\endgroup
			in $\mathbb{R}_+\times[T-i \delta_1, T-(i-1)\delta_1]$ with terminal conditions
			\begingroup \begin{equation}\label{eq. terminal VA,VB local}
				\overline{V}_\bA (x,T-(i-1)\delta_1)=\widetilde{V}_{\bA,i} (x),\ 
				\overline{V}_\bB (y,T-(i-1)\delta_1)=\widetilde{V}_{\bB,i} (y),
			\end{equation}\endgroup
			where $\widetilde{V}_{\bI,i}(\bigcdot)$ is a given function such that $\widetilde{V}_{\bI,i}(\bigcdot)\in \mathcal{H}_{\bI}$. The well-definedness of $\Psi_i$ will be established in the following sections. If $i=1$ and $\bI\in\{\bA ,\bB \}$, then we set the terminal value $\widetilde{V}_{\bI,1} (x) = h_\bI (x)$. 
			
			To apply the Banach fixed point theorem, the subinterval width $\delta_1$ must be sufficiently small, and the map $\Psi_i$ must be a self-map on $\mathcal{H}_{\bB,i}$. To ensure this, the constants must be chosen as follows: $\delta_1$ is defined in~\eqref{def_deltat}; $\Pi_\bA$, $k_\bA$, and $K_\bA$ are defined in~\eqref{def_PiA}, \eqref{def_kA}, and~\eqref{def_KA}, respectively; and $\Pi_\bB$, $k_\bB$, and $K_\bB$ are all defined in~\eqref{def_PiB,KB,kB}.

			\subsubsection{Existence and Uniqueness of Local Solution to Equation \eqref{NNHJBx}}\label{VA_Existence}
			Fix $i=1,2,\cdots,N_1$ and $V_\bB \in \mathcal{H}_{\bB ,i}$. To establish the existence and uniqueness of solution to \eqref{NNHJBx} in $\mathcal{H}_{\bA,i}$, we first introduce the larger set $\widetilde{\mathcal{H}}_{\bA ,i}:=\big\{V\in C([T-i \delta_1, T-(i-1)\delta_1];C(\mathbb{R}_{\geq 0};\mathbb{R}_{\geq 0})\big): \| V \|_{\mathcal{H},i}\le \Pi_\bA \text{ and }V(0,t)=0\big\}$, where 
			\begingroup \begin{equation}\label{def_PiA}
				\Pi_\bA :=\left[\frac{\lambda_\bB }{\rho}\left(\frac{2^{1+\nu}C_\bB }{\nu} + 
				\Gamma_\bA\right)\right]\vee L_\bA .
			\end{equation}\endgroup
			Next, we define the map $\mathcal{G}_{\bA,i} :\widetilde{\mathcal{H}}_{\bA ,i} \ni V_\bA  \longmapsto \widehat{V}_\bA =\mathcal{G}_{\bA,i} (V_\bA)$ such that for each $V_\bA\in\widetilde{\mathcal{H}}_{\bA ,i} $, the function $\widehat{V}_\bA =\mathcal{G}_{\bA,i} (V_\bA)$ is the solution to the equation:
			\begingroup \begin{equation}\label{NNNHJBx}
				\rho \widehat{V}_\bA (x,t)-\frac{\partial \widehat{V}_\bA (x,t) }{\partial t}
				=\mathbf{1}_{\{V_\bA (x,t) 
					\leq V_\bB^{-1}(x,t)\}} 
				\lambda_\bB  \int_{V_\bA (x,t)}^{V_\bB ^{-1}(x,t)}\big[y-V_\bA (x,t)\big] g_\bB (y,t) d y 
				+r_\bA (x,t),
			\end{equation}\endgroup
			in $\mathbb{R}_{\geq 0}\times[T-i \delta_1, T-(i-1)\delta_1]$ with the terminal condition $\widehat{V}_\bA (x,T-(i-1)\delta_1)=\widetilde{V}_{\bA,i}(x)$. The well-definedness of $\mathcal{G}_{\bA,i}$ will be shown in the lemma below. 
			
			\begin{lemma}\label{selfmap_PsiA}
				Suppose that Assumptions \ref{condition_running} and \ref{condition_terminal} hold. For each $V_\bA\in\widetilde{\mathcal{H}}_{\bA ,i}$, there exists a unique solution $\widehat{V}_\bA =\mathcal{G}_{\bA,i} (V_\bA) \in \widetilde{\mathcal{H}}_{\bA ,i}$ to the equation \eqref{NNNHJBx} in $\mathbb{R}_{\geq 0}\times[T-i \delta_1, T-(i-1)\delta_1]$, subject to the terminal condition $\widehat{V}_\bA (x,T-(i-1)\delta_1)=\widetilde{V}_{\bA,i}(x)$ for $x\in\mathbb{R}_{\geq 0}$. The solution {\color{black} is differentiable for a.e. $t$ and } is  given by
				\begingroup \begin{align}\label{intergal_NNNHJBx}
						\widehat{V}_\bA (x,t) 
						=& \int_t^{T-(i-1)\delta_1}   e^{-\rho (u-t)} \!\!
						\left\{\! \mathbf{1}_{\{V_\bA (x,u) 
							\leq V_\bB^{-1}(x,u)\}} 
						\lambda_\bB \!\!\int_{V_\bA (x,u)}^{V_\bB ^{-1}(x,u)}\!\!\big[y-V_\bA (x,u)\big] g_\bB (y,u) d y  + r_\bA (x,u)\!\right\}\! du\!\!\nonumber\\
						&+e^{-\rho [T-(i-1)\delta_1-t]}\widetilde{V}_{\bA,i}(x)\hspace{-5pt}
				\end{align}%
				\endgroup
				for all $(x,t) \in \mathbb{R}_{\geq 0}\times[T-i \delta_1, T-(i-1)\delta_1]$.
			\end{lemma}
			\begin{proof}
				Note that $V_\bA \in \widetilde{\mathcal{H}}_{\bA ,i}$, $V_\bB \in \mathcal{H}_{\bB ,i}$, $\mu_\bB= p_{\bB} \delta_{\textup{D}} + g_{\bB} d\mathcal{L}^1 \in \mathcal{D}_\bB$ and $\widetilde{V}_{\bA,i} \in C(\mathbb{R}_{\geq 0};\mathbb{R}_{\geq 0})$. We observe that the right hand side of \eqref{NNNHJBx} is integrable with respect to $t$. We can easily verify that the unique solution $\widehat{V}_\bA (x,t)$ to \eqref{NNNHJBx} is given by \eqref{intergal_NNNHJBx}. For any $(x,t)\in \mathbb{R}_{\geq 0}\times[T-i \delta_1, T-(i-1)\delta_1]$, \eqref{intergal_NNNHJBx} shows that $ \widehat{V}_\bA (x,t)\ge 0$. {\color{black}The fact that $\widehat{V}_\bA (x,t) \in 
					C([T-i \delta_1, T-(i-1)\delta_1];C(\mathbb{R}_{\geq 0};\mathbb{R}_{\geq 0}))$} is obvious by the continuity of $V_\bA$, $V_\bB$ and $\widetilde{V}_{\bA,i}$. Finally, using $\Pi_\bA$ defined in \eqref{def_PiA}, we find that 
				$
				\frac{| \widehat{V}_\bA (x,t) |}{1+x} \le \frac{\lambda_\bB }{\rho}\left(1-e^{-\rho [T-(i-1)\delta_1-t]}\right)\left(\frac{2^{1+\nu}C_\bB }{\nu} + \Gamma_\bA \right)
				+e^{-\rho [T-(i-1)\delta_1-t]}\Pi_\bA 
				\le \Pi_\bA.
				$ {\color{black} Since $V_\bB \in\mathcal{H}_{\bB,i}$, it holds that $V_\bB^{-1}(0,t)=0$ for all $t\in[T-i \delta_1, T-(i-1)\delta_1]$. Note that $V_\bA (0,t) = 0$, we have $\mathbf{1}_{\{V_\bA (0,t) \leq V_\bB ^{-1}(0,t)\}} = 1$. Thus \eqref{intergal_NNNHJBx} implies $\widehat{V}_\bA(0,t) = 0$ since $r_\bA(0,t)=\widetilde{V}_{\bA,i}(0)=0$.} Hence, we have $ \widehat{V}_\bA \in\widetilde{\mathcal{H}}_{\bA ,i}$.
			\end{proof}
			To facilitate the proofs, we define
			\begingroup \begin{equation}\label{SA}
				S (x,u,V_\bA ,V_\bB ,g_\bB ):=\int_{V_\bA (x,u)}^{V_\bB ^{-1}(x,u)}\big[y-V_\bA (x,u)\big] g_\bB (y,u) d y \geq 0.
			\end{equation} \endgroup
			\begin{lemma}\label{contractivemap_extsience_PsiA}
				Suppose that Assumptions \ref{condition_running} and \ref{condition_terminal} hold, and that $0 < \delta_1 \le \frac{1}{4\lambda_{\bB}}$. Then the map $\mathcal{G}_{\bA,i} :\widetilde{\mathcal{H}}_{\bA ,i} \to \widetilde{\mathcal{H}}_{\bA ,i}$ is contractive. Hence, there is a unique solution $\overline{V}_\bA \in\widetilde{\mathcal{H}}_{\bA ,i}$, {\color{black}which is differentiable for a.e. $t$}, to the equation \eqref{NNHJBx} in $\mathbb{R}_{\geq 0}\times[T-i \delta_1, T-(i-1)\delta_1]$, subject to the corresponding terminal conditions in \eqref{eq. terminal VA,VB local}.
			\end{lemma}
			
			\begin{proof}
				Let $V_\bA, v_\bA\in\widetilde{\mathcal{H}}_{\bA ,i} $ and denote $\widehat{V}_\bA=\mathcal{G}_{\bA,i}(V_\bA)$, $\widehat{v}_\bA=\mathcal{G}_{\bA,i}(v_\bA)$. Fix any $(x,u) \in \mathbb{R}_{\ge 0} \times [T - i\delta_1, T - (i-1)\delta_1]$ and assume, without loss of generality, that $V_{\bA}(x,u) \le v_{\bA}(x,u)$. The arguments for the opposite inequality are analogous. Subtracting the respective equations for $\widehat{v}_{\bA}$ and $\widehat{V}_{\bA}$ in \eqref{NNNHJBx}, it suffices to estimate
				\begingroup \begin{align*}
					&I_0 := S (x,u,v_\bA ,V_\bB ,g_\bB )\mathbf{1}_{\{v_\bA (x,u) \leq V_\bB ^{-1}(x,u)\}}-S (x,u,V_\bA ,V_\bB ,g_\bB )\mathbf{1}_{\{V_\bA (x,u) \leq V_\bB ^{-1}(x,u)\}} =I_1+I_2,    \, \textup{with}  \\
					&I_1= \left[S (x,u,v_\bA ,V_\bB ,g_\bB )-S (x,u,V_\bA ,V_\bB ,g_\bB )\right]
					\mathbf{1}_{\{v_\bA (x,u) \leq V_\bB ^{-1}(x,u) \text{ and } V_\bA (x,u) \leq V_\bB ^{-1}(x,u)\}},      \text{ and }                               \\
					&I_2=S (x,u,v_\bA ,V_\bB ,g_\bB )\mathbf{1}_{\{v_\bA (x,u) \leq V_\bB ^{-1}(x,u) < V_\bA (x,u) \}} -S (x,u,V_\bA ,V_\bB ,g_\bB )\mathbf{1}_{\{V_\bA (x,u) \leq V_\bB ^{-1}(x,u) < v_\bA (x,u)\}}.
				\end{align*} \endgroup
				
				\noindent {\bf Case 1. $v_\bA (x,u) \leq V_\bB ^{-1}(x,u) \text{ and } V_\bA (x,u) \leq V_\bB ^{-1}(x,u)$:} In this case, we have $I_2=0$ and
				\begingroup \begin{align*}
					|I_1|  
					\leq |V_\bA (x,u) - v_\bA (x,u)| 
					+\int_{V_\bA (x,u)}^{v_\bA (x,u)}\,\big|\,v_\bA (x,u)-V_\bA (x,u)\,\big|\, g_\bB (y,u) d y 
					\leq2| V_\bA (x,u) - v_\bA (x,u)|.
				\end{align*} \endgroup
				
				\noindent {\bf Case 2. $v_\bA (x,u) \leq V_\bB ^{-1}(x,u) < V_\bA (x,u) $ or $V_\bA (x,u) \leq V_\bB ^{-1}(x,u) < v_\bA (x,u)$:} As $S\geq 0$, it holds that  $
				I_2
				\leq\,      \mathbf{1}_{\{v_\bA (x,u) \leq V_\bB ^{-1}(x,u) < V_\bA (x,u) \}} \int_{v_\bA (x,u)}^{V_\bB ^{-1}(x,u)}\big[V_\bA (x,u)-v_\bA (x,u)\big] g_\bB (y,u) d y  
				\leq  \,     | V_\bA (x,u)-v_\bA (x,u)|.     
				$ Similarly, $
				I_2                 
				\geq      - | v_\bA (x,u)-V_\bA (x,u)|.$
				
				Hence, by combining results in Cases 1-2, we have $
				\left|I_0\right| 
				\leq  2 | v_\bA (x,u)-V_\bA (x,u)|$ in $\mathbb{R}_{\geq 0}\times[T-i \delta_1, T-(i-1)\delta_1]$. By the integral representation \eqref{intergal_NNNHJBx}, it follows that $
				\frac{| \widehat{v}_\bA (x,u)-\widehat{V}_\bA (x,u)|}{1+x}
				\le 2\lambda_\bB  \delta_1\| v_\bA -V_\bA \|_{\mathcal{H},i}.$	
				Therefore, if $\delta_1\le \frac{1}{4\lambda_\bB}$, then the Banach fixed-point theorem shows that $\mathcal{G}_{\bA,i}$ admits a unique fixed point $\overline{V}_{\bA} = \mathcal{G}_{\bA,i}(\overline{V}_{\bA})$ in $\widetilde{\mathcal{H}}_{\bA,i}$, yielding the claimed unique solution.
			\end{proof}

			\subsubsection{The Property 
				$\overline{V}_\bA \in \mathcal{H}_{\bA ,i}$}\label{VA_BiLip}
			
			We keep fixing $i=1,2,\cdots,N_1$ and $V_\bB \in \mathcal{H}_{\bB ,i}$. In this subsection, we prove that the solution $\overline{V}_\bA $ to equation \eqref{NNHJBx} in $\mathbb{R}_{\geq 0}\times[T-i \delta_1, T-(i-1)\delta_1]$, obtained in \Cref{contractivemap_extsience_PsiA}, indeed belongs to $\mathcal{H}_{\bA ,i}$. It is sufficient to verify the bi-Lipschitz condition specified in the definition of $\mathcal{H}_{\bA ,i}$ in \eqref{def.HIn}.
			
			\begin{lemma}\label{BilipschitzVA}
				Suppose that Assumptions \ref{condition_running} and \ref{condition_terminal} hold, and that $
				\delta_1\le \frac{1}{4\lambda_\bB }$. Then the solution $\overline{V}_\bA$ to \eqref{NNHJBx} in $\mathbb{R}_{\geq0}\times [T-i \delta_1, T-(i-1)\delta_1]$, subject to the corresponding terminal conditions in \eqref{eq. terminal VA,VB local}, belongs to $\mathcal{H}_{\bA ,i}$. 
			\end{lemma}
			\begin{proof}
				Fix $z,x$ such that $0\leq x \leq z$. We divide the proof into two parts.
				
				\noindent{\bf Part 1. Lower bound of $\overline{V}_\bA (z,t)-\overline{V}_\bA (x,t) $:}  
				Let $t\in [T-i \delta_1, T-(i-1)\delta_1]$. Suppose $V_\bB ^{-1}(x,t) \geq \overline{V}_\bA (z,t)$. We note the relation
				\begingroup \begin{align}\label{VZmaxproperty}
					&\max\limits_{w\geq 0}\left\{ 
					\mathbf{1}_{\{w\leq V_\bB ^{-1}(z,t)\}}
					\int_{w}^{V_\bB ^{-1}(z,t)}\big[y-\overline{V}_\bA (z,t)\big] g_\bB (y,t) d y\right\}\nonumber\\
					&=\mathbf{1}_{\{\overline{V}_\bA (z,t) \leq V_\bB ^{-1}(z,t)\}}
					\int_{\overline{V}_\bA (z,t)}^{V_\bB ^{-1}(z,t)}\big[y-\overline{V}_\bA (z,t)\big] g_\bB (y,t) d y
				\end{align}\endgroup
				and the inequality $V_\bB ^{-1}(z,t)\ge V_\bB ^{-1}(x,t)\geq \overline{V}_\bA (z,t)$ since $V_\bB\in\mathcal{H}_{\bB,i}$. Hence, choosing $w=\overline{V}_\bA (x,t) \land V_\bB ^{-1}(x,t)$ for the maximand of the left hand side of \eqref{VZmaxproperty}, we have 
				\begingroup \begin{align*}
					\rho \overline{V}_\bA (z,t)-\frac{\partial \overline{V}_\bA (z,t) }{\partial t} 
					\ge\,& \mathbf{1}_{\{\overline{V}_\bA (x,t) < V_\bB ^{-1}(x,t) 
						\}}
					\lambda_{\bB}
					\int_{\overline{V}_\bA (x,t) }
					^{V_\bB ^{-1}(z,t)}\big[y-\overline{V}_\bA (z,t)\big] g_\bB (y,t) d y+r_\bA (z,t)\nonumber\\
					&+\mathbf{1}_{\{V_\bB ^{-1}(x,t)\le \overline{V}_\bA (x,t) \le V_\bB ^{-1}(z,t)  
						\}}
					\lambda_{\bB}
					\int_{V_\bB ^{-1}(x,t) }
					^{V_\bB ^{-1}(z,t)}\big[y-\overline{V}_\bA (z,t)\big] g_\bB (y,t) d y\nonumber\\
					\ge\,& \mathbf{1}_{\{\overline{V}_\bA (x,t) 
						\leq V_\bB ^{-1}(x,t)\}}
					\lambda_{\bB}
					\int_{\overline{V}_\bA (x,t)}^{V_\bB ^{-1}(x,t)}\big[y-\overline{V}_\bA (z,t)\big] g_\bB (y,t) d y+r_\bA (z,t).
				\end{align*}\endgroup
				When $V_{\bB}^{-1}(x,t) < \overline{V}_{\bA}(z,t)$, the same lower bound follows directly from~\eqref{NNHJBx}:
				\begingroup 
					\begin{align}
						\rho \overline{V}_\bA (z,t)-\frac{\partial \overline{V}_\bA (z,t) }{\partial t}
						\ge r_\bA (z,t)
						\ge\mathbf{1}_{\{\overline{V}_\bA (x,t) \leq V_\bB ^{-1}(x,t)\}}
						\lambda_{\bB}
						\int_{\overline{V}_\bA (x,t)}^{V_\bB ^{-1}(x,t)}\big[y-\overline{V}_\bA (z,t)\big] g_\bB (y,t) d y+r_\bA (z,t).\label{subvalueVzt}
				\end{align}%
				\endgroup
				
				Combining these cases, \eqref{subvalueVzt} still holds in $\mathbb{R}_{\geq0}\times [T-i \delta_1, T-(i-1)\delta_1]$ generally. Subtracting the equation of $\overline{V}_\bA (x,t)$ in \eqref{NNHJBx} from $\eqref{subvalueVzt}$ and denoting by $\Delta \overline{V}_\bA (z,x,t):= \overline{V}_\bA (z,t)- \overline{V}_\bA (x,t)$, Assumption \ref{condition_running} yields
				\begingroup \begin{align}\label{1939}
						\rho \Delta \overline{V}_\bA (z,x,t)-\frac{\partial \Delta \overline{V}_\bA (z,x,t)}{\partial t} 
						\ge -\mathbf{1}_{\{\overline{V}_\bA (x,t) \leq V_\bB ^{-1}(x,t)\}}
						\lambda_{\bB}
						\int_{\overline{V}_\bA (x,t)}^{V_\bB ^{-1}(x,t)} \Delta \overline{V}_\bA (z,x,t) 
						g_\bB (y,t) d y+\gamma_\bA (z-x). 
				\end{align}%
				\endgroup
				{\color{black}Integrating with an appropriate integrating factor yields:
					\begingroup \begin{align*}
							\Delta \overline{V}_\bA (z,x,t)
							\ge & \exp\left\{\!- \int_t^{T-(i-1) \delta_1}   \rho
							+
							\mathbf{1}_{\{\overline{V}_\bA (x,u) \leq V_\bB ^{-1}(x,u)\}}
							\!\lambda_\bB\!\!
							\int_{\overline{V}_{\bA}(x,u)}^{V_\bB ^{-1}(x,u)} \! g_\bB (y,u) d y  du\! \right\} \!\!
							\Delta \overline{V}_\bA \!(z,x,T-(i-1) \delta_1)     \!        \\
							& + \gamma_\bA (z-x) \int_t^{T-(i-1) \delta_1} \!\!\exp\left(-\int_t^{s} \rho+ \mathbf{1}_{\{\overline{V}_\bA (x,u) \leq V_\bB ^{-1}(x,u)\}}
							\lambda_\bB\!
							\int_{\overline{V}_{\bA}(x,u)}^{V_\bB ^{-1}(x,u)} \!\!g_\bB (y,u) d y   du\right)ds  \!\!
					\end{align*}%
					\endgroup}%
				for $t\in [T-i \delta_1, T-(i-1)\delta_1]$. Hence, by the definition of the terminal condition in \eqref{eq. terminal VA,VB local}, we see that $\Delta \overline{V}_\bA (z,x,T-(i-1) \delta_1)\geq0$ which implies $\Delta \overline{V}_\bA (z,x,t)\geq 0$. Thus, we use \eqref{1939} to deduce that
				\begingroup \begin{align*}
						\rho \Delta \overline{V}_\bA (z,x,t)-\frac{\partial \Delta \overline{V}_\bA (z,x,t) }{\partial t}
						\geq & -\mathbf{1}_{\{\overline{V}_\bA (x,t) \leq V_\bB ^{-1}(x,t)\}}\lambda_\bB \Delta \overline{V}_\bA (z,x,t) \int_{\overline{V}_{\bA}(x,t)}^{V_\bB ^{-1}(x,t)} g_\bB (y,t) d y   
						+\gamma_\bA (z-x) \\
						\geq & -\lambda_\bB \Delta \overline{V}_\bA (z,x,t) +\gamma_\bA (z-x)    
				\end{align*}%
				\endgroup
				for $t\in [T-i \delta_1, T-(i-1)\delta_1]$. Thus, integrating this inequality gives
				\begingroup \begin{align*}
						 &\Delta \overline{V}_\bA (z,x,t) \\
						 &\geq e^{-(\rho+\lambda_\bB )[T-(i-1) \delta_1-t]}
						\!\Delta\overline{V}_\bA (z,x,T-(i-1) \delta_1) 
						+ \frac{1}{\rho+\lambda_\bB }\!
						\left\{1-e^{-(\rho+\lambda_\bB )[T-(i-1) \delta_1-t]}\right\}\!\gamma_\bA (z-x).\hspace{-10pt}
				\end{align*}%
				\endgroup
				Note that 
				$\Delta\overline{V}_\bA (z,x,T-(i-1) \delta_1)\geq k_\bA (z-x)$ by the definition of the terminal condition in \eqref{eq. terminal VA,VB local}. We have
				\begingroup \begin{equation}\label{value.difference.integral.2}
					\Delta \overline{V}_\bA (z,x,t) \geq \left[k_\bA e^{-(\rho+\lambda_\bB )[T-(i-1) \delta_1-t]} + \frac{\gamma_\bA }{\rho+\lambda_\bB }\left\{1-e^{-(\rho+\lambda_\bB )[T-(i-1) \delta_1-t]}\right\}\right](z-x).
				\end{equation}\endgroup
				\begingroup 
				We choose \begin{equation}\label{def_kA} 
					k_\bA  := \frac{\gamma_\bA }{\rho+\lambda_\bB } \wedge l_\bA ,
				\end{equation}\endgroup
				such that \eqref{value.difference.integral.2} is reduced to $
				\overline{V}_\bA (z,t)-\overline{V}_\bA (x,t) \geq  k_\bA (z-x)$.

				\noindent{\bf Part 2. Upper bound of $\overline{V}_\bA (z,t)-\overline{V}_\bA (x,t) $:} Subtracting $\overline{V}_\bA (z,t)$ and $\overline{V}_\bA (x,t) $, it is sufficient to estimate the following term:
				\begingroup \begin{align*}
					I_0 :=&\,  S (z,t,\overline{V}_\bA ,V_\bB ,g_\bB )\mathbf{1}_{\{\overline{V}_\bA (z,t) \leq V_\bB ^{-1}(z,t)\}}-S (x,t,\overline{V}_\bA ,V_\bB ,g_\bB )\mathbf{1}_{\{\overline{V}_\bA (x,t) \leq V_\bB ^{-1}(x,t)\}}                      \\ 
					=    & \underbrace{\left[S (z,t,\overline{V}_\bA ,V_\bB ,g_\bB )-S (x,t,\overline{V}_\bA ,V_\bB ,g_\bB )\right]\mathbf{1}_{\{\overline{V}_\bA (z,t) \leq V_\bB ^{-1}(z,t) \text{ and }\overline{V}_\bA (x,t) \leq V_\bB ^{-1}(x,t)\}}}_{I_1}                    \\[-1.5mm]
					& + \underbrace{S (z,t,\overline{V}_\bA ,V_\bB ,g_\bB )\mathbf{1}_{\{\overline{V}_\bA (z,t) \leq V_\bB ^{-1}(z,t)\text{ and }\overline{V}_\bA (x,t) > V_\bB ^{-1}(x,t)\}} }_{I_2}  \\[-1.5mm]
					& -S (x,t,\overline{V}_\bA ,V_\bB ,g_\bB )\mathbf{1}_{\{\overline{V}_\bA (x,t) \leq V_\bB ^{-1}(x,t) \text{ and } \overline{V}_\bA (z,t) > V_\bB ^{-1}(z,t)\}}.         
				\end{align*} \endgroup
				
				\noindent{\bf Case 2A. $\overline{V}_\bA (z,t) \leq V_\bB ^{-1}(z,t) \text{ and }\overline{V}_\bA (x,t) \leq V_\bB ^{-1}(x,t)$:} We have $I_2=0$ and
				\begingroup \begin{align*}
					I_0 =I_1 
					=&\, \int_{\overline{V}_\bA (z,t)}^{V_\bB^{-1}(z,t)}\big[\overline{V}_\bA (x,t)-\overline{V}_\bA (z,t)\big] g_\bB (y,t) d y
					+\int_{V_\bB^{-1}(x,t)}^{V_\bB^{-1}(z,t)}\big[y-\overline{V}_\bA (x,t)\big] g_\bB (y,t) d y\\
					&- \int_{\overline{V}_\bA (x,t)}^{\overline{V}_\bA (z,t)}\big[y-\overline{V}_\bA (x,t)\big] g_\bB (y,t) d y.
				\end{align*}\endgroup
				Using the monotonicity of $\overline{V}_\bA $ and $V_\bB^{-1}$ in the state variable, and the property of $g_\bB$ in \eqref{DA}, it holds 
				\begingroup \begin{equation*}
					I_0 = I_1 \le  \int_{V_\bB^{-1}(x,t)}^{V_\bB^{-1}(z,t)}\big[y-\overline{V}_\bA (x,t)\big] g_\bB (y,t) d y
					\le C_\bB  \left[V_\bB^{-1}(z,t)-V_\bB^{-1}(x,t)\right] 
					\le \frac{ C_\bB }{k_\bB } (z-x).
				\end{equation*}\endgroup
				
				\noindent{\bf Case 2B. $\overline{V}_\bA (z,t) \leq V_\bB ^{-1}(z,t)\text{ and }\overline{V}_\bA (x,t) > V_\bB ^{-1}(x,t)$:} By the monotonicity of $\overline{V}_\bA $ in the state variable, it follows that $
				I_0 = I_2 \leq\, V_\bB ^{-1}(z,t)-\overline{V}_\bA (z,t) 
				\leq  V_\bB ^{-1}(z,t)-\overline{V}_\bA (x,t) 
				\leq  V_\bB ^{-1}(z,t)-V_\bB ^{-1}(x,t) 
				\leq \frac{1}{k_\bB } (z-x).$

				Combining results in Cases 2A and 2B, and noting that $C_\bB  > 1$, we conclude that $
				I_0 \leq \frac{C_\bB }{k_\bB } (z-x)$. Subtracting the equations satisfied by $\overline{V}_\bA (z,t)$ and $\overline{V}_\bA (x,t)$, we use Assumption \ref{condition_running} to deduce $
				\rho 	\Delta \overline{V}_\bA (z,x,t)
				-\frac{\partial 	\Delta \overline{V}_\bA (z,x,t)}{\partial t}
				\leq \lambda_{\bB}I_0+r_\bA(z,t)-r_\bA(x,t)
				\le \left(\frac{\lambda_{\bB}C_\bB }{k_\bB }+
				\Gamma_\bA\right) (z-x).$ Since the corresponding terminal condition in \eqref{eq. terminal VA,VB local} satisfies $
				\overline{V}_\bA (z,T-(i-1) \delta_1)-\overline{V}_\bA (x,T-(i-1) \delta_1) \le K_\bA  (z-x)$, integrating with an integrating factor yields
				\begingroup \begin{equation*}
					\overline{V}_\bA (z,t)-\overline{V}_\bA (x,t) 
					\le \left\{1-e^{-\rho [T-(i-1)\delta_1-t]}\right\}
					\left(\frac{\lambda_{\bB}C_\bB }{\rho k_\bB }+
					\frac{\Gamma_\bA }{\rho}\right)(z-x)+e^{-\rho [T-(i-1)\delta_1-t]}K_\bA  (z-x).
				\end{equation*}\endgroup
				We set
				\begingroup \begin{equation}\label{def_KA}
					K_\bA :=\left(\frac{\lambda_{\bB}C_\bB }{\rho k_\bB }+
					\frac{\Gamma_\bA }{\rho}\right) \vee L_\bA ,
				\end{equation}\endgroup
				to obtain $
				\overline{V}_\bA (z,t)-\overline{V}_\bA (x,t) \le K_\bA  (z-x)$ for $t\in [T-i \delta_1, T-(i-1)\delta_1]$ and $0\leq x \leq z$. 
			\end{proof}

			\subsubsection{Existence and Uniqueness of Global Solution to Equations \eqref{NHJBx}-\eqref{NHJBy}}
			In this section, we establish the existence of unique global-in-time solution of the system \eqref{NHJBx}-\eqref{NHJBy} in $\mathbb{R}_{>0}\times[0, T]$. For any $i=1,2,\cdots,N_1$ and $V_\bB \in \mathcal{H}_{\bB ,i}$, Lemmas \ref{contractivemap_extsience_PsiA} and \ref{BilipschitzVA} assert that the equation \eqref{NNHJBx} admits a unique solution $\overline{V}_\bA \in \mathcal{H}_{\bA ,i}$. It then uniquely determines  $\overline{V}_\bB  \in \mathcal{H}_{\bB ,i}$ solving equation \eqref{NNHJBy}, that is:
			\begin{lemma}\label{BilipschitzVB}
				Suppose that Assumptions \ref{condition_running} and \ref{condition_terminal} hold, and that $
				\delta_1\le \frac{1}{4\lambda_\bA}\wedge \frac{1}{4\lambda_\bB}$, then there is a unique solution $\overline{V}_\bB \in \mathcal{H}_{\bB ,i}$, {\color{black}which is differentiable for a.e. $t$}, to equation \eqref{NNHJBy} in $\mathbb{R}_{\geq0}\times [T-i \delta_1, T-(i-1)\delta_1]$, subject to the corresponding terminal condition given in \eqref{eq. terminal VA,VB local}.
			\end{lemma}
			\noindent The proof follows the same arguments as Lemmas \ref{selfmap_PsiA}, \ref{contractivemap_extsience_PsiA}, and \ref{BilipschitzVA}, and is therefore omitted. Therefore, following the pattern of the constants defined in \eqref{def_PiA}, \eqref{def_kA}, and \eqref{def_KA}, we set
			\begingroup \begin{equation}\label{def_PiB,KB,kB}
				\Pi_\bB :=\left[\frac{\lambda_\bA }{\rho}\left(\frac{2^{1+\nu}C_\bA }{\nu} + 
				\Gamma_\bB\right)\right]\vee L_\bB ,\quad
				k_\bB  := \frac{\gamma_\bB }{\rho+\lambda_\bA } \wedge l_\bB,\quad
				K_\bB =\left(\frac{\lambda_{\bA}C_\bA }{\rho k_\bA }+
				\frac{\Gamma_\bB }{\rho}\right) \vee L_\bB.
			\end{equation} \endgroup
			Thus, the map $\Psi_i:\mathcal{H}_{\bB ,i}\ni V_\bB   \longmapsto \overline{V}_\bB =\Psi_i(V_\bB )$ defined through equations \eqref{NNHJBx}-\eqref{NNHJBy} is a self-map. 
			
			To prove that $\Psi_i:\mathcal{H}_{\bB ,i}\to \mathcal{H}_{\bB ,i}$ is contractive, we first establish a Lipschitz-type estimate for their inverses.

			\begin{lemma}\label{lem Lip inv of VA and VB}
				Suppose that Assumptions \ref{condition_running} and \ref{condition_terminal} hold, and that $
				\delta_1\le \frac{1}{4\lambda_\bA}\wedge \frac{1}{4\lambda_\bB}$. For any $t\in [T-i\delta_1, T-(i-1)\delta_1]$ and $V^1_\bB , V^2_\bB  \in \mathcal{H}_{\bB ,i}$, the corresponding solutions $(\overline{V}_\bA^j , \overline{V}_\bB^j )\in \mathcal{H}_{\bA ,i}\times \mathcal{H}_{\bB ,i}$ to equations \eqref{NNHJBx}-\eqref{eq. terminal VA,VB local} with $V_\bB = V^j_\bB$ and $j=1,2$ satisfy
				\begingroup 	\begin{equation}\label{continuityofinverseVA}
					\| (\overline{V}_\bI^1 )^{-1}(\bigcdot,t)-(\overline{V}_\bI^2 )^{-1}(\bigcdot,t) \|_{0,1/w} \le \frac{1+k_\bI }{(k_\bI )^2}  \| \overline{V}_\bI^1 (\bigcdot,t)-\overline{V}_\bI^2 (\bigcdot,t) \|_{0,1/w},\hspace{5pt} \text{for $\bI\in\{\bA,\bB\}$}.
				\end{equation}\endgroup
			\end{lemma}

			\begin{proof}
				By Lemma \ref{BilipschitzVA} and the definition in \eqref{def.HIn}, we have $
				\pig| (\overline{V}_\bA^1 )^{-1}(y,t)-(\overline{V}_\bA^2 )^{-1}(y,t) \pig| 
				\le \frac{1}{k_\bA } \pig| \overline{V}_\bA^1 \pig((\overline{V}_\bA^1 )^{-1}(y,t),t\pig)
				-\overline{V}_\bA^1 \pig((\overline{V}_\bA^2 )^{-1}(y,t),t\pig) \pig|$ in $\mathbb{R}_{\geq0}\times [T-i \delta_1, T-(i-1)\delta_1]$. Hence, the fact that $\overline{V}_\bA^1 \pig((\overline{V}_\bA^1 )^{-1}(y,t),t\pig)=y=\overline{V}_\bA^2 \pig((\overline{V}_\bA^2 )^{-1}(y,t),t\pig)$ deduces
				\begingroup 	\begin{equation*}
					\pig| (\overline{V}_\bA^1 )^{-1}(y,t)-(\overline{V}_\bA^2 )^{-1}(y,t) \pig| \le \frac{1}{k_\bA } \pig| \overline{V}_\bA^2 \pig((\overline{V}_\bA^2 )^{-1}(y,t),t\pig)
					-\overline{V}_\bA^1 \pig((\overline{V}_\bA^2 )^{-1}(y,t),t\pig) \pig|.
				\end{equation*}\endgroup
				Since $\overline{V}_\bA^2\in \mathcal{H}_{\bA ,i}$, we have $\overline{V}_\bA^2 (0,t)\ge 0$ and hence $\overline{V}_\bA^2 (x,t) \ge k_\bA  x+\overline{V}_\bA^2 (0,t) \geq k_\bA  x $. In other words, $(\overline{V}_\bA^2 )^{-1}(y,t) \le \frac{1}{k_\bA } y$.
				Therefore, in $\mathbb{R}_{\geq0}\times [T-i \delta_1, T-(i-1)\delta_1]$, we have
				\begingroup 	\begin{align}\label{inequality:Vinv}
					\nonumber		\frac{\pig| (\overline{V}_\bA^1 )^{-1}(y,t)-(\overline{V}_\bA^2 )^{-1}(y,t) \pig|}{1+y} 
					&\le \frac{1+k_\bA }{k_\bA } \frac{\pig| (\overline{V}_\bA^1 )^{-1}(y,t)-(\overline{V}_\bA^2 )^{-1}(y,t) \pig|}{1+(\overline{V}_\bA^2 )^{-1}(y,t)} \nonumber\\
					&\le \frac{1+k_\bA }{(k_\bA )^2} \frac{\pig| \overline{V}_\bA^2 \pig((\overline{V}_\bA^2 )^{-1}(y,t),t\pig)-\overline{V}_\bA^1 \pig((\overline{V}_\bA^2 )^{-1}(y,t),t\pig) \pig|}{1+(\overline{V}_\bA^2 )^{-1}(y,t)} \nonumber\\
					&\le \frac{1+k_\bA }{(k_\bA )^2}  \| \overline{V}_\bA^1 (\bigcdot,t)-\overline{V}_\bA^2 (\bigcdot,t) \|_{0,1/w},
				\end{align}\endgroup
				which leads to \eqref{continuityofinverseVA} for $\bI=\bA$; the same argument applies to $\bI=\bB$.
			\end{proof}

			By the preceding lemma, we establish the existence and uniqueness of global solution to \eqref{NHJBx}-\eqref{NHJBy}.

			\begin{proposition}\label{existence_NHJBxy}
				Suppose that Assumptions \ref{condition_running} and  \ref{condition_terminal} hold. Then there is a unique solution $(V_\bA^*, V_\bB^*) \in \mathcal{H}_\bA \times \mathcal{H}_\bB$ to equations \eqref{NHJBx}-\eqref{NHJBy} in $\mathbb{R}_{\geq 0}\times[0, T]$, subject to the terminal conditions in \eqref{Nboundary}.
			\end{proposition}
			\begin{proof}
				
				\noindent\textbf{Part 1. Estimate of integral terms for fixed $i=1,2,\cdots,N_1$:} \sloppy Fix $i=1,2,\cdots,N_1$, $x\ge 0$ and $u\in [T-i\delta_1, T-(i-1)\delta_1]$. Let $V^1_\bB , V^2_\bB  \in \mathcal{H}_{\bB ,i}$ be two inputs for $\Psi_i$ and denote the corresponding solutions to \eqref{NNHJBx}-\eqref{eq. terminal VA,VB local} with $V_\bB=V^j_\bB$ by $(\overline{V}_\bA^j , \overline{V}_\bB^j)\in\mathcal{H}_{\bA ,i}\times \mathcal{H}_{\bB ,i}$ for $j=1,2$. Without loss of generality, we assume $\overline{V}_\bA^2 (x,u) \le \overline{V}_\bA^1(x,u)$; the reverse case is analogous. To estimate $\overline{V}_\bA^1 (x,u)-\overline{V}_\bA^2 (x,u)$, we consider
				\begingroup \begin{align*}
					I_0:=&\, S (x,u,\overline{V}_\bA^1 ,V^1_\bB ,g_\bB )\mathbf{1}_{\{\overline{V}_\bA^1 (x,u) \leq (V^1_\bB )^{-1}(x,u)\}}-S (x,u,\overline{V}_\bA^2 ,V^2_\bB ,g_\bB )\mathbf{1}_{\{\overline{V}_\bA^2 (x,u) \leq (V^2_\bB )^{-1}(x,u)\}}  
				\end{align*}\endgroup    
				We split
				\begingroup \begin{align*}
					I_0 = &\, \underbrace{\left[S (x,u,\overline{V}_\bA^1 ,V^1_\bB ,g_\bB )-S (x,u,\overline{V}_\bA^2 ,V^2_\bB ,g_\bB )\right]\mathbf{1}_{\{\overline{V}_\bA^1 (x,u) \leq (V^1_\bB )^{-1}(x,u) \text{ and }\overline{V}_\bA^2 (x,u) \leq (V^2_\bB )^{-1}(x,u)\}}}_{I_1} \\[-2.5mm]
					& +\underbrace{S (x,u,\overline{V}_\bA^1 ,V^1_\bB ,g_\bB )\mathbf{1}_{\{\overline{V}_\bA^1 (x,u) \leq (V^1_\bB )^{-1}(x,u)\text{ and }\overline{V}_\bA^2 (x,u) > (V^2_\bB )^{-1}(x,u)\}}}_{I_2}                                                                                    \\[-2.5mm]
					& -\underbrace{S (x,u,\overline{V}_\bA^2 ,V^2_\bB ,g_\bB )\mathbf{1}_{\{\overline{V}_\bA^1 (x,u) > (V^1_\bB )^{-1}(x,u)\text{ and }\overline{V}_\bA^2 (x,u) \leq (V^2_\bB )^{-1}(x,u)\}}}_{I_3}.                                                                                   
				\end{align*}\endgroup    
				
				\noindent {\bf Case 1A. $\overline{V}_\bA^1 (x,u) \leq (V^1_\bB )^{-1}(x,u) \text{ and }\overline{V}_\bA^2 (x,u) \leq (V^2_\bB )^{-1}(x,u)$:} In this regime,
				\begingroup \begin{align*}
					I_1 
					=    & \underbrace{\int_{\overline{V}_\bA^1 (x,u)}^{(V^1_\bB )^{-1}(x,u)}\big[\overline{V}_\bA^2 (x,u)-\overline{V}_\bA^1 (x,u)\big] g_\bB (y,u) d y}_{I_{1,1}}+\underbrace{\int_{(V^2_\bB )^{-1}(x,u)}^{(V^1_\bB )^{-1}(x,u)}\big[y-\overline{V}_\bA^2 (x,u)\big] g_\bB (y,u) d y}_{I_{1,2}} \\[-3mm]
					& -\underbrace{\int_{\overline{V}_\bA^2 (x,u)}^{\overline{V}_\bA^1 (x,u)}\big[y-\overline{V}_\bA^2 (x,u)\big]g_\bB (y,u) d y}_{I_{1,3}}.         
				\end{align*}\endgroup
				It is evident that both $| I_{1,1} |$ and $| I_{1,3} |$ are bounded by $\pig| \overline{V}_\bA^2 (x,u)-\overline{V}_\bA^1 (x,u) \pig|$. For $| I_{1,2} |$, suppose $(V^2_\bB )^{-1}(x,u)\leq (V^1_\bB )^{-1}(x,u)$, then the boundedness of $g_\bB$ implies $0\leq\int_{(V^2_\bB )^{-1}(x,u)}^{(V^1_\bB )^{-1}(x,u)}\big[y-\overline{V}_\bA^2 (x,u)\big] g_\bB (y,u) d y \leq C_\bB \left| (V^2_\bB )^{-1}(x,u)-(V^1_\bB )^{-1}(x,u) \right|.$
				Otherwise, there are two cases to study. One of them is $  (V^1_\bB)^{-1}(x,u)
				<
				\overline{V}_\bA^2(x,u)
				\le
				(V^2_\bB)^{-1}(x,u),$ but this contradicts the assumption of this case and that $\overline{V}_\bA^2 (x,u) \le \overline{V}_\bA^1(x,u)$. The remaining possibility is $\overline{V}_\bA^2(x,u)
				\le
				(V^1_\bB)^{-1}(x,u)
				<
				(V^2_\bB)^{-1}(x,u)$ which implies $0\leq\int^{(V^2_\bB )^{-1}(x,u)}_{(V^1_\bB )^{-1}(x,u)}\big[y-\overline{V}_\bA^2 (x,u)\big] g_\bB (y,u) d y \leq C_\bB \left| (V^2_\bB )^{-1}(x,u)-(V^1_\bB )^{-1}(x,u) \right|.$ Consequently, we have $| I_{1,2} | 
				\le C_\bB \left| (V^2_\bB )^{-1}(x,u)-(V^1_\bB )^{-1}(x,u) \right|$. Combining these estimates, we conclude: $|I_0| = |I_1| \leq 2\pig| \overline{V}_\bA^2 (x,u)-\overline{V}_\bA^1 (x,u) \pig| + C_\bB \left| (V^2_\bB )^{-1}(x,u)-(V^1_\bB )^{-1}(x,u) \right|$.
				
				\noindent {\bf Case 1B. $\overline{V}_\bA^1 (x,u) \leq (V^1_\bB )^{-1}(x,u)\text{ and }\overline{V}_\bA^2 (x,u) > (V^2_\bB )^{-1}(x,u)$:} By the assumption that $\overline{V}_\bA^2 (x,u) \le \overline{V}_\bA^1(x,u)$, we have $
				0\leq I_0=I_2 \leq  (V^1_\bB )^{-1}(x,u)-\overline{V}_\bA^1(x,u)
				\leq | (V^2_\bB )^{-1}(x,u)-(V^1_\bB )^{-1}(x,u) |.$
				
				\noindent {\bf Case 1C. $\overline{V}_\bA^1 (x,u) > (V^1_\bB )^{-1}(x,u)\text{ and }\overline{V}_\bA^2 (x,u) \leq (V^2_\bB )^{-1}(x,u)$:} In this case, we decompose $0\leq I_3 =I_{3,1}+I_{3,2}+I_{3,3}$, where 
				\begingroup \begin{align*}
					I_{3,1}= &\,  \mathbf{1}_{\{\overline{V}_\bA^2 (x,u) > (V^1_\bB )^{-1}(x,u)\}}
					\int_{\overline{V}_\bA^2 (x,u)}^{(V^2_\bB )^{-1}(x,u)}\big[y-\overline{V}_\bA ^2(x,u)\big] g_\bB (y,u) d y;                                                                         \\[-1mm]
					I_{3,2}= &\, \mathbf{1}_{\{\overline{V}_\bA^2 (x,u) \leq (V^1_\bB )^{-1}(x,u)\text{ and }(V^2_\bB )^{-1}(x,u) \leq \overline{V}_\bA^1 (x,u)\}}\int_{\overline{V}_\bA^2 (x,u)}^{(V^2_\bB )^{-1}(x,u)}\big[y-\overline{V}_\bA ^2(x,u)\big] g_\bB (y,u) d y; \\
					I_{3,3}= &\, \mathbf{1}_{\{\overline{V}_\bA^2 (x,u) \leq (V^1_\bB )^{-1}(x,u)\text{ and }(V^2_\bB )^{-1}(x,u) > \overline{V}_\bA^1 (x,u)\}}\int_{\overline{V}_\bA^2 (x,u)}^{(V^2_\bB )^{-1}(x,u)}\big[y-\overline{V}_\bA ^2(x,u)\big] g_\bB (y,u) d y.  
				\end{align*}\endgroup
				By estimating the integrands, it holds that $
				0\leq I_{3,1} 
				\leq | (V^2_\bB )^{-1}(x,u)-(V^1_\bB )^{-1}(x,u) |$ and $0\leq I_{3,2} \leq  | \overline{V}_\bA^2 (x,u)-\overline{V}_\bA^1 (x,u) |$. Next, we split
				\begingroup 	\begin{align*}
					0\leq I_{3,3} =\, & \Bigg\{\int_{(V^1_\bB )^{-1}(x,u)}^{(V^2_\bB )^{-1}(x,u)}\big[y-\overline{V}_\bA^2(x,u)\big] g_\bB (y,u) d y 
					-\int_{(V^1_\bB )^{-1}(x,u)}^{\overline{V}_\bA^1 (x,u)}\big[y-\overline{V}_\bA^2(x,u)\big] g_\bB (y,u) d y \\
					& \,\,\,+\int_{\overline{V}_\bA^2 (x,u)}^{\overline{V}_\bA^1 (x,u)}\big[y-\overline{V}_\bA^2(x,u)\big] g_\bB (y,u) d y\Bigg\}  \mathbf{1}_{\{ \overline{V}_\bA^2 (x,u) \leq (V^1_\bB )^{-1}(x,u) < \overline{V}_\bA^1 (x,u) < (V^2_\bB )^{-1}(x,u)\}}  .
				\end{align*}\endgroup
				It then follows
				\begingroup 	\begin{align*}
					0\leq I_{3,3}  
					\leq          \,  & \Bigg\{\int_{(V^1_\bB )^{-1}(x,u)}^{(V^2_\bB )^{-1}(x,u)}y g_\bB(y,u) d y +\int_{\overline{V}_\bA^2 (x,u)}^{\overline{V}_\bA^1 (x,u)}\big[\overline{V}_\bA^1 (x,u)-\overline{V}_\bA^2(x,u)\big] g_\bB (y,u) d y\Bigg\}    \\
					& \,\cdot \mathbf{1}_{\{ \overline{V}_\bA^2 (x,u) \leq (V^1_\bB )^{-1}(x,u) < \overline{V}_\bA^1 (x,u) < (V^2_\bB )^{-1}(x,u)\}}                                                                                                               \\
					\leq    \,        & C_\bB | (V^2_\bB )^{-1}(x,u)-(V^1_\bB )^{-1}(x,u) |+| \overline{V}_\bA^2 (x,u)-\overline{V}_\bA^1 (x,u)|.                    
				\end{align*}\endgroup
				Therefore
				$|I_0| = |I_3| \leq (1+C_\bB )| (V^2_\bB )^{-1}(x,u)-(V^1_\bB )^{-1}(x,u) |+2| \overline{V}_\bA^2 (x,u)-\overline{V}_\bA^1 (x,u)|$.
				
				Combining the results in Cases 1A-1C, Lemma \ref{lem Lip inv of VA and VB} implies $ \frac{\left| I_0 \right|}{1+x} \le 2 \| \overline{V}_\bA^2 (\bigcdot,u)-\overline{V}_\bA^1 (\bigcdot,u) \|_{0,1/w} 
				+ \frac{(1+C_\bB )(1+k_\bB )}{(k_\bB )^2}  \| V^1_\bB (\bigcdot,u)-V^2_\bB (\bigcdot,u) \|_{0,1/w}.$

				\noindent\textbf{Part 2. Construction of global solution:}
				Consider $i=1$. In $\mathbb{R}_{\geq0}\times [T- \delta_1, T]$, subtracting the integral forms of solutions $\overline{V}_\bA^2 (x,t)$ and $\overline{V}_\bA^1 (x,t)$ to \eqref{NNHJBx} in $\mathbb{R}_{\geq0}\times [T- \delta_1, T]$ gives
				\begingroup  \begin{align*}
						\| \overline{V}_\bA^2\! -\overline{V}_\bA^1  \|_{\mathcal{H},1} 
						&\le  \!\sup_{\mathbb{R}_{\geq0}\times [T- \delta_1, T]} \!\int_t^{T}  \lambda_\bB e^{-\rho (u-t)} \!\frac{\left|I_0 \right|}{1+x} du \\
						&\leq \lambda_\bB \delta_1 \!\!\left[\!2 \| \overline{V}_\bA^2 -\overline{V}_\bA^1  \|_{\mathcal{H},1}+\frac{(1+C_\bB )(1+k_\bB )}{(k_\bB )^2}  \| V^1_\bB -V^2_\bB  \|_{\mathcal{H},1} \!\right], \!
				\end{align*}%
				\endgroup
				which yields
				\begingroup \begin{equation}\label{2250a}
					\| \overline{V}_\bA^2 -\overline{V}_\bA^1  \|_{\mathcal{H},1} \le  \frac{\lambda_\bB \delta_1(1+C_\bB )
						(1+k_\bB )}{(1-2\lambda_\bB\delta_1)(k_\bB )^2}  
					\| V^1_\bB -V^2_\bB  \|_{\mathcal{H},1}.
				\end{equation}\endgroup
				Similarly, subtracting the integral forms of solutions $\overline{V}_\bB^2 (x,t)$ and $\overline{V}_\bB^1 (x,t)$ to \eqref{NNHJBy} in $\mathbb{R}_{\geq0}\times [T- \delta_1, T]$ gives 
				\begingroup \begin{equation}\label{2250}
					\| \overline{V}_\bB^2 -\overline{V}_\bB^1  \|_{\mathcal{H},1} \le  \frac{\lambda_\bA \delta_1(1+C_\bA )(1+k_\bA )}{(1-2\lambda_\bA\delta_1)(k_\bA )^2}  
					\| \overline{V}_\bA^2 -\overline{V}_\bA^1  \|_{\mathcal{H},1}.
				\end{equation}\endgroup
				Choosing
				\begingroup \begin{equation}\label{def_deltat}
					\delta_1 := \left(\frac{1}{4\lambda_\bB } \right) \land \left(\frac{1}{4\lambda_\bA } \right) \land \left(\frac{k_\bA k_\bB }{2\sqrt{2\lambda_\bA (1+C_\bA )(1+k_\bA )\lambda_\bB (1+C_\bB )(1+k_\bB )}} \right)
				\end{equation}\endgroup
				ensures that  $
				\| \overline{V}_\bB^2 -\overline{V}_\bB^1  \|_{\mathcal{H},1} \le \frac12 \| V^1_\bB -V^2_\bB  \|_{\mathcal{H},1}.$
				By the Banach fixed point theorem, there exists a unique fixed point $V_{\bB,1}=\Psi_1(V_{\bB,1}) \in \mathcal{H}_{\bB ,1}$. Combining this with the definition of $\Psi_1$ and Lemmas \ref{contractivemap_extsience_PsiA}, \ref{BilipschitzVA} shows that there is a unique pair $(V_{\bA,1},V_{\bB,1}) \in \mathcal{H}_{\bA ,1} \times \mathcal{H}_{\bB ,1}$ solving equations \eqref{NHJBx}-\eqref{NHJBy} in $\mathbb{R}_{\geq0}\times [T- \delta_1, T]$ with terminal conditions $h_\bA(x)$ and $h_\bB(y)$. 
				
				Proceeding inductively, suppose that for $i=2,3,\cdots,N_1-1$, the solution $(V_{\bA,i},V_{\bB,i}) \in \mathcal{H}_{\bA ,i} \times \mathcal{H}_{\bB ,i}$ to \eqref{NHJBx}-\eqref{NHJBy} in $\mathbb{R}_{\geq0}\times [T- i\delta_1, T-(i-1)\delta_1]$ exists with the terminal conditions $\widetilde{V}_{\bI,i}=V_{\bI,i-1}(\bigcdot,T-(i-1)\delta_1)\in \mathcal{H}_{\bI}$ for $\bI\in\{\bA,\bB\}$. We set $\widetilde{V}_{\bI,1}(z)=V_{\bI,0}(z,T)=h_\bI(z)$. Note that $\delta_1$, $k_\bA$, $K_\bA$, $\Pi_\bA$, $k_\bB$, $K_\bB$ and $\Pi_\bB$ are independent of $i=1,2,\cdots,N_1$. Thus, with the terminal conditions $\widetilde{V}_{\bI,i+1}=V_{\bI,i}(\bigcdot,T-i\delta_1)\in \mathcal{H}_{\bI}$ for $\bI\in\{\bA,\bB\}$, we can repeat the estimates for \eqref{2250a}-\eqref{2250} to deduce that $\Psi_{i+1}$ remains contractive, yielding the unique solution $(V_{\bA,i+1},V_{\bB,i+1}) \in \mathcal{H}_{\bA ,i+1} \times \mathcal{H}_{\bB ,i+1}$ to \eqref{NHJBx}-\eqref{NHJBy} in $\mathbb{R}_{\geq0}\times [T-(i+1) \delta_1, T-i\delta_1]$. 
				
				Hence, we define $V_\bI^*(z,t)
				:=h_\bI(z)\mathbf{1}_{\{t=T\}}
				+\sum^{N_1}_{i=1}V_{\bI,i}(z,t)\mathbf{1}_{\{t\in[T- i\delta_1, T-(i-1)\delta_1)\}} $ in $ \mathbb{R}_{\geq 0}\times[0, T]$ for $\bI\in\{\bA,\bB\}$. Using the integral representation of $V_{\bA,i }$ iteratively, we obtain
				\begingroup	\begin{align}\label{eq. integral rep. of V_A} 
						V_\bA^*(x,t)
						=\,& \int_t^{T}\!\! e^{-\rho (u-t)} \!\!\left\{ \mathbf{1}_{\{V_{\bA}^* (x,u) \leq (V_{\bB}^*) ^{-1}(x,u)\}}\lambda_\bB \int_{V_{\bA}^* (x,u)}^{(V_{\bB}^*) ^{-1}(x,u)}\!\!\big[y-V_{\bA}^* (x,u)\big] g_\bB (y,u) d y +r_\bA (x,u)\right\} du\nonumber\\
						&+ e^{-\rho (T-t)}h_\bA(x)\hspace{-10pt}
				\end{align}%
				\endgroup
				for any $(x,t) \in \mathbb{R}_{\geq 0}\times[0, T]$; and an analogous expression holds for $V_\bB^*$. These representations, together with Lemmas \ref{contractivemap_extsience_PsiA}-\ref{BilipschitzVB}, clearly show that $(V_\bA^*, V_\bB^*)\in \mathcal{H}_\bA \times \mathcal{H}_\bB$ solves \eqref{NHJBx}-\eqref{NHJBy} in $\mathbb{R}_{\geq 0}\times[0, T]$.
				
				\noindent\textbf{Part 3. Uniqueness of global solution:} Suppose there exist two solutions $(V_\bA^*, V_\bB^*),(V_\bA^\ddagger, V_\bB^\ddagger)\in \mathcal{H}_\bA \times \mathcal{H}_\bB$ satisfying \eqref{NHJBx}-\eqref{NHJBy} in $\mathbb{R}_{\geq 0}\times[0, T]$ with the terminal conditions in \eqref{Nboundary}. As both $V_\bB^*$ and $V_\bB^\ddagger$ are fixed points of $\Psi_1$, we may substitute $V_\bB^1=\overline{V}_\bB^1=V_\bB^*$, $V_\bB^2=\overline{V}_\bB^2=V_\bB^\ddagger$, $\overline{V}_\bA^1=V_\bA^*$ and $\overline{V}_\bA^2=V_\bA^\ddagger$ in Parts 1 and 2 in this proof. Inequalities \eqref{2250} and \eqref{2250a} imply that $\| (V_\bA^* -V_\bA^\ddagger)\mathbf{1}_{\{t\in[T- \delta_1, T]\}}   \|_{\mathcal{H}} \le \frac12 \| (V_\bA^* -V_\bA^\ddagger)\mathbf{1}_{\{t\in[T- \delta_1, T]\}}  \|_{\mathcal{H}},$ which implies $V_\bA^* =V_\bA^\ddagger$ in $\mathbb{R}_{\geq 0}\times[T-\delta_1, T]$. Repeating this step for each interval $[T-i\delta_1, T-(i-1)\delta_1]$ for $i=2,3,\cdots,N_1$, we conclude $V_\bA^* =V_\bA^\ddagger$ on $\mathbb{R}_{\geq 0}\times[0, T]$ and, similarly, $V_\bB^* =V_\bB^\ddagger$ in $\mathbb{R}_{\geq 0}\times[0, T]$, as $\delta_1$ is independent of $i$.
			\end{proof}

			\subsection{Global Well-posedness of Equations \eqref{NFPx}-\eqref{NFPy}}\label{section.wellpose.FP}

			In this section, we fix $(\mu_\bA ,\mu_\bB )\in \mathcal{D}_\bA \times \mathcal{D}_\bB $ and recall the corresponding solution $(V_\bA^* ,V_\bB^* )\in \mathcal{H}_\bA \times \mathcal{H}_\bB $ obtained in \Cref{existence_NHJBxy}. We now complete Step 2B outlined in Appendix \ref{sec. Main result and Sketch of Proof}, showing the existence of a unique global-in-time solution to \eqref{NFPx}–\eqref{NFPy} with $(V_\bA,V_\bB)=(V_\bA^* ,V_\bB^* )$, subject to initial conditions in \eqref{Nboundary}. For any $f \in L^1(\mathbb{R}_{> 0};\mathbb{R}_{\geq 0})$, we define the norm $\|f\|_{1,w}:=\int_0^{\infty}(1+z)|f(z)|dz$; and for any $F \in C([0, T];L^1(\mathbb{R}_{> 0};\mathbb{R}_{\geq 0}))$, we define the norm
			$\| F \|_{\mathcal{A}}:=\sup_{t\in[0, T]}\| F(\bigcdot, t) \|_{1,w}$. Following the approach of \Cref{section.wellpose.HJB}, we partition the time interval $[0,T]$ into $N_2$ uniform subintervals of width $\delta_2 > 0$. In contrast to the backward construction for the HJB system in \Cref{section.wellpose.HJB}, the forward system \eqref{NFPx}–\eqref{NFPy} is solved and pasted forward in time. Fix $i=1,2,\cdots,N_2$ and $\bI\in\{\bA ,\bB \}$. We define the norm $\Vert F \Vert_{\mathcal{A},i}=\sup\limits_{t\in [(i-1)\delta_2, i\delta_2]}\|F (\bigcdot,t)\|_{1,w}$ and the set
			\begingroup\begin{equation*}
				\scalebox{0.95}{$\mathcal{A}_{\bI,i} := \left\{
					f  \in C\big([(i-1)\delta_2, i\delta_2]; L^1(\mathbb{R}_{> 0};  \mathbb{R}_{\geq 0})\big) \,\middle|\, \Vert f\Vert_{\mathcal{A},i} \leq  \dfrac{2^{1+\nu}C_\bI}{\nu}
					\hspace{5pt} \text{and} 
					\sup\limits_{t\in [(i-1)\delta_2, i\delta_2]}\displaystyle\int^\infty_0f(x,t)dx\leq 1
					\right\}$}.
			\end{equation*} \endgroup
			The set $\mathcal{A}_{\bI}$ is defined in \eqref{def.AI}. We define a map $\Upsilon_i$ on $\mathcal{A}_{\bB ,i}$, by setting $\overline{f}_\bB =\Upsilon_i(f_\bB)$, where $(\overline{f}_\bA,\overline{f}_\bB)$ is the unique solution to the following system: 
			\begingroup\begin{equation}\label{NNFPx}
				\frac{\partial \overline{f}_\bA (x,t) }{\partial t}= \left(-\lambda_\bB  \overline{f}_\bA (x,t) \int_{V_\bA^* (x,t)}^{(V_\bB^*) ^{-1}(x,t)} f_\bB (y,t) d y\right)\land 0;
			\end{equation}
			\begin{equation}\label{NNFPy}
				\frac{\partial \overline{f}_\bB (y,t) }{\partial t}= \left(-\lambda_\bA  \overline{f}_\bB (y,t) \int_{V_\bB^* (y,t)}^{(V_\bA^*) ^{-1}(y,t)}\overline{f}_\bA (x,t) d x\right)\land 0,
			\end{equation}\endgroup
			in $\mathbb{R}_{> 0}\times[(i-1) \delta_2, i\delta_2]$ with initial conditions
			\begingroup\begin{equation}\label{eq. ter for fA, fB local}
				\overline{f}_{\bA}(x,(i-1) \delta_2)=\widetilde{f}_{\bA,i-1} (x)\hspace{5pt} \text{and} \hspace{5pt} 
				\overline{f}_{\bB}(y,(i-1) \delta_2)=\widetilde{f}_{\bB,i-1} (y),
			\end{equation}\endgroup
			where $\widetilde{f}_{\bI,i-1}\in\mathcal{A}_{\bI} $ is a given function for $i=1,2,\cdots,N_2$ and $\bI\in\{\bA ,\bB \}$. If $i=1$, then we set $\widetilde{f}_{\bI,0} = f_{\bI,0}$ which lies in $\mathcal{A}_{\bI}$ by Lemma \ref{lemma.integrablepdf} and \Cref{condition_initial}. The well-definedness of $\Upsilon_i$ is established by the following lemma. 
			\begin{lemma}\label{selfmap_Upsilon}
				Suppose that Assumptions \ref{condition_initial}-\ref{condition_terminal} hold. For $f_\bB \in \mathcal{A}_{\bB ,i}$ and $i=1,2,\cdots,N_2$, there is a unique solution $(\overline{f}_\bA,\overline{f}_\bB) \in \mathcal{A}_{\bA ,i} \times \mathcal{A}_{\bB ,i}$ to \eqref{NNFPx}-\eqref{eq. ter for fA, fB local} in $\mathbb{R}_{\geq 0}\times[(i-1) \delta_2, i\delta_2]$. Moreover, $\Upsilon_i$ is a self-map on $\mathcal{A}_{\bB ,i}$.
			\end{lemma}
			\begin{proof}
				Fix $i= 1,2,\cdots,N_2$. The unique explicit solution is $\overline{f}_\bA (x,t)=\widetilde{f}_{\bA,i-1}(x)\exp\left[ \int_{(i-1)\delta_2}^t\left(-\lambda_\bB   \int_{V_\bA^* (x,u)}^{(V_\bB^*) ^{-1}(x,u)} f_\bB (y,u) d y\right)\land 0du\right]$ in $\mathbb{R}_{\geq 0}\times[(i-1)\delta_2, i\delta_2]$. Hence, $\overline{f}_\bA (z,t)\le \widetilde{f}_{\bA,i-1}(z)$ in $\mathbb{R}_{\geq 0}\times[(i-1)\delta_2, i\delta_2]$, 
				implying 
				$\overline{f}_{\bA} \in \mathcal{A}_{\bA,i}$ as $\widetilde{f}_{\bA,i-1}  \in \mathcal{A}_{\bA}$. The proof for $\overline{f}_{\bB}$ is similar. 
			\end{proof}
			
			\begin{proposition}\label{existence_NFPxy}
				Suppose that Assumptions \ref{condition_initial}-\ref{condition_terminal} hold. Then, for any $i=1,2,\cdots,N_2$ and
				\begingroup	\begin{equation}\label{def_FP_delta}
					\delta_2 \le 1 \land\left(
					\dfrac{\nu^2 e^{-(\lambda_\bA+\lambda_\bB)}}{2^{3+2\nu}\lambda_\bA\lambda_\bB C_\bA C_\bB }\right)^{1/2},
				\end{equation}\endgroup
				the map	$\Upsilon_i$ is contractive on $\mathcal{A}_{\bB ,i}$. Hence, there is a unique solution $(f_\bA^*,f_\bB^*) \in \mathcal{A}_{\bA} \times \mathcal{A}_{\bB}$ to \eqref{NFPx}-\eqref{NFPy} in $\mathbb{R}_{\geq 0}\times[0, T]$, with $(V_\bA,V_\bB)=(V_\bA^* ,V_\bB^*)$ and the initial conditions in \eqref{Nboundary}.
			\end{proposition}
			\begin{proof}
				Let $f_\bB \in \mathcal{A}_{\bB ,i}$. The system \eqref{NNFPx}–\eqref{NNFPy} can be written in integral form:
				\begingroup\begin{equation*}
					\overline{f}_\bA (x,t)
					=\widetilde{f}_{\bA,i-1}(x)-\int_{(i-1)\delta_2}^t
					\mathbf{1}_{\{V_\bA^* (x,u) \leq (V_\bB^* )^{-1}(x,u)\}}
					\left(\lambda_\bB  \overline{f}_\bA (x,u) \int_{V_\bA^* (x,u)}^{(V_\bB^*) ^{-1}(x,u)} f_\bB (y,u) d y  \right)du;
				\end{equation*}
				\begin{equation*}
					\overline{f}_\bB (y,t)
					=\widetilde{f}_{\bB,i-1}(y)-\int_{(i-1)\delta_2}^t
					\mathbf{1}_{\{V_\bB^* (y,u) \leq (V_\bA^* )^{-1}(y,u)\}}
					\left(\lambda_\bA  \overline{f}_\bB (y,u) \int_{V_\bB^* (y,u)}^{(V_\bA^*) ^{-1}(y,u)} \overline{f}_\bA (x,u) d x \right)du.
				\end{equation*}\endgroup
				\sloppy Let $f^j_\bB \in \mathcal{A}_{\bB ,i} $, and denote by $(\overline{f}_\bA^j,\overline{f}_\bB^j)$ the corresponding solutions to \eqref{NNFPx}-\eqref{NNFPy} with $f_\bB=f^j_\bB$, for $j=1,2$. Let $R (x,u,V_\bA ,V_\bB ,f_\bA ,f_\bB ):=f_\bA (x,u) \int_{V_\bA (x,u)}^{V_\bB^{-1}(x,u)} f_\bB (y,u) d y$ and $
				I_0:=\big[R (x,u,V_\bA^* ,V_\bB^* ,\overline{f}^2_\bA ,f^2_\bB ) 
				- R (x,u,V_\bA^* ,V_\bB^* ,\overline{f}^1_\bA ,f^1_\bB )\big]\mathbf{1}_{\{V_\bA^* (x,u) \leq (V_\bB^* )^{-1}(x,u)\}}.$ On the set $\pig\{(x,u):V_\bA^* (x,u) \leq (V_\bB^* )^{-1}(x,u)\pig\}$, we have
				\begingroup\begin{align*}
						|I_0| 
						\le\,      \pig| \overline{f}^2_\bA (x,u)-\overline{f}^1_\bA (x,u) \pig|
						+ \pig|\overline{f}^1_\bA (x,u)\pig| \cdot
						\pig\Vert f^2_\bB (\bigcdot,u)-f^1_\bB (\bigcdot,u)\pigr\Vert_{1,w}.                           
				\end{align*}%
				\endgroup
				Integrating and applying \Cref{selfmap_Upsilon} gives $
				\int_0^\infty (1+x)|I_0|dx 
				\leq     \pig\| \overline{f}^1_\bA (\bigcdot,u)
				- \overline{f}^2_\bA (\bigcdot,u)\pigr\|_{1,w}
				+ \frac{2^{1+\nu}C_\bA }{\nu}
				\pig\|f^1_\bB (\bigcdot,u)-f^2_\bB (\bigcdot,u)\pigr\|_{1,w}.$
				Hence, subtracting the integral forms of $ \overline{f}^2_\bA$ and $ \overline{f}^1_\bA$, we obtain
				\begingroup\begin{align*}
						\pig\| \overline{f}^2_\bA (\bigcdot,t)
						-\overline{f}^1_\bA (\bigcdot,t) \pigr\|_{1,w} 
						\leq  \int_{(i-1)\delta_2}^t \lambda_\bB
						\bigg(  \pig\| \overline{f}^1_\bA (\bigcdot,u)- \overline{f}^2_\bA (\bigcdot,u)\pigr\|_{1,w}
						+ \frac{2^{1+\nu}C_\bA }{\nu}\pig\|f^1_\bB (\bigcdot,u)-f^2_\bB (\bigcdot,u)\pigr\|_{1,w}\bigg) du,
				\end{align*}%
				\endgroup
				for any $t\in [(i-1)\delta_2, i\delta_2]$. Applying Gr\"{o}nwall's inequality yields $
				\pig\| \overline{f}^2_\bA 
				-\overline{f}^1_\bA  \pigr\|_{\mathcal{A},i}
				\le e^{\lambda_\bB \delta_2} \delta_2 \lambda_\bB \frac{2^{1+\nu}C_\bA }{\nu}\pig\|f^1_\bB - f^2_\bB \pigr\|_{\mathcal{A},i}.$
				Similarly, we have $\pig\| \overline{f}^2_\bB 
				-\overline{f}^1_\bB\pigr\|_{\mathcal{A},i}
				\le e^{\lambda_\bA \delta_2} \delta_2 \lambda_\bA \frac{2^{1+\nu}C_\bB }{\nu}	\pig\| \overline{f}^2_\bA 
				-\overline{f}^1_\bA  \pigr\|_{\mathcal{A},i}
				\le e^{(\lambda_\bA+\lambda_\bB) \delta_2} \delta_2^2 \lambda_\bA\lambda_\bB \frac{2^{2+2\nu}C_\bA C_\bB }{\nu^2} \pig\|f^1_\bB -f^2_\bB \pigr\|_{\mathcal{A},i}.$ By the choice of $\delta_2$ in \eqref{def_FP_delta}, the map $\Upsilon_i$ is contractive on $\mathcal{A}_{\bB,i}$. The existence and uniqueness of solution $(f_\bA^*,f_\bB^*) \in \mathcal{A}_{\bA} \times \mathcal{A}_{\bB}$ to \eqref{NFPx}-\eqref{NFPy} in $\mathbb{R}_{\geq 0}\times[0, T]$ subject to the initial data in \eqref{Nboundary}, then follows from the same arguments as in Parts 2-3 of the proof of \Cref{existence_NHJBxy}.
			\end{proof}

			\subsection{Existence of Global-in-time Solution of HJB-FP System \eqref{HJBx}-\eqref{boundary}}

			In this section, we complete Step 3 outlined in Appendix \ref{sec. Main result and Sketch of Proof}, establishing the existence of a global-in-time solution to the coupled HJB-FP system \eqref{HJBx}-\eqref{boundary}. In particular, as described in Step 3A, we show that the map $\Phi$ (defined in Step 1 in Appendix \ref{sec. Main result and Sketch of Proof}) is indeed a self-map on $\mathcal{D}_\bA \times \mathcal{D}_\bB$:
			\begin{lemma}\label{theorem.selfmap}
				Suppose that Assumptions \ref{condition_initial}-\ref{condition_terminal} hold. Then $\Phi$ is a self-map on $\mathcal{D}_\bA \times \mathcal{D}_\bB$.
			\end{lemma}
			\begin{proof}
				For $\bI\in\{\bA ,\bB \}$ and any $\mu_\bI\in \mathcal{D}_\bI$, the measure $\mu_{\bI}$ admits the unique decomposition $\mu_{\bI} = p_{\bI} \delta_{\textup{D}} + g_{\bI} d\mathcal{L}^1$ where $g_\bI \in L^\infty\big([0,T];L^1(\mathbb{R}_{> 0};\mathbb{R}_{\geq 0})\big)$ and $p_{\bI} \in L^\infty\big([0,T];[0,1]\big)$. By Propositions~\ref{existence_NHJBxy} and~\ref{existence_NFPxy}, we obtain the corresponding unique solution $(V_\bA^*, V_\bB^*,f_\bA^*,f_\bB^*) \in \mathcal{H}_\bA \times \mathcal{H}_\bB\times \mathcal{A}_{\bA} \times \mathcal{A}_{\bB}$ to equations \eqref{NHJBx}-\eqref{Nboundary}. By the definition of $\Phi(\mu_\bA ,\mu_\bB )=(\mu_\bA^*,\mu_\bB^*)$ in Step 1 in Appendix \ref{sec. Main result and Sketch of Proof}, we have $\mu_\bI^*= p_{\bI}^* \delta_{\textup{D}} + f_{\bI}^* d\mathcal{L}^1$, where $p_{\bI}^*(t) =1-\int^\infty_0 f_{\bI}^* (z,t)dz$. We only prove $\mu_\bA^*  \in \mathcal{D}_{\bA}$; the argument for $\mu_\bB^* \in \mathcal{D}_{\bB}$ is analogous. Recall the definitions of $\mathcal{D}_{\bA}$ and $\mathcal{K}_{\bA}$ in \eqref{DA} and \eqref{KA} respectively, it suffices to verify: (1). $\sup_{x > 0,\; t \in [0,T]}
				\big(1 + x^{2+\nu}\big) f_\bA^*(x,t) \le C_\bA$, and (2). $\frac{W_1\big(\mu_\bA^*(\bigcdot,t),\, \mu_\bA^*(\bigcdot,s)\big)}{|t - s|} \;\le\; c_\bA$. 
				For (1), the equation of $f_\bA^* (x,t)$ in \eqref{NFPx} and \Cref{condition_initial} imply $
				(1+x^{2+\nu})f_\bA^* (x,t) \le (1+x^{2+\nu})f_{\bA ,0}(x)\le C_\bA$ in $\mathbb{R}_{\geq 0}\times[0, T]$. For (2), by integrating \eqref{NFPx} with respect to $t$, we see that for any $x\geq 0$ and $0 \leq s \leq t\leq T$,
				\begingroup\begin{equation}
					f_\bA^* (x,t) 
					=f_\bA^* (x,s)+\int_s^t\left(-\lambda_\bB  f_\bA^* (x,u) \int_{V_\bA^* (x,u)}^{(V_\bB^*) ^{-1}(x,u)} 
					f_\bB^* (y,u)d y\right)\land 0\,du.
					\label{eq. integral rep of f_A}
				\end{equation}\endgroup
				Therefore, for any $\varphi\in\text{Lip}_1(\mathbb{R}_{\ge 0})$ such that $\varphi(0)=0$, Lemma \ref{lemma.closureofK} further implies
				\begingroup\begin{align*}
					\int_0^{\infty} \varphi(x)\pig[f_\bA^* (x,t)-f_\bA^* (x,s)\pig] dx
					\le \lambda_\bB  \int_s^t \int_0^{\infty}  x  f_\bA^* (x,u)dxdu
					\le \frac{2^{1+\nu}\lambda_\bB C_\bA }{\nu} |t-s|.
				\end{align*}\endgroup
				Applying the Kantorovich-Rubinstein duality, we conclude the lemma by
				\begingroup
					\begin{align*}
						W_1(\mu_\bA^* (\bigcdot,t),\mu_\bA^* (\bigcdot,s))
						=\!\!\sup_{\varphi\in\{\text{Lip}_1(\mathbb{R}_{\ge 0})\,:\,\varphi(0)=0\}} \!\!
						\left\{\! \int_0^{\infty}  \varphi(x)f_\bA^* (x,t)dx-\!\int_0^{\infty}  \varphi(x)f_\bA^* (x,s)dx\!\right\} 
						\leq \frac{2^{1+\nu}\lambda_\bB C_\bA }{\nu} | t-s |.  \!\!
				\end{align*}%
				\endgroup
			\end{proof} 
			
			In Step 3B in Appendix \ref{sec. Main result and Sketch of Proof}, we show that $\Phi$ is continuous on $\mathcal{D}_\bA \times \mathcal{D}_\bB$ with respect to $D_1 \otimes D_1$, where
			\begingroup\begin{equation}
				\label{def.metric.D1}
				D_1(\mu,\nu):=\sup\limits_{t\in[0,T]}W_1(\mu(\bigcdot,t),\nu(\bigcdot,t)).
			\end{equation}\endgroup

			\begin{lemma}\label{theorem.continuous}
				Suppose that Assumptions \ref{condition_initial}-\ref{condition_terminal} hold. Then $\Phi$ is a continuous map on $\mathcal{D}_\bA \times \mathcal{D}_\bB$ with respect to the topology induced by the metric $D_1 \otimes D_1$.
			\end{lemma} 
			\begin{proof}
				For $\bI\in\{\bA ,\bB \}$, we let $\mu_\bI^1,\mu_\bI^2\in \mathcal{D}_\bI$ such that $W_1(\mu^1_\bI (\bigcdot,u),\mu^2_\bI (\bigcdot,u))>0$ for $u\in[0,T]$. For $j=1,2$, each measure $\mu_{\bI}^j$ admits the unique decomposition $ p_{\bI}^j \delta_{\textup{D}} + g_{\bI}^j d\mathcal{L}^1$ with $g_\bI^j \in L^\infty\big([0,T];L^1(\mathbb{R}_{> 0};\mathbb{R}_{\geq 0})\big)$ and $p_{\bI}^j \in L^\infty\big([0,T];[0,1]\big)$. By Propositions \ref{existence_NHJBxy} and \ref{existence_NFPxy}, we let  $(V^j_\bA ,V_\bB^j ,f^j_\bA ,f^j_\bB )$ denote the corresponding unique solutions to \eqref{NHJBx}-\eqref{NFPy} in $\mathbb{R}_{>0}\times[0, T]$, with $(g_\bA,g_\bB)=(g_\bA^j,g_\bB^j)$ and initial-terminal conditions in \eqref{Nboundary}.

				\noindent{\bf Part 1. Estimate of $\| V^1_\bA (\bigcdot,t)-V^2_\bA (\bigcdot,t) \|_{0,1/w}  + \| V^1_\bB (\bigcdot,t)-V^2_\bB (\bigcdot,t) \|_{0,1/w}$:} Fix $x> 0$ and $u\in [0,T]$, we assume $V^1_\bA (x,u) \ge V^2_\bA (x,u)$. When $V^1_\bA (x,u) < V^2_\bA (x,u)$, we obtain the same estimate by analogous arguments. Recalling \eqref{SA}, we consider $
				I_0:=S (x,u,V^1_\bA ,V_\bB^1 ,g^1_\bB )\mathbf{1}_{\{V^1_\bA (x,u) \leq (V_\bB^1 )^{-1}(x,u)\}}-S (x,u,V^2_\bA ,V_\bB^2 ,g^2_\bB )\mathbf{1}_{\{V^2_\bA (x,u) \leq (V^2_\bB )^{-1}(x,u)\}}.$ We decompose $
				I_0 
				= I_1 + I_2 - I_3,$ where
				\begingroup\[
				\begin{aligned}
					I_1 &:= 
					\left[S (x,u,V^1_\bA ,V_\bB^1 ,g^1_\bB )-S (x,u,V^2_\bA ,V^2_\bB ,g^2_\bB )\right]
					\mathbf{1}_{\{V^1_\bA (x,u) \leq (V_\bB^1 )^{-1}(x,u) \text{ and }V^2_\bA (x,u) \leq (V^2_\bB )^{-1}(x,u)\}},\\
					I_2 &:= 
					S (x,u,V^1_\bA ,V_\bB^1 ,g^1_\bB )\mathbf{1}_{\{V^1_\bA (x,u) \leq (V_\bB^1 )^{-1}(x,u)\text{ and }V^2_\bA (x,u) > (V^2_\bB )^{-1}(x,u)\}},\\
					I_3 &:= 
					S (x,u,V^2_\bA ,V^2_\bB ,g^2_\bB )\mathbf{1}_{\{V^1_\bA (x,u) > (V_\bB^1 )^{-1}(x,u)\text{ and }V^2_\bA (x,u) \leq (V^2_\bB )^{-1}(x,u)\}}.
				\end{aligned}
				\]\endgroup

				\noindent{\bf Case 1A. $V^1_\bA (x,u) \leq (V_\bB^1 )^{-1}(x,u) \text{ and }V^2_\bA (x,u) \leq (V^2_\bB )^{-1}(x,u)$:} In this case, we have
				\begingroup\begin{align*}
					I_1
					=   & \underbrace{\int_{V^1_\bA (x,u)}^{(V_\bB^1 )^{-1}(x,u)}\big[V^2_\bA (x,u)-V^1_\bA (x,u)\big] g^2_\bB (y,u) d y}_{I_{1,1}} 
					+\underbrace{\int_{V^1_\bA (x,u)}^{(V_\bB^1 )^{-1}(x,u)}\big[y-V^1_\bA (x,u)\big]\big[g^1_\bB (y,u)-g^2_\bB (y,u)\big]dy}_{I_{1,2}} \\[-1.5mm]
					& +\underbrace{\int_{(V^2_\bB )^{-1}(x,u)}^{(V_\bB^1 )^{-1}(x,u)}\big[y-V^2_\bA (x,u)\big] g^2_\bB (y,u) d y}_{I_{1,3}} 
					-\underbrace{\int_{V^2_\bA (x,u)}^{V^1_\bA (x,u)}\big[y-V^2_\bA (x,u)\big] g^2_\bB (y,u) d y}_{I_{1,4}}.                                            
				\end{align*}\normalsize\endgroup
				Using the definition of $g^j_\bB$, the terms $I_{1,1}$, $I_{1,3}$ and $I_{1,4}$ can be easily estimated as follows:
				\begingroup\begin{align*}
					| I_{1,1} | \le & | V^2_\bA (x,u)-V^1_\bA (x,u) |,      \,\,                                                                                                                                              
					| I_{1,3} | \le  \left| \int_{(V^2_\bB )^{-1}(x,u)}^{(V_\bB^1 )^{-1}(x,u)} y  g^2_\bB (y,u) d y \right| \le C_\bB | (V^2_\bB )^{-1}(x,u)-(V_\bB^1 )^{-1}(x,u) |, \\
					| I_{1,4} | \le & \int_{V^2_\bA (x,u)}^{V^1_\bA (x,u)}\left|y-V^2_\bA (x,u)\right| g^2_\bB (y,u) dy          
					\le | V^2_\bA (x,u)-V^1_\bA (x,u) |.
				\end{align*}\endgroup
				For $I_{1,2}$, we assume $ W_1(\mu^1_\bB (\bigcdot,u),\mu^2_\bB (\bigcdot,u))>0$, otherwise $I_{1,2}=0$. Note that $x$ and $u$ are fixed, we consider the following test function
				\begingroup\begin{align*}
					\phi_\delta(y)
					=\,& \big(y - V^1_\bA(x,u)\big)\,\mathbf{1}_{[\,V^1_\bA(x,u),\, (V_\bB^1)^{-1}(x,u)\,]}(y)\\
					&+ \frac{(V_\bB^1)^{-1}(x,u)+\delta - y}{\delta}\,\big((V_\bB^1)^{-1}(x,u)- V^1_\bA(x,u)\big)\,
					\mathbf{1}_{(\,(V_\bB^1)^{-1}(x,u),\, (V_\bB^1)^{-1}(x,u)+\delta\,]}(y),
				\end{align*} \endgroup
				for some $\delta>0$ to be determined later. 
				Then, we can rewrite $I_{1,2}$ in the following way
				\begingroup\begin{align*}
					I_{1,2} =\,&\, \int_0^{\infty} \phi_{\delta}(y) \pig[g^1_\bB (y,u)-g^2_\bB (y,u)\pig]dy\\
					&-\int_{(V_\bB^1 )^{-1}(x,u)}^{(V_\bB^1 )^{-1}(x,u)+\delta} \frac{(V_\bB^1 )^{-1}(x,u)+\delta-y}{\delta}
					\pig[(V_\bB^1 )^{-1}(x,u)-V^1_\bA (x,u)\pig] \pig[g^1_\bB (y,u)-g^2_\bB (y,u)\pig]dy.
				\end{align*}\endgroup
				Note that, for fixed $x\ge 0$ and $u\in [0,T]$, $\phi_{\delta}$ is a bounded Lipschitz function such that $\phi_{\delta}(0)=0$. Then, by the Kantorovich-Rubinstein duality \cite[Proposition 2.6.6]{figalli2021invitation}, together with $V^2_\bA (x,u) \le V^1_\bA (x,u) \leq (V_\bB^1 )^{-1}(x,u)$, we obtain $\pm \int_0^{\infty} \phi_{\delta}(y) \pig[g^1_\bB (y,u)-g^2_\bB (y,u)\pig]dy   
				\le \left[\frac{(V_\bB^1 )^{-1}(x,u)-V^1_\bA (x,u)}{\delta} \vee 1\right] W_1(\mu^1_\bB (\bigcdot,u),\mu^2_\bB (\bigcdot,u)).$ From the definition of $\mathcal{D}_\bB$ in \eqref{DA} and Lemma \ref{lemma.closureofK}, we have
				\begingroup\begin{align*}
					&\left\vert \int_{(V_\bB^1 )^{-1}(x,u)}^{(V_\bB^1 )^{-1}(x,u)+\delta} \frac{(V_\bB^1 )^{-1}(x,u)+\delta-y}{\delta}
					\pig[(V_\bB^1 )^{-1}(x,u)-V^1_\bA (x,u)\pig] \pig[g^1_\bB (y,u)-g^2_\bB (y,u)\pig]dy \right\vert \\
					&\le 2\pig[(V_\bB^1 )^{-1}(x,u)-V^1_\bA (x,u)\pig] C_\bB \delta.
				\end{align*}\endgroup
				Hence, it leads to the following estimate: $\vert I_{1,2} \vert 
				\le \left[\frac{(V_\bB^1 )^{-1}(x,u)}{\delta} \vee 1\right] 
				W_1(\mu^1_\bB (\bigcdot,u),\mu^2_\bB (\bigcdot,u))+2(V_\bB^1 )^{-1}(x,u) C_\bB \delta.$ Choosing $\delta = (V_\bB^1 )^{-1}(x,u) \land \pig[W_1(\mu^1_\bB (\bigcdot,u),\mu^2_\bB (\bigcdot,u))\pigr]^{\frac12}>0$ yields
				\begingroup
				\begin{align*}
					\vert I_{1,2} \vert 
					\le\,& \left\{(V_\bB^1 )^{-1}(x,u)
					\pig[W_1(\mu^1_\bB (\bigcdot,u),\mu^2_\bB (\bigcdot,u))\pigr]^{\frac12}\right\} \vee W_1(\mu^1_\bB (\bigcdot,u),\mu^2_\bB (\bigcdot,u))\nonumber\\
					&+2C_\bB (V_\bB^1 )^{-1}(x,u) 
					\pig[W_1(\mu^1_\bB (\bigcdot,u),\mu^2_\bB (\bigcdot,u))\pigr]^{\frac12}.
				\end{align*}\endgroup
				Therefore, it holds that
				\begingroup\begin{align*}
					| I_0 | 
					\le       \,& C_\bB | (V^2_\bB )^{-1}(x,u)-(V_\bB^1 )^{-1}(x,u) |+2| V^2_\bA (x,u)-V^1_\bA (x,u)| \\
					& +\left\{(V_\bB^1 )^{-1}(x,u)
					\pig[W_1(\mu^1_\bB (\bigcdot,u),\mu^2_\bB (\bigcdot,u))\pigr]^{\frac12}\right\} \vee W_1(\mu^1_\bB (\bigcdot,u),\mu^2_\bB (\bigcdot,u))\nonumber\\
					&+2C_\bB (V_\bB^1 )^{-1}(x,u) 
					\pig[W_1(\mu^1_\bB (\bigcdot,u),\mu^2_\bB (\bigcdot,u))\pigr]^{\frac12}.
				\end{align*}\endgroup
				
				\noindent{\bf Case 1B. $V^1_\bA (x,u) \leq (V_\bB^1 )^{-1}(x,u)\text{ and }V^2_\bA (x,u) > (V^2_\bB )^{-1}(x,u)$:} By the assumption that $V^1_\bA (x,u) \ge V^2_\bA (x,u)$, we have
				\begingroup$$
					0\leq I_2 = \int_{V^1_\bA (x,u)}^{(V_\bB^1 )^{-1}(x,u)}\big[y-V_\bA^1(x,u)\big] g^1_\bB (y,u) d y                                                     
					\leq\, (V_\bB^1 )^{-1}(x,u)-V_\bA ^1(x,u)
					\leq\,| (V^2_\bB )^{-1}(x,u)-(V_\bB^1 )^{-1}(x,u) |. 
					$$
				\endgroup
				Hence $|I_0| = |I_2| \leq | (V^2_\bB )^{-1}(x,u)-(V_\bB^1 )^{-1}(x,u) |$.
				
				\noindent{\bf Case 1C. $V^1_\bA (x,u) > (V_\bB^1 )^{-1}(x,u)\text{ and }V^2_\bA (x,u) \leq (V^2_\bB )^{-1}(x,u)$:} In this case, we decompose $0\leq I_3= I_{3,1}+I_{3,2}+I_{3,3}$ where
				\begingroup	\begin{align*} 
					I_{3,1}= &\,  \mathbf{1}_{\{V^2_\bA (x,u) > (V_\bB^1 )^{-1}(x,u)\}}
					\int_{V^2_\bA (x,u)}^{(V^2_\bB )^{-1}(x,u)}
					\big[y-V_\bA ^2(x,u)\big] g^2_\bB (y,u) d y ;   \\
					I_{3,2}=&\, \mathbf{1}_{\{V^2_\bA (x,u) \leq (V_\bB^1 )^{-1}(x,u)\text{ and }(V^2_\bB )^{-1}(x,u) \leq V^1_\bA (x,u)\}}
					\int_{V^2_\bA (x,u)}^{(V^2_\bB )^{-1}(x,u)}\big[y-V_\bA ^2(x,u)\big] g^2_\bB (y,u) d y;  \\
					I_{3,3}=&\,  \mathbf{1}_{\{V^2_\bA (x,u) \leq (V_\bB^1 )^{-1}(x,u)\text{ and }(V^2_\bB )^{-1}(x,u) > V^1_\bA (x,u)\}}
					\int_{V^2_\bA (x,u)}^{(V^2_\bB )^{-1}(x,u)}\big[y-V_\bA ^2(x,u)\big] g^2_\bB (y,u) d y.   
				\end{align*}\endgroup
				The definition of $g^j_\bB$ yields 	$0 \leq I_{3,1} \leq\, | (V^2_\bB )^{-1}(x,u)-(V_\bB^1 )^{-1}(x,u) |$ and
				\begingroup\begin{align*}
					0 \leq I_{3,2} \leq\, &  \mathbf{1}_{\{V^2_\bA (x,u) \leq (V_\bB^1 )^{-1}(x,u)\text{ and }(V^2_\bB )^{-1}(x,u) \leq V^1_\bA (x,u)\}}
					\int_{V^2_\bA (x,u)}^{(V^2_\bB )^{-1}(x,u)}
					\big[V^1_\bA (x,u)-V_\bA ^2(x,u)\big] g^2_\bB (y,u) d y \\
					\leq\,& | V^2_\bA (x,u)-V_\bA ^1(x,u) |.                 
				\end{align*}\endgroup
				Next, we decompose\begingroup\begin{align*}
					0 \leq I_{3,3} = &\, \Bigg[\int_{(V_\bB^1 )^{-1}(x,u)}^{(V^2_\bB )^{-1}(x,u)}
					\big[y-V_\bA ^2(x,u)\big] g^2_\bB (y,u) d y 
					- \int_{(V_\bB^1 )^{-1}(x,u)}^{V^1_\bA (x,u)}\big[y-V_\bA ^2(x,u)\big] g^2_\bB (y,u) d y \\
					&\,\, + \int_{V^2_\bA (x,u)}^{V^1_\bA (x,u)}\big[y-V_\bA ^2(x,u)\big] g^2_\bB (y,u) d y\Bigg]  \mathbf{1}_{\{ V^2_\bA (x,u) \leq (V_\bB^1 )^{-1}(x,u) < V^1_\bA (x,u) < (V^2_\bB )^{-1}(x,u)\}}.
				\end{align*}  \endgroup
				Lemma \ref{lemma.closureofK} and assumption of this case  imply
				\begingroup\begin{align*}
					0 \leq I_{3,3} 
					\leq             & \Bigg[\int_{(V_\bB^1 )^{-1}(x,u)}^{(V^2_\bB )^{-1}(x,u)}y g^2_\bB (y,u) d y
					+ \int_{V^2_\bA (x,u)}^{V^1_\bA (x,u)}
					\big[V^1_\bA (x,u)-V_\bA ^2(x,u)\big] g^2_\bB (y,u) d y\Bigg]  \\
					& \cdot \mathbf{1}_{\{ V^2_\bA (x,u) \leq (V_\bB^1 )^{-1}(x,u) < V^1_\bA (x,u) < (V^2_\bB )^{-1}(x,u)\}}                                                                                                              \\
					\leq             & C_\bB | (V^2_\bB )^{-1}(x,u)-(V_\bB^1 )^{-1}(x,u) |+| V^2_\bA (x,u)-V^1_\bA (x,u)|.                      
				\end{align*}  \endgroup
				Hence, we have $
				|I_0|= |I_3|
				\leq   (1+C_\bB )| (V^2_\bB )^{-1}(x,u)-(V_\bB^1 )^{-1}(x,u) |+2| V^2_\bA (x,u)-V^1_\bA (x,u)|. $ 
				
				From the integral representation (see, e.g., \eqref{eq. integral rep. of V_A}) of solution $(V_\bA^1, V_\bB^1)$ satisfying \eqref{NHJBx}-\eqref{NHJBy}, we know that $V_\bI^1 (0,t)\ge 0$ for $\bI \in \{\bA,\bB\}$. Since $V_\bI^1 \in\mathcal{H}_{\bI}$, it holds that $V_\bI^1(y,t)\ge V_\bI^1 (0,t)+k_\bI y\geq k_\bI y$ for any $y \geq 0$ and $t\in [0,T]$. Consequently, the inverse satisfies $(V_\bI^1 )^{-1}(x,u)\le \frac{1}{k_\bI}x$ so that $\frac{(V_\bI^1 )^{-1}(x,u)}{1+x}\le \frac{1}{k_\bI}$. Combining this bound with the estimates from Cases 1A–1C, and using arguments similar to those in Lemma~\ref{lem Lip inv of VA and VB}, we conclude that
				\begingroup$$
				\frac{| I_0 |}{1+x} 
				\le   2 \| V^2_\bA (\bigcdot,u)-V^1_\bA (\bigcdot,u) \|_{0,1/w} + \frac{(1+C_\bB )(1+k_\bB )}{(k_\bB )^2}  \| V_\bB^1 (\bigcdot,u)-V^2_\bB (\bigcdot,u) \|_{0,1/w}+F_\bB(\mu^1_\bB (\bigcdot,u),\mu^2_\bB (\bigcdot,u))                         
				$$\endgroup
				where $
				F_\bB(\mu^1_\bB (\bigcdot,u),\mu^2_\bB (\bigcdot,u))
				:= \left\{ \frac{1}{k_\bB} \left[ W_1\left(\mu^1_\bB (\bigcdot,u),\mu^2_\bB (\bigcdot,u) \right) \right]^{\frac{1}{2}} \right\}
				\vee\; W_1\left(\mu^1_\bB (\bigcdot,u),\mu^2_\bB (\bigcdot,u) \right) 
				+ \frac{2C_\bB}{k_\bB}  \left[ W_1\left(\mu^1_\bB (\bigcdot,u),\mu^2_\bB (\bigcdot,u) \right) \right]^{\frac{1}{2}}.$ Utilizing the integral representations of $V_\bA^1$ and $V_\bA^2$ in \eqref{eq. integral rep. of V_A}, we estimate their difference as follows:
				\begingroup\begin{align}\label{estimate_VAdiff}
						\| V^1_\bA (\bigcdot,t)-V^2_\bA (\bigcdot,t) \|_{0,1/w} 
						\le\,                                                          & \lambda_\bB \!\!\int_t^{T}\!\! e^{-\rho (u-t)} \!  \Big[\! M_\bB  \sum_{\bI=\bA,\bB}\!\!\| V^2_\bI (\bigcdot,u)-V^1_\bI (\bigcdot,u) \|_{0,1/w}
						+F_\bB(\mu^1_\bB (\bigcdot,u),\mu^2_\bB (\bigcdot,u))\! \Big] du                     
				\end{align}%
				\endgroup
				where $M_\bB  := 2 \vee \frac{(1+C_\bB )(1+k_\bB )}{(k_\bB )^2}$. A similar estimate holds for $ \| V^1_\bB (\bigcdot,t)-V^2_\bB (\bigcdot,t) \|_{0,1/w}$ with $M_\bA$ and $	F_\bA(\mu^1_\bA(\bigcdot,u),\mu^2_\bA(\bigcdot,u))$ defined analogously. Summing \eqref{estimate_VAdiff} for $V^1_\bA$ and $V^1_\bB$, and then applying Gr\"{o}nwall's inequality, we have
				\begingroup\begin{align}\label{estimate_VAVB}
					\nonumber     & \| V^1_\bA (\bigcdot,t)-V^2_\bA (\bigcdot,t) \|_{0,1/w}  + \| V^1_\bB (\bigcdot,t)-V^2_\bB (\bigcdot,t) \|_{0,1/w} \\
					\nonumber&\le   e^{\int_t^T \left(\lambda_\bA M_\bA +\lambda_\bB M_\bB \right)  du}\int_t^{T} e^{-\rho (u-t)} \Big[\lambda_\bA F_\bA(\mu^1_\bA(\bigcdot,u),\mu^2_\bA(\bigcdot,u)) + \lambda_\bB F_\bB(\mu^1_\bB (\bigcdot,u),\mu^2_\bB (\bigcdot,u)) \Big] du \\
					&\leq  M_0  H(\mu^1_\bA ,\mu^2_\bA, \mu^1_\bB ,\mu^2_\bB),
				\end{align}\endgroup
				recalling \eqref{def.metric.D1}, we write $M_0 := e^{\left(\lambda_\bA M_\bA +\lambda_\bB M_\bB \right)T} \cdot \frac{1-e^{-\rho T}}{\rho} (\lambda_\bA  \vee \lambda_\bB )$ and $H(\mu^1_\bA ,\mu^2_\bA, \mu^1_\bB ,\mu^2_\bB)
				:=\; \sum_{\bI=\bA,\bB} 
				\left\{ \frac{1}{k_\bI} \left[ D_1(\mu^1_\bI ,\mu^2_\bI ) \right]^{1/2} \right\}
				\vee D_1(\mu^1_\bI ,\mu^2_\bI ) 
				+ \frac{2C_\bI}{k_\bI}  \left[ D_1(\mu^1_\bI ,\mu^2_\bI ) \right]^{1/2}.$

				\noindent{\bf Part 2. Estimate of $D_1(f^1_\bA ,f^2_\bA ) + D_1(f^1_\bB ,f^2_\bB )$:} \sloppy Let $R (x,u,V_\bA ,V_\bB ,f_\bA ,f_\bB ):=  f_\bA (x,u) \int_{V_\bA (x,u)}^{V_\bB^{-1}(x,u)} f_\bB (y,u) d y$, we consider
				\begingroup\begin{align*}
					J_0:=&\, R (x,u,V^2_\bA ,V^2_\bB ,f^2_\bA ,f^2_\bB )\mathbf{1}_{\{V^2_\bA (x,u) \leq (V^2_\bB )^{-1}(x,u)\}} - R (x,u,V^1_\bA ,V^1_\bB ,f^1_\bA ,f^1_\bB )\mathbf{1}_{\{V^1_\bA (x,u) \leq (V^1_\bB )^{-1}(x,u)\}}                       \\
					= &\, \underbrace{\left[R (x,u,V^2_\bA ,V^2_\bB ,f^2_\bA ,f^2_\bB ) - R (x,u,V^1_\bA ,V^1_\bB ,f^1_\bA ,f^1_\bB )\right]\mathbf{1}_{\{V^2_\bA (x,u) \leq (V^2_\bB )^{-1}(x,u)\text{ and }V^1_\bA (x,u) \leq (V^1_\bB )^{-1}(x,u)\}}}_{J_1} \\[-1.5mm]
					& +\underbrace{ R (x,u,V^2_\bA ,V^2_\bB ,f^2_\bA ,f^2_\bB )\mathbf{1}_{\{V^2_\bA (x,u) \leq (V^2_\bB )^{-1}(x,u)\text{ and }V^1_\bA (x,u) > (V^1_\bB )^{-1}(x,u)\}}}_{J_2}                                                                                               \\[-1.5mm]
					& -\underbrace{R (x,u,V^1_\bA ,V^1_\bB ,f^1_\bA ,f^1_\bB )\mathbf{1}_{\{V^2_\bA (x,u) > (V^2_\bB )^{-1}(x,u)\text{ and }V^1_\bA (x,u) \leq (V^1_\bB )^{-1}(x,u)\}}}_{J_3}.                                                                                               
				\end{align*}\endgroup
				\noindent{\bf Case 2A. $V^2_\bA (x,u) \leq (V^2_\bB )^{-1}(x,u) \text{ and } V^1_\bA (x,u) \leq (V^1_\bB )^{-1}(x,u)$:}
				Similar to the arguments in Case 1A, we directly expand
				\begingroup\begin{align*}
					J_1  
					=      & \underbrace{[f^2_\bA (x,u)-f^1_\bA (x,u)]\int_{V^2_\bA (x,u)}^{(V^2_\bB )^{-1}(x,u)} f^2_\bB (y,u) d y}_{J_{1,1}}                                                                                                                      +\underbrace{f^1_\bA (x,u)\int_{V^2_\bA (x,u)}^{(V^2_\bB )^{-1}(x,u)} [f^2_\bB (y,u)-f^1_\bB (y,u)] d y}_{J_{1,2}}                                                                                                                    \\[-1mm]
					& +\underbrace{f^1_\bA (x,u)\int_{(V^1_\bB )^{-1}(x,u)}^{(V^2_\bB )^{-1}(x,u)} f^1_\bB (y,u) d y}_{J_{1,3}}+\underbrace{f^1_\bA (x,u)\int_{V^2_\bA (x,u)}^{V^1_\bA (x,u)} f^1_\bB (y,u) d y}_{J_{1,4}} . 
				\end{align*}\endgroup
				It is obvious that $
				|J_{1,1}| \le | f^2_\bA (x,u)-f^1_\bA (x,u) |$ and $|J_{1,2}| \le  f^1_\bA (x,u) \int_{0}^{\infty} (1+y)\left|f^2_\bB (y,u)-f^1_\bB (y,u)\right| d y \le  f^1_\bA (x,u)  \Vert f^2_\bB (\bigcdot,u)-f^1_\bB (\bigcdot,u)\Vert_{1,w}.$ For $J_{1,3}$, if $(V^1_\bB )^{-1}(x,u) \le (V^2_\bB )^{-1}(x,u)$, then Lemmas \ref{theorem.selfmap} and \ref{lemma.closureofK} yield
				\begingroup\begin{align}
					|J_{1,3}| \le  \frac{f^1_\bA (x,u)}{1+(V^1_\bB )^{-1}(x,u)}\int_{(V^1_\bB )^{-1}(x,u)}^{(V^2_\bB )^{-1}(x,u)} (1+y) f^1_\bB (y,u) d y
					\le  2 C_\bB  f^1_\bA (x,u) \frac{(V^2_\bB )^{-1}(x,u)-(V^1_\bB )^{-1}(x,u)}{1+(V^1_\bB )^{-1}(x,u)}.\nonumber
				\end{align}\endgroup
				Applying \eqref{inequality:Vinv} to the above inequality, we have
				\begingroup\begin{equation}\label{2784}
					|J_{1,3}| \le \frac{2 C_\bB}{k_\bB} f^1_\bA (x,u) 
					\| V^2_\bB (\bigcdot,u)-V^1_\bB (\bigcdot,u) \|_{0,1/w}.
				\end{equation}\endgroup
				If instead $(V^1_\bB )^{-1}(x,u) > (V^2_\bB )^{-1}(x,u)$, then a symmetric argument yields the same bound, by proceeding in a similar manner and interchanging the upper and lower limits of the integral. Similarly, for $J_{1,4}$, if $V^2_\bA (x,u) \le V^1_\bA (x,u)$, we use the fact $
				1+k_\bA x \ge \frac{k_\bA}{1+k_\bA} (1+x)$ for any $x \ge 0$ and Lemma \ref{lemma.closureofK} to deduce
				\begingroup\begin{align}
					|J_{1,4}| \le \frac{ f^1_\bA (x,u)}{1+V^2_\bA (x,u)}\int_{V^2_\bA (x,u)}^{V^1_\bA (x,u)} (1+y)f^1_\bB (y,u) d y
					&\le  2 C_\bB  f^1_\bA (x,u) \frac{V^1_\bA (x,u)-V^2_\bA (x,u)}{1+k_\bA x}\nonumber\\
					&\le  \frac{2 C_\bB (1+k_\bA)}{k_\bA}  f^1_\bA (x,u)
					\| V^2_\bA (\bigcdot,u)-V^1_\bA (\bigcdot,u) \|_{0,1/w}.\label{2801}
				\end{align}\endgroup Otherwise, we can also yield the same estimate by proceeding in a similar manner.

				\noindent{\bf Case 2B. $V^2_\bA (x,u) \leq (V^2_\bB )^{-1}(x,u)\text{ and }V^1_\bA (x,u) > (V^1_\bB )^{-1}(x,u)$:} 
				If $V^2_\bA (x,u) > V^1_\bA (x,u)$, then, by an estimate analogous to \eqref{2784}, we have:
				\begingroup\begin{align*}
					0\leq J_2 
					\leq \,    f^2_\bA (x,u) \int_{(V^1_\bB )^{-1}(x,u)}^{(V^2_\bB )^{-1}(x,u)} f^2_\bB (y,u) d y  
					\le\,  \frac{2 C_\bB}{k_\bB} f^2_\bA (x,u) \| V^2_\bB (\bigcdot,u)-V^1_\bB (\bigcdot,u) \|_{0,1/w}. \!
				\end{align*}\endgroup
				If $V^2_\bA (x,u) \leq V^1_\bA (x,u)$, then estimates similar to \eqref{2784} and \eqref{2801} give
				\begingroup\begin{align*}
					0\leq J_2 \leq\, & f^2_\bA (x,u) \int_{(V^1_\bB )^{-1}(x,u)}^{(V^2_\bB )^{-1}(x,u)} f^2_\bB (y,u) d y + f^2_\bA (x,u) \int_{V^2_\bA (x,u)}^{V^1_\bA (x,u)} f^2_\bB (y,u) d y  \\
					\leq \,        &   \frac{2 C_\bB}{k_\bB} f^2_\bA (x,u) \| V^2_\bB (\bigcdot,u)-V^1_\bB (\bigcdot,u) \|_{0,1/w}
					+\frac{2 C_\bB (1+k_\bA)}{k_\bA}  f^2_\bA (x,u)  \| V^2_\bA (\bigcdot,u)-V^1_\bA (\bigcdot,u) \|_{0,1/w}.
				\end{align*}\endgroup
				Hence, combining these two cases, the above inequality is also valid for $J_2$ in general.
				
				\noindent{\bf Case 2C. $V^2_\bA (x,u) > (V^2_\bB )^{-1}(x,u)\text{ and }V^1_\bA (x,u) \leq (V^1_\bB )^{-1}(x,u)$:} We can interchange $(V^1_\bA ,V^1_\bB ,f^1_\bA ,f^1_\bB)$ and  $(V^2_\bA ,V^2_\bB ,f^2_\bA ,f^2_\bB)$, and then use the result in Case 2B to conclude $0\leq J_3 
				\leq \,          \frac{2 C_\bB}{k_\bB} f^1_\bA (x,u) \| V^2_\bB (\bigcdot,u)-V^1_\bB (\bigcdot,u) \|_{0,1/w}
				+\frac{2 C_\bB (1+k_\bA)}{k_\bA}  
				f^1_\bA (x,u)  \| V^2_\bA (\bigcdot,u)-V^1_\bA (\bigcdot,u) \|_{0,1/w}.$
				
				Combining the results in Cases 2A–2C, we find that
				\begingroup\begin{align*}
						|J_0|                       
						\leq   & \, | f^2_\bA (x,u)-f^1_\bA (x,u) |
						+ f^1_\bA (x,u)\Vert f^2_\bB (\bigcdot,u)-f^1_\bB (\bigcdot,u)\Vert_{1,w}\\
						& +  2C_\bB[f^1_\bA (x,u)\vee f^2_\bA (x,u)]
						\bigg( \frac{1}{k_\bB}
						\|V^1_\bB (\bigcdot,u)-V^2_\bB (\bigcdot,u) \|_{0,1/w} 
						+ \frac{1+k_\bA}{k_\bA} 
						\| V^2_\bA (\bigcdot,u)-V^1_\bA (\bigcdot,u) \|_{0,1/w}\bigg). 
				\end{align*}%
				\endgroup
				Therefore, subtracting the integral forms of  $f^1_\bA$ and $f^2_\bA$ in \eqref{eq. integral rep of f_A}, and applying Lemma~\ref{lemma.closureofK} gives
				\begingroup\begin{align}
					\nonumber&\hspace{-10pt} \| f^2_\bA (\bigcdot,t)-f^1_\bA (\bigcdot,t) \|_{1,w}\\
					\nonumber\leq \,&  \lambda_\bB\int_0^t 
					\bigg\{\|f^1_\bA (\bigcdot,u)-f^2_\bA (\bigcdot,u)\|_{1,w}
					+  \frac{2^{1+\nu}C_\bA }{\nu}\|f^1_\bB (\bigcdot,u)-f^2_\bB (\bigcdot,u)\|_{1,w}\\
					& + C_\bB\left[ \frac{2^{3+\nu}C_\bA }{\nu}
					\right] 
					\bigg( \frac{1}{k_\bB}
					\|V^1_\bB (\bigcdot,u)-V^2_\bB (\bigcdot,u) \|_{0,1/w} 
					+ \frac{1+k_\bA}{k_\bA} 
					\| V^2_\bA (\bigcdot,u)-V^1_\bA (\bigcdot,u) \|_{0,1/w}\bigg)\bigg\}du
					\label{estimate_fdiff1A1}
				\end{align}\endgroup
				for all $t\in[0,T]$. A similar bound holds for $\| f^2_\bB (\bigcdot,t)-f^1_\bB (\bigcdot,t) \|_{1,w}$. Hence, summing these yields
				\begingroup\begin{align}\label{estimate_fdiff1}
					&\sum_{\bI=\bA,\bB} \| f^1_\bI (\bigcdot,t) - f^2_\bI (\bigcdot,t) \|_{1,w} \notag \\
					& \leq \int_0^t
					\left(\lambda_\bA + \lambda_\bB \tfrac{2^{1+\nu} C_\bA}{\nu}\right)
					\vee
					\left(\lambda_\bB + \lambda_\bA \tfrac{2^{1+\nu} C_\bB}{\nu}\right)
					\sum_{\bI=\bA,\bB} \| f^1_\bI (\bigcdot,u) - f^2_\bI (\bigcdot,u) \|_{1,w} \notag\\
					& \quad\quad\quad 
					+ \frac{2^{3+\nu} C_\bA C_\bB}{\nu}
					\left(\dfrac{\lambda_\bA+\lambda_\bB(1+k_\bA)}{k_\bA}
					\vee  \dfrac{\lambda_\bB+\lambda_\bA(1+k_\bB)}{k_\bB}\right)
					\sum_{\bI=\bA,\bB} \| V^1_\bI (\bigcdot,u) - V^2_\bI (\bigcdot,u) \|_{0,1/w}du
				\end{align}\endgroup
				for all $t\in[0,T]$. Plugging \eqref{estimate_VAVB} in \eqref{estimate_fdiff1}, it holds that $
				\sum_{\bI=\bA,\bB} \| f^1_\bI (\bigcdot,t) - f^2_\bI (\bigcdot,t) \|_{1,w}
				\leq 
				\int_0^t M_2 \sum_{\bI=\bA,\bB} \| f^1_\bI (\bigcdot,u) - f^2_\bI (\bigcdot,u) \|_{1,w}
				+M_1H(\mu^1_\bA ,\mu^2_\bA, \mu^1_\bB ,\mu^2_\bB)du,$ where $
				M_1:=  \frac{2^{3+\nu}C_\bA C_\bB M_0}{\nu }
				\left[\frac{\lambda_\bA+\lambda_\bB(1+k_\bA)}{k_\bA}
				\vee  \frac{\lambda_\bB+\lambda_\bA(1+k_\bB)}{k_\bB}\right]$ and $M_2:=	\left(\lambda_\bA + \lambda_\bB \tfrac{2^{1+\nu} C_\bA}{\nu}\right)
				\vee
				\left(\lambda_\bB + \lambda_\bA \tfrac{2^{1+\nu} C_\bB}{\nu}\right)$. By Gr\"{o}nwall's inequality, we deduce that for all $t\in[0,T]$,
				\begingroup\begin{align}
					\sum_{\bI=\bA,\bB} \| f^1_\bI (\bigcdot,t) - f^2_\bI (\bigcdot,t) \|_{1,w}
					\leq T 
					M_1 H(\mu^1_\bA ,\mu^2_\bA, \mu^1_\bB ,\mu^2_\bB)
					e^{M_2 T},
					\label{2925}
				\end{align}\endgroup
				where $H$ is defined below equation \eqref{estimate_VAVB}.
				
				Finally, define the probability measure $\mu^{*,j}_{\bI}:=q_{\bI}^j\delta_{\textup{D}} +f^j_\bI d\mathcal{L}^1$ where $q_{\bI}^j(t) :=1-\int^\infty_0 f_{\bI}^j (z,t)dz$, for $j=1,2$ and $\bI\in\{\bA,\bB\}$. Then, for any $t \in [0,T]$ and $\varphi\in\text{Lip}_1(\mathbb{R}_{\ge 0})$ with $\varphi(0)=0$, $\int_{[0,\infty)} \varphi(z)d\left[\mu^{*,2}_\bA (z,t)- \mu^{*,1}_\bA (z,t)\right]
				\leq    \,                                                                               \int_0^{\infty} z\left|f^2_\bA (z,t)- f^1_\bA (z,t)\right|dz                                 
				\leq   \,                                                                                \| f^2_\bA (\bigcdot,t)-f^1_\bA (\bigcdot,t) \|_{1,w}.$ Thus, the Kantorovich-Rubinstein duality deduces $
				W_1(\mu^{*,1}_\bA(\bigcdot,t) ,\mu^{*,2}_\bA (\bigcdot,t))
				\le \| f^2_\bA (\bigcdot,t)-f^1_\bA (\bigcdot,t) \|_{1,w}.$ A similar inequality holds for $W_1(\mu^{*,1}_\bB(\bigcdot,t) ,\mu^{*,2}_\bB (\bigcdot,t))$. We take the supremum over time of these two inequalities and then use \eqref{2925} to conclude that $\Phi$ is a continuous map with respect to the metric $D_1 \otimes D_1$. 
			\end{proof}
			
			In Step 3C of Appendix \ref{sec. Main result and Sketch of Proof}, we establish the compactness of $\mathcal{D}_{\bA}$ and $ \mathcal{D}_{\bB}$. 
			
			\begin{lemma}\label{theorem.compact}
				\sloppy	For each $\bI\in\{\bA,\bB\}$, the set $\mathcal{D}_{\bI}$ defined in \eqref{DA} is a compact subset of $C\big([0,T];\mathcal{M}^1_f(\mathbb{R}_{\ge 0})\big)$ under the topology induced by the metric $D_1$.
			\end{lemma}
			\begin{proof} First, by Lemma \ref{lemma.closureofK}, the set $\overline{\mathcal{K}_\bI} $ is compact in $\mathcal{P}_1(\mathbb{R}_{\ge 0})$ with respect to metric $W_1$, guaranteeing the uniform boundedness of $\mathcal{D}_{\bI}$. Second, any sequence $\{\mu_n\}_{n\in \mathbb{N}} \subseteq \mathcal{D}_{\bI}$ satisfies the uniform Lipschitz bound $\sup_{n\in \mathbb{N}}\sup\limits_{ s,t \in [0,T],s \neq t}\frac{W_1(\mu_n(\bigcdot,t),\mu_n(\bigcdot,s))}{|t-s|}\le c_\bI$ which implies the equicontinuity of $\{\mu_n(\bigcdot,t)\}_{n\in \mathbb{N}}$ in variable $t$. Therefore, by the Arzela-Ascoli theorem, there exists a subsequence $\{\mu_{n_k}\}_{k\in\mathbb{N}}$ such that $\mu_{n_k} \rightarrow \mu \in \mathcal{D}_{\bI}$ with respect to $D_1$. Therefore, $\mathcal{D}_{\bI}$ is a compact subset of $C([0,T];\mathcal{P}_1(\mathbb{R}_{\ge 0}))$. {\color{black}Note that  $C([0,T];\mathcal{P}_1(\mathbb{R}_{\ge 0})) \subseteq C\big([0,T];\mathcal{M}^1_f(\mathbb{R}_{\ge 0})\big)$; and $\mathcal{P}_1(\mathbb{R}_{\ge 0})$ and $\mathcal{M}^1_f(\mathbb{R}_{\ge 0})$ share the same Kantorovich-Rubinstein metric $W_1$,} then $\mathcal{D}_{\bI}$ is a compact subset of $C\big([0,T];\mathcal{M}^1_f(\mathbb{R}_{\ge 0})\big)$. 
			\end{proof}

			For each $\bI\in\{\bA,\bB\}$, it is straightforward to verify that $\mathcal{D}_{\bI}$ is a convex subset of the Banach space $C\big([0,T];\mathcal{M}^1_f(\mathbb{R}_{\ge 0})\big)$. Combining Lemmas \ref{theorem.selfmap}, \ref{theorem.continuous}, and \ref{theorem.compact} with the Schauder fixed point theorem, we immediately deduce Theorem \ref{maintheorem}.

			\section{Proof of \Cref{thm unique}: Uniqueness of System \eqref{HJBx}-\eqref{boundary}}\label{sec. unique}
			The core of the proof of uniqueness parallels that of Lemma~\ref{theorem.continuous}, where we derived bounds on $V^1_\bA (\bigcdot,t)-V^2_\bA (\bigcdot,t)$ in terms of the $W_1$ metric.  In the present proof, we adopt those estimates to the $\|\bigcdot\|_{1,w}$-norm. As such, most steps remain unchanged from Lemma~\ref{theorem.continuous}, and we focus primarily on modifying the essential term previously controlled by $W_1$.

			\begin{proof}[Proof of \Cref{thm unique}]\sloppy We assume there are two solutions to \eqref{HJBx}-\eqref{boundary}, namely $(V^1_\bA ,V^1_\bB ,f^1_\bA ,f^1_\bB )$ and $(V^2_\bA ,V^2_\bB ,f^2_\bA ,f^2_\bB )$, both belong to the desired space. We reuse the arguments from Lemma~\ref{theorem.continuous}, replacing $g^j_\bI$ there by $f^j_\bI$ and replace $\mu^{*,j}_\bI$ there by $\mu^{j}_\bI$ for $j=1,2$ and $\bI \in \{\bA,\bB\}$.

				\noindent{\bf Part 1. Uniqueness under Condition \ref{condition.uniqueness.loc}:} Recall the term $I_{1,2}$ in Case 1A of the proof of Lemma \ref{theorem.continuous}. Since $V^1_\bA (x,u)\le (V^1_\bB )^{-1}(x,u)$ in Case 1A there, then we have $|I_{1,2}|\leq\int_{V^1_\bA (x,u)}^{(V_\bB^1 )^{-1}(x,u)}y\,\big|\,f^1_\bB (y,u)-f^2_\bB (y,u)\,\big|\,dy
				\leq \| f^1_\bB (\bigcdot,u)-f^2_\bB (\bigcdot,u) \|_{1,w}.
				$ Combining this with the original estimates for the remaining terms $I_{1,1}$, $I_{1,3}$, $I_{1,4}$, $I_{2}$ and $I_{3}$ in Part 1 of the proof of Lemma \ref{theorem.continuous}, inequality \eqref{estimate_VAdiff} implies
				\begingroup\begin{align} \label{2731}
						&\| V^1_\bA (\bigcdot,t)-V^2_\bA (\bigcdot,t) \|_{0,1/w}\nonumber\\
						&\le                                                          \lambda_\bB \int_t^{T} e^{-\rho (u-t)}   \bigg( M_\bB \sum_{\bI=\bA,\bB} \| V^2_\bI (\bigcdot,u)-V^1_\bI (\bigcdot,u) \|_{0,1/w} 
						+\| f^1_\bB (\bigcdot,u)-f^2_\bB (\bigcdot,u) \|_{1,w} \bigg) du        
				\end{align}%
				\endgroup
				for all $t\in[0,T]$. A similar estimate holds for $\| V^1_\bB (\bigcdot,t)-V^2_\bB (\bigcdot,t) \|_{0,1/w}$. Summing the two inequalities and applying an estimate analogous to~\eqref{estimate_VAVB}, we obtain
				\begingroup\begin{align}\label{2709}
					\sum_{\bI=\bA,\bB}\| V^1_\bI (\bigcdot,t)-V^2_\bI (\bigcdot,t) \|_{0,1/w}
					\leq e^{\left(\lambda_\bA M_\bA +\lambda_\bB M_\bB \right)(T-t)}  (\lambda_\bA  \vee \lambda_\bB )
					\left( \frac{1-e^{-\rho T}}{\rho} \right) 
					\sum_{\bI=\bA,\bB}
					\| f^1_\bI -f^2_\bI \|_{\mathcal{A}}
				\end{align}\endgroup
				for all $t\in[0,T]$. Next, applying Grönwall's inequality to \eqref{estimate_fdiff1} yields $\sum_{\bI=\bA,\bB} \| f^1_\bI (\bigcdot,t) - f^2_\bI (\bigcdot,t) \|_{1,w}
				\leq \left( \dfrac{ M_1}{M_0}
				\int^t_0 \sum_{\bI=\bA,\bB}\| V^1_\bI (\bigcdot,u)-V^2_\bI (\bigcdot,u) \|_{0,1/w}du\right) 
				e^{M_2 t}.$ Substituting \eqref{2709} into the above, we find
				\begingroup\begin{align*}
					\sum_{\bI=\bA,\bB} \| f^1_\bI - f^2_\bI \|_{\mathcal{A}}
					\leq   
					e^{TM_2}\dfrac{ M_1}{M_0}(\lambda_\bA  \vee \lambda_\bB )
					\left( \frac{1-e^{-\rho T}}{\rho} \right) 
					\left( \frac{e^{\left(\lambda_\bA M_\bA +\lambda_\bB M_\bB \right)T}-1 }{\lambda_\bA M_\bA +\lambda_\bB M_\bB} \right) 
					\sum_{\bI=\bA,\bB} \| f^1_\bI - f^2_\bI\|_{\mathcal{A}}.
				\end{align*}\endgroup
				Therefore, the solution is unique if \ref{condition.uniqueness.loc} holds.

				\noindent{\bf Part 2. Under Condition \ref{condition.uniqueness.global}:} Let $\eta>0$, we define the weighted (in time) norms: $\| V \|_{\mathcal{H},\eta}
				:=\sup\limits_{t\in [0,T]} e^{-\eta t} \| V(\bigcdot,t) \|_{0,1/w}$ and $\| f \|_{\mathcal{A},\eta}
				:=\sup\limits_{t\in [0,T]} e^{-\eta t} \| f(\bigcdot,t) \|_{1,w}.$ Multiplying both sides of~\eqref{2731} by $e^{-\eta t}$,
				\begingroup\begin{align*} 
					&\hspace{-13pt}e^{-\eta t}\| V^1_\bA (\bigcdot,t)-V^2_\bA (\bigcdot,t) \|_{0,1/w}\nonumber\\
					\le\,                                                          & \lambda_\bB \int_t^{T} e^{-(\rho-\eta) (u-t)}   
					e^{-\eta u}\bigg( M_\bB\sum_{\bI=\bA,\bB}\| V^2_\bI (\bigcdot,u)-V^1_\bI (\bigcdot,u) \|_{0,1/w} 
					+\| f^1_\bB (\bigcdot,u)-f^2_\bB (\bigcdot,u) \|_{1,w} \bigg) du.                                         
				\end{align*}\endgroup
				An analogous bound holds for $e^{-\eta t} \| V^1_\bB - V^2_\bB \|_{0,1/w}$. Summing both gives 
				\begingroup\begin{align*} 
						\sum_{\bI=\bA,\bB}\!\!
						e^{-\eta t}\| V^1_\bI (\bigcdot,t)-V^2_\bI (\bigcdot,t) \|_{0,1/w}
						\le\!\!                                                           \int_t^{T}  \!e^{-(\rho-\eta) (u-t)-\eta u}  \bigg[\!\! \left(\lambda_\bA M_\bA +\lambda_\bB M_\bB \right)
						\!\!\sum_{\bI=\bA,\bB}\!\!\| V^2_\bI (\bigcdot,u)-V^1_\bI (\bigcdot,u) \|_{0,1/w} \!                      
				\end{align*}%
				\endgroup
				\begingroup\begin{align*} 
						\hspace{250pt}+ (\lambda_\bA  \vee \lambda_\bB )
						\sum_{\bI=\bA,\bB}\| f^1_\bI (\bigcdot,u)-f^2_\bI (\bigcdot,u) \|_{1,w} \bigg] du. \!\!                                         
				\end{align*}%
				\endgroup
				If $\rho>\eta$, we rearrange the terms to get
				\begingroup\begin{align} \label{2979}
						\left[1- 
						\left(\lambda_\bA M_\bA +\lambda_\bB M_\bB \right)
						\dfrac{1-e^{-(\rho-\eta) T}}{\rho-\eta}\right] 
						\sum_{\bI=\bA,\bB}
						\| V^1_\bI-V^2_\bI \|_{\mathcal{H},\eta} 
						\le\,   
						(\lambda_\bA  \vee \lambda_\bB )	\dfrac{1-e^{-(\rho-\eta) T}}{\rho-\eta}        
						\sum_{\bI=\bA,\bB}\| f^1_\bI-f^2_\bI \|_{\mathcal{A},\eta}.                            
				\end{align}%
				\endgroup
				From~\eqref{estimate_fdiff1A1}, multiplying both sides by $e^{-\eta t}$ and using similar reasoning, we obtain
				\begingroup	\begin{align*}
					\nonumber&\hspace{-10pt}e^{-\eta t} \| f^2_\bA (\bigcdot,t)-f^1_\bA (\bigcdot,t) \|_{1,w}\\ 
					\nonumber\leq \,&  \lambda_\bB\int_0^t 
					e^{-\eta (t-u)}   	e^{-\eta u} 
					\bigg\{\|f^1_\bA (\bigcdot,u)-f^2_\bA (\bigcdot,u)\|_{1,w}
					+  \frac{2^{1+\nu}C_\bA }{\nu}\|f^1_\bB (\bigcdot,u)-f^2_\bB (\bigcdot,u)\|_{1,w}\\
					& + C_\bB\left[ \frac{2^{3+\nu}C_\bA }{\nu}
					\right] 
					\bigg( \frac{1}{k_\bB}
					\|V^1_\bB (\bigcdot,u)-V^2_\bB (\bigcdot,u) \|_{0,1/w} 
					+ \frac{1+k_\bA}{k_\bA} 
					\| V^2_\bA (\bigcdot,u)-V^1_\bA (\bigcdot,u) \|_{0,1/w}\bigg)\bigg\}du. 
				\end{align*}\endgroup
				Similar estimate applies to $f^1_\bB - f^2_\bB$. Summing both inequalities, we obtain
				\begingroup\begin{align}
					\nonumber&\hspace{-10pt}
					\sum_{\bI=\bA,\bB}
					e^{-\eta t}\| f^1_\bI (\bigcdot,t)-f^2_\bI (\bigcdot,t) \|_{1,w}\\
					\nonumber
					\leq \,&  \int_0^t 
					e^{-\eta (t-u)}   	e^{-\eta u} 
					\bigg\{\sum_{\bI=\bA,\bB}  M_2\| f^1_\bI (\bigcdot,u)-f^2_\bI (\bigcdot,u) \|_{1,w}
					+\sum_{\bI=\bA,\bB}
					\dfrac{M_1}{M_0}\| V^1_\bI (\bigcdot,u)-V^2_\bI (\bigcdot,u) \|_{0,1/w}\bigg\}du. 
				\end{align}\endgroup
				We rearrange the terms to give $\left( 1- M_2\frac{1-e^{-\eta T}}{\eta} \right) \sum_{\bI=\bA,\bB}\| f^1_\bI-f^2_\bI \|_{\mathcal{A},\eta}
				\leq 	       
				\frac{1-e^{-\eta T}}{\eta} \bigg\{  	 
				\frac{M_1}{M_0}	\sum_{\bI=\bA,\bB}\| V^1_\bI-V^2_\bI \|_{\mathcal{H},\eta}   \bigg\}.$ Assuming $\rho - \eta > \lambda_\bA M_\bA + \lambda_\bB M_\bB$, we can insert \eqref{2979} into the above inequality, yielding:
				\begingroup\begin{align*} 
					&\left( 1- M_2\dfrac{1-e^{-\eta T}}{\eta} \right) \sum_{\bI=\bA,\bB}\| f^1_\bI-f^2_\bI \|_{\mathcal{A},\eta}\nonumber\\
					&\leq 	       
					\dfrac{1-e^{-\eta T}}{\eta} 	 
					\dfrac{M_1}{M_0}  \left[1- 
					\left(\lambda_\bA M_\bA +\lambda_\bB M_\bB \right)
					\dfrac{1-e^{-(\rho-\eta) T}}{\rho-\eta}\right]^{-1}                                                 
					(\lambda_\bA  \vee \lambda_\bB )	\dfrac{1-e^{-(\rho-\eta) T}}{\rho-\eta}        
					\sum_{\bI=\bA,\bB}\| f^1_\bI-f^2_\bI \|_{\mathcal{A},\eta} 
					. 
				\end{align*}\endgroup
				To ensure uniqueness independent of the time horizon $T$, it suffices to choose $\eta > 0$ such that $\rho-\eta>  
				\left(\lambda_\bA M_\bA +\lambda_\bB M_\bB \right)$, 
				$ \eta > M_2$ and $\left( 1- M_2\frac{1}{\eta} \right)^{-1}
				\frac{1}{\eta} 	 
				\frac{M_1}{M_0}  \left[1- 
				\left(\lambda_\bA M_\bA +\lambda_\bB M_\bB \right)
				\frac{1}{\rho-\eta}\right]^{-1}                              
				(\lambda_\bA  \vee \lambda_\bB )	\frac{1}{\rho-\eta}<1.$ It is sufficient to assume that  there is $\eta>0$ such that $M_2<\eta<\rho -\left(\lambda_\bA M_\bA +\lambda_\bB M_\bB \right) $ and $\frac{1}{\eta-M_2} 	 
				\frac{M_1}{M_0}  
				\frac{1}{\rho-\eta- 
					\left(\lambda_\bA M_\bA +\lambda_\bB M_\bB \right)}                          
				(\lambda_\bA  \vee \lambda_\bB ) <1.$ The left-hand side, viewed as a function of $\eta$, attains its minimum at $\eta = [M_2 + \rho -\left(\lambda_\bA M_\bA +\lambda_\bB M_\bB \right) ]/2$. Thus, we assume that $\frac{2}{ -M_2 + \rho -\left(\lambda_\bA M_\bA +\lambda_\bB M_\bB \right)} 	 
				\frac{M_1}{M_0}  
				\frac{2(\lambda_\bA  \vee \lambda_\bB ) }{\rho- M_2  -\left(\lambda_\bA M_\bA +\lambda_\bB M_\bB \right) }                          
				<1$ which leads to the desired sufficient condition. 
			\end{proof}

			\section{Proof of \Cref{thm.verification}: Verification Theorem}\label{sec. thm.verification}
			
			\begin{proof}[Proof of \Cref{thm.verification}]
				
				By \eqref{optimalcontrol.A} and \eqref{HJB_maximize.VA}, the equation 
				\eqref{HJBx} leads to the inequality
				\begingroup\begin{align} \label{3211}
					\rho V_\bA^* (x,t)-\frac{\partial V_\bA^* (x,t) }{\partial t} 
					&\geq\mathbf{1}_{\{v (x,t) \leq (V_\bB^*) ^{-1}(x,t)\}}\lambda_\bB  \int_{v (x,t)}^{ (V_\bB^*) ^{-1}(x,t)}\big[y-V_\bA^* (x,t)\big] f_\bB^* (y,t) d y+r_\bA (x,t), 
				\end{align}\endgroup
				for any $v\in \mathcal{U}_\bA$. Hence, integrating the following inequality with respect to $t$ gives 
				\begingroup\begin{align*} 
					& \left(\rho+ \mathbf{1}_{\{v (x,t) \leq (V_\bB^*) ^{-1}(x,t)\}}\lambda_\bB  \int_{v (x,t)}^{ (V_\bB^*) ^{-1}(x,t)} f_\bB^* (y,t) d y\right) V_\bA^* (x,t)-\frac{\partial V_\bA^* (x,t) }{\partial t} \\ 
					&\geq\mathbf{1}_{\{v (x,t) \leq (V_\bB^*) ^{-1}(x,t)\}} \lambda_\bB  \int_{v (x,t)}^{ (V_\bB^*) ^{-1}(x,t)}y f_\bB^* (y,t) d y+r_\bA (x,t),
				\end{align*}\endgroup
				which deduces
				\begingroup\begin{align} 
					e^{-\rho t}V_\bA^* (x,t) \geq & \int_t^T \exp\left(-\rho s-\lambda_\bB \int_t^s  \int_{v (x, r )}^{(V_\bB^*)^{-1}(x, r )} \mathbf{1}_{\{v (x, r ) \leq (V_\bB^*) ^{-1}(x, r )\}}f_\bB^* (y, r ) d y d r \right)  \nonumber\\
					& \cdot\left(\mathbf{1}_{\{v (x,s) \leq (V_\bB^*) ^{-1}(x,s)\}}
					\lambda_\bB  \int_{v (x,s)}^{(V_\bB^*)^{-1}(x,s)}y f_\bB^* (y,s) d y+r_\bA (x,s)\right) ds\nonumber\\
					& + \exp\left(-\rho T-\lambda_\bB \int_t^T  \int_{v (x, r )}^{(V_\bB^*)^{-1}(x, r )} \mathbf{1}_{\{v (x, r ) \leq (V_\bB^*) ^{-1}(x, r )\}}f_\bB^* (y, r ) d y d r \right)h_\bA (x).
					\label{3231}
				\end{align}\endgroup
				
				For the thresholds $v \in \mathcal{U}_\bA$ and $V_\bB^* \in \mathcal{U}_\bB$, the matching region of type-$\bA$ agent $(x,s)$ is given by $\mathcal{O}_\bA(x,s):=\{y:v  (x,s) \leq y \leq (V_\bB^*) ^{-1}(x,s)\}$. By \eqref{645}, we have $\mathbb{P}_{\bA,x,s}\pig(\mathbf{QL}_{\bB,s} \in \mathcal{O}_\bA(x,s) \text{ and } \mathbf{ST}_{\bB,s} = 0{\color{black} \,\pig|\, \text{$\exists \, \tau^\bJ \in \mathcal{T}^\bJ$ s.t. $\tau^\bJ = s$}}\pig)=\mathbf{1}_{\{v  (x,s) \leq (V_\bB^*) ^{-1}(x,s)\}}
				\int_{v  (x,s)}^{(V_\bB^*)^{-1}(x,s)} f_\bB^* (y,s) d y$. Similar to \eqref{eq.Ptau}, we have, for $s\in[0,T)$, 
				\begingroup	\begin{equation}\label{eq.Ptau.ts}
					-\frac{\p \mathbb{P}_{\bA,x,0}(\tau_\bA(x,0) >s)}{\p s} = \mathbb{P}_{\bA,x,0}(\tau_\bA(x,0) >s)
					\mathbf{1}_{\{v  (x,s) \leq (V_\bB^*) ^{-1}(x,s)\}}\lambda_\bB
					\int_{v  (x,s)}^{(V_\bB^*)^{-1}(x,s)} f_\bB^* (y,s) d y, 
				\end{equation}\endgroup
				leading to $$\mathbb{P}_{\bA,x,0}(\tau_\bA(x,0) >s)
				=  \mathbb{P}_{\bA,x,0}(\tau_\bA(x,0) >t)
				\exp\left[-\lambda_\bB  
				\int_t^s\int_{v  (x,r)}^{(V_\bB^*)^{-1}(x,r)} 
				\mathbf{1}_{\{v  (x,r) \leq (V_\bB^*) ^{-1}(x,r)\}}
				f_\bB^* (y,r) d ydr\right].$$ Hence, by the memoryless property of $\tau_\bA$ as shown in \eqref{property.memoryless}, we have, for any $0 \leq t \leq s \leq T$, $
				\mathbb{E}_{\bA,x,t}\mathbf{1}_{\{\tau_\bA(x,t) >s\}}= \mathbb{P}_{\bA,x,t}(\tau_\bA(x,t) >s) = \mathbb{P}(\tau_\bA(x,0) >s \,|\, \tau_\bA(x,0) >t) = \frac{\mathbb{P}_{\bA,x,0}(\tau_\bA(x,0) >s )}{\mathbb{P}_{\bA,x,0}( \tau_\bA(x,0) >t)},$  or
				\begingroup	\begin{align}\label{Pts}
					\mathbb{E}_{\bA,x,t}\mathbf{1}_{\{\tau_\bA(x,t) >s\}} = \exp\left[-\lambda_\bB  
					\int_t^s\int_{v  (x,r)}^{(V_\bB^*)^{-1}(x,r)} 
					\mathbf{1}_{\{v  (x,r) \leq (V_\bB^*) ^{-1}(x,r)\}}
					f_\bB^* (y,r) d ydr\right].
				\end{align}\endgroup
				Putting \eqref{Pts} into \eqref{3231}, we deduce 
				\begingroup	\begin{align} 
					e^{-\rho t}V_\bA^* (x,t) \geq&\, \mathbb{E}_{\bA,x,t}\Bigg[\int_t^T \mathbf{1}_{\{\tau_\bA(x,t) >s\}}e^{-\rho s} 
					r_\bA (x,s) ds  + \mathbf{1}_{\{\tau_\bA(x,t) >T\}}
					e^{-\rho T} h_\bA (x)\Bigg]\nonumber\\
					&+\int_t^T \exp\left(-\rho s-\lambda_\bB \int_t^s  \int_{v (x, r )}^{(V_\bB^*)^{-1}(x, r )} \mathbf{1}_{\{v (x, r ) \leq (V_\bB^*) ^{-1}(x, r )\}}f_\bB^* (y, r ) d y d r \right)\nonumber\\
					& \hspace{20pt}\cdot\left(\mathbf{1}_{\{v (x,s) \leq (V_\bB^*) ^{-1}(x,s)\}}
					\lambda_\bB  \int_{v (x,s)}^{(V_\bB^*)^{-1}(x,s)}y f_\bB^* (y,s) d y\right) ds.
					\label{3271}
				\end{align}\endgroup
				
				On the one hand, we have
				\begingroup	\fontsize{9pt}{11pt}\begin{align*}
					&\int_t^\infty \mathbf{1}_{\{s \le T\}} e^{-\rho s}  \frac{\int_{v (x,s)}^{(V_\bB^*)^{-1}(x,s)}y f_\bB^* (y,s) d y}{\int_{v (x,s)}^{(V_\bB^*)^{-1}(x,s)}f_\bB^* (y,s) d y} \mathbb{P}_{\bA,x,t}(\tau_\bA(x,t) \in ds) \\
					&=\int_t^\infty   \frac{\mathbb{E}\Big[e^{-\rho s}\mathbf{QL}_{\bB, s} \mathbf{1}_{\{s \le T\}} \mathbf{1}_{\{ v  (x,s) \leq \mathbf{QL}_{\bB,s} \leq  (V_\bB^*) ^{-1}(x,s),\, \mathbf{ST}_{\bB,s}=0 \}}{\color{black} \,\pig|\, \text{$\exists \, \tau^\bJ \in \mathcal{T}^\bJ$ s.t. $\tau^\bJ = s$}}\Big]}{\mathbb{P}\Big(v  (x,s) \leq \mathbf{QL}_{\bB,s} \leq  (V_\bB^*) ^{-1}(x,s),\, \mathbf{ST}_{\bB,s}=0{\color{black} \,\pig|\, \text{$\exists \, \tau^\bJ \in \mathcal{T}^\bJ$ s.t. $\tau^\bJ = s$}}\Big)} \mathbb{P}_{\bA,x,t}(\tau_\bA(x,t) \in ds) \\
					&=\int_t^\infty \mathbb{E}\bigg[e^{-\rho s}\mathbf{QL}_{\bB, s}\cdot \mathbf{1}_{\{s \le T\}} \,\bigg|{\color{black}  \text{$\exists \, \tau^\bJ \in \mathcal{T}^\bJ$ s.t. $\tau^\bJ = s$,}}\, v  (x,s) \leq \mathbf{QL}_{\bB,s} \leq  (V_\bB^*) ^{-1}(x,s),\, \mathbf{ST}_{\bB,s}=0 \bigg] \mathbb{P}_{\bA,x,t}(\tau_\bA(x,t) \in ds)\hspace{-10pt}
					\hspace{-10pt}\end{align*}\normalsize\endgroup
				Suppose the denominator in the first line is not zero. Otherwise the analysis below can be completed by using simpler arguments. Since $\{\tau_\bA(x,t) = s\}$ is conditionally independent of $\mathbf{QL}_{\bB, s}$ given $\{{\color{black}  \text{$\exists \, \tau^\bJ \in \mathcal{T}^\bJ$ s.t. $\tau^\bJ = s$,}}\, v  (x,s) \leq \mathbf{QL}_{\bB,s} \leq  (V_\bB^*) ^{-1}(x,s),\, \mathbf{ST}_{\bB,s}=0\}$, we have
				\begingroup	\fontsize{10pt}{11pt}\begin{align}\label{3398}
					\nonumber &\int_t^\infty \mathbf{1}_{\{s \le T\}} e^{-\rho s}  \frac{\int_{v (x,s)}^{(V_\bB^*)^{-1}(x,s)}y f_\bB^* (y,s) d y}{\int_{v (x,s)}^{(V_\bB^*)^{-1}(x,s)}f_\bB^* (y,s) d y} \mathbb{P}_{\bA,x,t}(\tau_\bA(x,t) \in ds) \\
					\nonumber &=\int_t^\infty \mathbb{E}\bigg[e^{-\rho s}\mathbf{QL}_{\bB, s}\cdot \mathbf{1}_{\{s \le T\}} \,\bigg| \tau_\bA(x,t) = s, v  (x,s) \leq \mathbf{QL}_{\bB,s} \leq  (V_\bB^*) ^{-1}(x,s),\, \mathbf{ST}_{\bB,s}=0 \bigg] \mathbb{P}_{\bA,x,t}(\tau_\bA(x,t) \in ds) \\
					&=\int_t^\infty  \mathbb{E}\bigg[e^{-\rho s}\mathbf{QL}_{\bB, s} \mathbf{1}_{\{s \le T\}} \mathbf{1}_{\{v  (x,s) \leq \mathbf{QL}_{\bB,s} \leq  (V_\bB^*) ^{-1}(x,s), \mathbf{ST}_{\bB,s}=0\} } \,\bigg| \tau_\bA(x,t) = s \bigg] \mathbb{P}_{\bA,x,t}(\tau_\bA(x,t) \in ds).
				\end{align}\normalsize\endgroup
				The last equality is due to $\mathbb{P}\left(v  (x,s) \leq \mathbf{QL}_{\bB,s} \leq  (V_\bB^*) ^{-1}(x,s), \mathbf{ST}_{\bB,s}=0 \big| \tau_\bA(x,t) = s \right) = 1.$
				
				On the other hand, by \eqref{property.memoryless} and \eqref{eq.Ptau.ts}, we have
				\begingroup\begin{align*}
						\frac{\p \mathbb{P}_{\bA,x,t}(\tau_\bA(x,t) >  s)}{\p s} 
						&= \frac{\p }{\p s}\frac{\mathbb{P}_{\bA,x,0}(\tau_\bA(x,0) >  s)}{\mathbb{P}_{\bA,x,0}(\tau_\bA(x,0) >  t)}\\ 
						&= -\frac{\mathbb{P}_{\bA,x,0}(\tau_\bA(x,0) >s)}{\mathbb{P}_{\bA,x,0}(\tau_\bA(x,0) >t)}\cdot
						\mathbf{1}_{\{v  (x,s) \leq (V_\bB^*) ^{-1}(x,s)\}}\cdot\lambda_\bB
						\int_{v  (x,s)}^{(V_\bB^*)^{-1}(x,s)} \!\!f_\bB^* (y,s) d y \! \\
						&= -\mathbb{P}_{\bA,x,t}(\tau_\bA(x,t) >  s)\cdot \mathbf{1}_{\{v  (x,s) \leq (V_\bB^*) ^{-1}(x,s)\}}\cdot\lambda_\bB
						\int_{v  (x,s)}^{(V_\bB^*)^{-1}(x,s)} \!\!f_\bB^* (y,s) d y. \!
				\end{align*}%
				\endgroup
				Therefore, putting the above quantity into \eqref{3398} and then using \eqref{Pts}, we have
				\begingroup\begin{align}\label{3284}
						\nonumber&\mathbb{E}_{\bA,x,t}\left[e^{-\rho \tau_\bA(x,t)}\mathbf{QL}_{\bB, \tau_\bA(x,t)}
						\cdot \mathbf{1}_{\{\tau_\bA(x,t) \le T\}} \cdot \mathbf{1}_{\{\mathbf{ST}_{\bB, \tau_\bA(x,t)} =0,\, v  (x,\tau_\bA(x,t)) \leq \mathbf{QL}_{\bB,\tau_\bA(x,t)} \leq  (V_\bB^*) ^{-1}(x,\tau_\bA(x,t))\}} \right]\\
						\nonumber&=\int_t^T e^{-\rho s} \mathbb{P}_{\bA,x,t}(\tau_\bA(x,t) >  s)\cdot\mathbf{1}_{\{v  (x,s) \leq (V_\bB^*) ^{-1}(x,s)\}}\cdot\lambda_\bB
						\int_{v  (x,s)}^{(V_\bB^*)^{-1}(x,s)} y f_\bB^* (y,s) d y ds\\
						\nonumber&= \int_t^T \exp\left(-\rho s-\lambda_\bB \int_t^s  \int_{v (x, r )}^{(V_\bB^*)^{-1}(x, r )} \mathbf{1}_{\{v (x, r ) \leq (V_\bB^*) ^{-1}(x, r )\}}f_\bB^* (y, r ) d y d r \right)\\
						& \hspace{20pt}\cdot\left(\mathbf{1}_{\{v (x,s) \leq (V_\bB^*) ^{-1}(x,s)\}}
						\lambda_\bB  \int_{v (x,s)}^{(V_\bB^*)^{-1}(x,s)}y f_\bB^* (y,s) d y\right) ds.
				\end{align} %
				\endgroup          
				Bringing \eqref{3284} into \eqref{3271}, we have, for any $v\in \mathcal{U}_\bA$,
				\begingroup	\begin{align}
					\nonumber V_\bA^* (x,t) \geq\,  & \mathbb{E}_{\bA,x,t}\bigg[\int_t^{T \wedge \tau_\bA } e^{-\rho (s-t)} r_\bA(x,s) ds 
					+  e^{-\rho (\tau_\bA(x,t)-t)}\mathbf{QL}_{\bB, \tau_\bA(x,t)}
					\cdot \mathbf{1}_{\{\tau_\bA(x,t) \le T\}} \\
					\nonumber &\hspace{80pt}\cdot \mathbf{1}_{\{\mathbf{ST}_{\bB, \tau_\bA(x,t)} =0, \, v  (x,\tau_\bA(x,t)) \leq \mathbf{QL}_{\bB,\tau_\bA(x,t)} \leq  (V_\bB^*) ^{-1}(x,\tau_\bA(x,t))\}}\\
					\nonumber &\hspace{80pt}+ e^{-\rho (T-t)}\mathbf{1}_{\{\tau_\bA > T\}}h_\bA (x)\bigg] \\
					= &\, J_\bA (v;x,t,V_\bB^*,\pi_{\bB,\bigcdot}^*).	\label{inequaility.verification}                                  \end{align}\endgroup

				On the other hand, as elaborated in Section \ref{sec.HJB} that deducing  \eqref{eq HJB VA} from \eqref{HJB_maximize.VA}, if we choose $v  (x,s)  = V_\bA^* (x,s)$, then the inequality in \eqref{3211} becomes equality. We can repeat the above analysis and arrive at an equality sign for \eqref{inequaility.verification}. By \Cref{BilipschitzVA}, $V_\bA^*(\bigcdot,s)$ is a strictly increasing function for all $s\in[t,T]$, which means $V_\bA^*\in \mathcal{U}_\bA$. By definition \eqref{def.V_I.threshold}, $V_\bA^*(x,t) = \mathcal{V}_\bA (x,t,V_\bB^*,\pi_{\bB,\bigcdot}^*) = J_\bA (V_\bA^*;x,t,V_\bB^*,\pi_{\bB,\bigcdot}^*).$ By the uniqueness of the value function and the uniqueness of the FP equation \eqref{FPx}, we see that the joint distribution is uniquely given by \eqref{3343} when the objective functional is maximized, due to the discussion in \Cref{sec.markov}. Similar arguments apply to type-$\bB$ agents. The proof is completed. 
			\end{proof}
			
		\end{appendix}

	\end{document}